\documentclass[a4paper, 10pt,reqno]{amsart}

\usepackage[usenames,dvipsnames]{xcolor}

\usepackage[top=3cm,bottom=2cm,left=2.9cm,right=2.9cm]{geometry}
\usepackage{mathrsfs}

\usepackage{amssymb,amsthm,amsfonts,amsbsy,amstext,latexsym}
\usepackage[reqno]{amsmath}

\usepackage{color}
\usepackage{graphicx}
\usepackage{bbm}
\usepackage{comment}
\usepackage{graphicx, listings,lscape, epstopdf, units, textcomp, fancyhdr, url, soul, color}
\usepackage[textsize=small]{todonotes}
\usepackage{enumitem}

\usepackage{mathtools}
\mathtoolsset{showonlyrefs}

\definecolor{MyDarkblue}{rgb}{0,0.08,0.50}
\definecolor{Brickred}{rgb}{0.65,0.08,0}

\usepackage{hyperref,upref}
\hypersetup{
	colorlinks=true,       % false: boxed links; true: colored links
	linkcolor=MyDarkblue,          % color of internal links
	citecolor=Brickred,        % color of links to bibliography
	filecolor=red,      % color of file links
	urlcolor=cyan           % color of external links
}

\newtheorem*{theorem*}{Theorem}
\newtheorem{theorem}{Theorem}[section]
\newtheorem{lemma}[theorem]{Lemma}
\newtheorem{proposition}[theorem]{Proposition}
\newtheorem{corollary}[theorem]{Corollary}

\theoremstyle{definition}
\newtheorem{definition}[theorem]{Definition}
\newtheorem{assumption}[theorem]{Assumption}
\newtheorem{remark}[theorem]{Remark}
\newtheorem{example}[theorem]{Example}

\renewcommand{\P}{\mathbb{P}}
\newcommand{\Pv}{\mathbb{P}}

\newcommand{\eps}{\varepsilon}

\newcommand{\s}[1]{\sum_{#1}}

\newcommand{\cO}{\mathcal{O}}
\newcommand{\cP}{\mathcal{P}}
\newcommand{\cS}{\mathcal{S}}

\newcommand{\e}{{\mathrm e}}

\newcommand{\R}{\mathbb{R}}
\newcommand{\N}{\mathbb{N}}
\newcommand{\Z}{\mathbb{Z}}

\newcommand*{\wt}{\widetilde}

\newcommand*{\be}{\begin{equation}}
	\newcommand*{\ee}{\end{equation}}
\newcommand*{\ba}{\begin{aligned}}
	\newcommand*{\ea}{\end{aligned}}
\newcommand*{\barr}{\begin{array}{c}}
	\newcommand*{\earr}{\end{array}}
\def \toinp    {\buildrel {\Pv}\over{\longrightarrow}}
\def \toindis  {\buildrel {d}\over{\longrightarrow}}
\def \toas     {\buildrel {a.s.}\over{\longrightarrow}}

\newcommand*{\ind}{\mathbbm{1}}
\makeatletter
\def\namedlabel#1#2{\begingroup
	#2%
	\def\@currentlabel{#2}%
	\phantomsection\label{#1}\endgroup
}

\makeatother
\newcommand{\bes}{\begin{equation*}}
	\newcommand{\ees}{\end{equation*}}

\renewcommand{\P}[1]{\mathbb{P}\!\left(#1\right)}
\newcommand{\E}[1]{\mathbb{E}\left[#1\right]}

\newcommand{\F}{W}

\newcommand{\Zm}{\mathcal{Z}}
\renewcommand{\N}{\mathbb{N}}

\newcommand{\Ef}[2]{\mathbb{E}_\F#1[#2#1]}
\newcommand{\Pf}[1]{\mathbb{P}_\F\!\left(#1\right)}

\renewcommand{\d}{\mathrm{d}}

\newcommand{\M}{\mathcal{M}}
\newcommand{\zni}{\Zm_n(i)}

\newcommand{\inn}{i\in[n]}

\numberwithin{equation}{section}

\renewcommand{\e}{\mathrm{e}}
\newcommand{\jnn}{j\in[n]}
\newcommand{\znj}{\Zm_n(j)}

\makeatletter
\newcommand{\leqnomode}{\tagsleft@true\let\veqno\@@leqno}
\newcommand{\reqnomode}{\tagsleft@false\let\veqno\@@eqno}
\makeatother
\newlength{\tagmarginsep} % Distance required
\makeatletter
\def\author@andify{%
	\nxandlist {\unskip ,\penalty-1 \space\ignorespaces}%
	{\unskip {} \@@and~}%
	{\unskip \penalty-2 \space \@@and~}%
}
\makeatother

\begin{document}
	
	\title[Maximum degree in WRT with bounded random weights]{Fine asymptotics for the maximum degree in weighted recursive trees with bounded random weights}

	\date{\today}
	
	\author[Eslava]{Laura Eslava$^\dagger$}
	\address{$^\dagger$Universidad Nacional Autonoma Mexico, Instituto de investigaciones en matematicas y en sistemas, CDMX, 04510, Mexico}
	\email{laura@sigma.iimas.unam.mx}
	\author[Lodewijks]{Bas Lodewijks$^{\dagger\dagger}$}
	\address{$^{\dagger\dagger}$ Universit\'e Jean Monnet, Saint-Etienne, France and  Institut Camille Jordan, Lyon and Saint-Etienne, France}
	\email{bas.lodewijks@univ-st-etienne.fr}
	\author[Ortgiese]{Marcel Ortgiese$^\star$}
	
	\address{$^\star$Department of Mathematical Sciences,
		University of Bath,
		Claverton Down,
		Bath,
		BA2 7AY,
		United Kingdom.}

	\email{m.ortgiese@bath.ac.uk}
	
	\begin{abstract}
		A weighted recursive tree is an evolving tree in which vertices are assigned random vertex-weights and new vertices connect to a predecessor with a probability proportional to its weight. Here, we study the maximum degree and near-maximum degrees in weighted recursive trees when the vertex-weights are almost surely bounded and their distribution function satisfies a mild regularity condition near zero. 
		We are able to specify higher-order corrections to the first order growth of the maximum degree established in prior work. The accuracy of the results depends on the behaviour of the weight distribution near the largest possible value and in certain cases we manage to find the corrections up to random order. Additionally, we describe the tail distribution of the maximum degree, the distribution of the number of vertices attaining the maximum degree, and establish asymptotic normality of the number of vertices with near-maximum degree. Our analysis extends the results proved for random recursive trees (where the weights are constant) to the case of random weights. The main technical result shows that the degrees of several uniformly chosen vertices are asymptotically independent with explicit error corrections.
	\end{abstract}
	
	\maketitle 
	
	\section{Introduction}	
	
	The Weighted Recursive Tree model (WRT), first introduced by Borovkov and Vatutin~\cite{BorVat06}, is a recursive tree process $(T_n,n\in\N)$ and a generalisation of the random recursive tree model. Here we consider a variation, first studied by Hiesmayr and I\c slak~\cite{HieIsl17}, where the first vertex does not necessarily have weight one. Let $(W_i)_{i\in\N}$ be a sequence of positive vertex-weights. Initialise the process with the tree $T_1$, which consists of the vertex $1$ (which denotes the root) and assign vertex-weight $W_1$ to it. Recursively, at every step $n\geq 2$, we obtain $T_n$ by adding to $T_{n-1}$ the vertex $n$, assigning vertex-weight $W_n$ to it and connecting $n$ to a vertex $i\in[n-1]$, which, conditionally on the vertex-weights $W_1,\ldots, W_{n-1}$, is selected with a probability proportional to $W_i$. In this paper, we consider edges to be directed towards the vertex with the smaller label. We note that allowing every vertex to connect to $m\in\N$ many predecessors, each one selected independently, yields the more general Weighted Recursive Graph model (WRG) introduced in~\cite{LodOrt21}. The focus of this paper is the WRT model in the case when the vertex-weights are \emph{almost surely bounded random variables}.
	
	Lodewijks and Ortgiese~\cite{LodOrt21} established that, in the case of positive, bounded random vertex-weights, the maximum degree $\Delta_n$ of the WRG model grows logarithmically and that $\Delta_n/\log n\toas 1/\log \theta_m$, where $\theta_m:=1+\E W/m$ with $\E W$ the mean of the vertex-weight distribution and $m\in\N$ the out-degree of each vertex. Note that setting $m=1$ yields the result for the WRT model. In this paper, we improve this result by describing the higher-order asymptotic behaviour of the maximum degree when the vertex-weights are almost surely bounded. In this case we are able to distinguish several classes of vertex-weight distributions for which different higher-order behaviour can be observed.
	
	Beyond the initial work of Borovkov and Vatutin and also Hiesmayr and I\c slak studying the height, depth and size of branches of the WRT model, other properties such as the degree distribution, large and maximum degrees, and weighted profile and height of the tree have been studied. Mailler and Uribe Bravo~\cite{MaiBra19}, as well as S\'enizergues~\cite{Sen19} and S\'enizergues and Pain~\cite{SenPain21} study the weighted profile and height of the WRT model. Mailler and Uribe Bravo consider random vertex-weights with particular distributions, whereas S\'enizergues and Pain allow for a more general model with both sequences of deterministic as well as random weights.
	
	Iyer~\cite{Iyer20} and the more general work by Fountoulakis and Iyer~\cite{FouIyer21} study the degree distribution of a large class of evolving weighted random trees, and Lodewijks and Ortgiese~\cite{LodOrt21} study the degree distribution of the WRG model. In both cases, the WRT model is a particular example of the models studied and all results prove the existence of an almost sure limiting degree distribution for the empirical degree distribution.
	
	Finally, Lodewijks and Ortgiese~\cite{LodOrt21} and Lodewijks~\cite{Lod21} study the maximum degree and the labels of the maximum degree vertices of the WRG model for a large range of vertex-weight distributions. In particular, a distinction between distributions with unbounded support and bounded support is observed. In the former case the behaviour and size of the label of the maximum degree is mainly controlled by a balance of vertices being old (i.e.\  having a small label) and having a large vertex-weight. In the latter case, due to the fact that the vertex-weights are bounded, the behaviour is instead controlled by a balance of vertices being old and having a degree which significantly exceeds their mean degree.
	
	A particular case of the WRT model is the Random Recursive Tree (RRT) model, which is obtained when each vertex-weight equals one almost surely. As a result, techniques used to study the maximum degree in the RRT model can be adapted to analyse the maximum degree in the WRT model. Lodewijks and Ortgiese~\cite{LodOrt21} demonstrate this by adapting the approach of Devroye and Lu~\cite{DevLu95} for proving the almost sure convergence of the rescaled maximum degree in the Directed Acyclic Graphs model (DAG) (the multigraph case of the RRT model) and using it for the analysis of the maximum degree in the WRG model, as discussed above. Hence, we survey the development of the properties of the maximum degree of the RRT model.
	
	Szyma\'nsky was the first to study the maximum degree of the RRT model and proved its convergence of the mean; $\E{\Delta_n/\log n}\to 1/\log 2$. Later, Devroye and Lu~\cite{DevLu95} extend this to almost sure convergence and extended this to the DAG model as well. Goh and Schmutz~\cite{GohSch02} showed that $\Delta_n-\lfloor \log_2 n\rfloor$ converges in distribution along suitable subsequences and identified possible distributions for the limit. Adarrio-Berry and Eslava~\cite{AddEsl18} provide a precise characterisation of the subsequential limiting distribution of rescaled large degrees in terms of a Poisson point process as well as a central limit theorem result for near-maximum degrees (of order $\log_2 n- i_n$ where $i_n\to\infty,i_n=o(\log n)$). Eslava~\cite{Esl16} extends this to the joint convergence of the degree and depth of high degree vertices.
	
	In this paper we adapt part of the techniques developed by Adarrio-Berry and Eslava in~\cite{AddEsl18}. They consist of two main components: First, they establish an equivalence between the RRT model and a variation of the Kingman $n$-coalescent and use this to provide a detailed asymptotic description of the tail distribution of the degrees of $k$ vertices selected uniformly at random, for any $k\in\N$. This variation of the Kingman $n$-coalescent is a process which starts with $n$ trees, each consisting of only a single root. Then, at every step $1$ through $n-1$, a pair of roots is selected uniformly at random and independently of this selection, each possibility with probability $1/2$, one of the two roots is connected to the other with a directed edge. This reduces the number of trees by one and, after $n-1$ steps, yields a directed tree. It turns out that this directed tree is equal in law to the random recursive tree. In the $n$-coalescent all $n$ roots in the initialisation are equal in law and the degrees of the vertices are exchangeable. This allows Adarrio-Berry and Eslava to obtain the degree tail distribution with a precise error rate. Second, this precise tail distribution is used to obtain joint factorial moments of the quantities 
	\be \ba\label{eq:xnirrt}
	X^{(n)}_i&:=|\{j\in[n]:\Zm_n(j)= \lfloor \log_2 n\rfloor +i\}|,\quad i\in\Z,\\
	X^{(n)}_{\geq i}&:=|\{j\in[n]:\Zm_n(j)\geq \lfloor \log_2 n\rfloor +i\}|,\quad i\in\Z,
	\ea\ee  
	where $\Zm_n(j)$ denotes the in-degree of vertex $j$ in the tree of size $n$. The joint factorial moments of these $X^{(n)}_i,X^{(n)}_i$ are used to identify the limiting distribution of high degrees in the tree. The sub-sequential convergence, as mentioned above, is due to the floor function applied to $\log_2 n$ and the integer-valued in-degrees $\Zm_n(j)$.
	
	For the WRT model, however, it provides no advantage to construct a `weighted' Kingman $n$-coalescent to obtain precise asymptotic expression for the tail distribution of vertex degrees. As pairs of roots in the Kingman $n$-coalescent are selected uniformly at random and hence the roots are equal in law, it is not necessary to keep track of which roots are selected at what step. In a weighted version of the Kingman $n$-coalescent, pairs of roots would have to be selected with probabilities proportional to their weights, so that it is necessary to record which roots are selected at which step. As a result, a weighted Kingman $n$-coalescent is not (more) useful in analysing the tail distribution of vertex degrees. 
	
	Instead, we improve results on the convergence of the empirical degree distribution of the WRT model obtained by Iyer~\cite{Iyer20} and Lodewijks and Ortgiese~\cite{LodOrt21}. We obtain a convergence rate to the limiting degree distribution, the asymptotic empirical degree distribution for degrees $k=k(n)$ which diverge with $n$, as well as asymptotic independence of degrees of vertices selected uniformly at random. We combine this with the joint factorial moments of quantities similar to~\eqref{eq:xnirrt} and use the techniques developed by Adarrio-Berry and Eslava~\cite{AddEsl18} to derive fine asymptotics of the maximum degree in the WRT model. 
	
	\textbf{Notation.} Throughout the paper we use the following notation: we let $\N:=\{1,2,\ldots\}$ be the natural numbers, set $\N_0:=\{0,1,\ldots\}$ to include zero and let $[t]:=\{i\in\N: i\leq t\}$ for any $t\geq 1$. For $x\in\R$, we let $\lceil x\rceil:=\inf\{n\in\Z: n\geq x\}$ and $\lfloor x\rfloor:=\sup\{n\in\Z: n\leq x\}$, and for $x\in\R,k\in\N$, let $(x)_k:=x(x-1)(x-2)\cdots(x-(k-1))$ and $(x)^{(k)}:=x(x+1)(x+2)\cdots(x+(k-1))$, and $(x)_0=(x)^{(0)}:=1$. For $x,y \in \R$, we write $x\wedge y := \min \{ x,y\}$ and $x \vee y := \max\{ x, y\}$. Moreover, for sequences $(a_n,b_n)_{n\in\N}$ such that $b_n$ is positive for all $n$ we say that $a_n=o(b_n), a_n\sim b_n, a_n=\mathcal{O}(b_n)$ if $\lim_{n\to\infty}a_n/b_n=0, \lim_{n\to\infty} a_n/b_n=1$ and if there exists a constant $C>0$ such that $|a_n|\leq Cb_n$ for all $n\in\N$, respectively. For random variables $X,(X_n)_{n\in\N},Y$ we denote $X_n\toindis X, X_n\toinp X$ and $X_n\toas X$ for convergence in distribution, probability and almost sure convergence of $X_n$ to $X$, respectively. Also, we write $X_n=o_\mathbb{P}(1)$ if $X_n \toinp 0$ and $X\preceq Y$ if $Y$ stochastically dominates $X$. Finally, we use the conditional probability measure $\Pf{\cdot}:=\mathbb{P}(\,\cdot\, |(\F_i)_{i\in\N})$ and conditional expectation $\Ef{}{\cdot}:=\E{\,\cdot\,|(\F_i)_{i\in\N}}$, where the $(W_i)_{i\in\N}$ are the i.i.d.\ vertex-weights of the WRT model.

	\section{Definitions and main results}\label{sec:results}
	
	The weighted recursive tree (WRT) model is a growing random tree model that generalises the random recursive tree (RRT), in which vertices are assigned (random) weights and new vertices connect with existing vertices with a probability proportional to the vertex-weights.	
	
	The definition of the WRT model follows the one in~\cite{HieIsl17}:
	
	\begin{definition}[Weighted Recursive Tree]\label{def:WRT}
		Let $(\F_i)_{i\geq 1}$ be a sequence of i.i.d.\ copies of a positive random variable $\F$ such that $\P{\F>0}=1$ and set
		\be 
		S_n:=\sum_{i=1}^n\F_i.
		\ee
		We construct the \emph{weighted recursive tree} as follows:
		\begin{enumerate}
			\item[1)] Initialise the tree with a single vertex $1$, denoted as the root, and assign to the root a vertex-weight $\F_1$. Denote this tree by $T_1$.
			\item[2)] For $n\geq 1$, introduce a new vertex $n+1$ and assign to it the vertex-weight $\F_{n+1}$. Conditionally on $T_n$, connect to some $\inn$ with probability $\F_i/S_n$. Denote the resulting tree by $T_{n+1}$.
		\end{enumerate}
		We treat $T_n$ as a directed tree, where edges are directed from new vertices towards old vertices.
	\end{definition}
	
	\begin{remark}\label{remark:def} 
		$(i)$ Note that the edge connection probabilities are invariant under a rescaling of the vertex-weights. In particular, we may without loss of generality assume for vertex-weight distributions with bounded support that $x_0:=\sup\{x\in\R\,|\, \P{\F\leq x}<1\}=1$.
		
		$(ii)$ The model can be adapted to allow for a \emph{random out-degree}: at every step the newly introduced vertex $n+1$ connects with \emph{every} vertex $\inn$ \emph{independently} with a probability equal to $W_i/S_n$. The results presented below still hold for this model definition as well.
	\end{remark}
	
	Lodewijks and Ortgiese studied certain properties of the Weighted Recursive Graph (WRG) model in~\cite{LodOrt21}. This is a more general version of the WRT model that allows every vertex to connect to $m\in\N$ vertices when introduced, yielding a multigraph when $m>1$. This paper aims to recover and extend some of these results in the tree case ($m=1$) when the vertex-weights are almost surely bounded, i.e.\ $x_0<\infty$. As stated in Remark~\ref{remark:def}$(i)$, we can set $x_0=1$ without loss of generality. To formulate the results we need to assume that the distribution of the weights is sufficiently regular, allowing us to control their extreme value behaviour. In certain cases it is more convenient to formulate the assumptions in terms of the distribution of the random variable $(1-W)^{-1}$:
	\newpage 
	\begin{assumption}[Vertex-weight distribution]\label{ass:weights}
		The vertex-weights $W,(W_i)_{i\in\N}$ are i.i.d.\ strictly positive random variables, which satisfy:
		\begin{itemize}
			\item[\namedlabel{ass:weightsup}{$($C$1)$}] The random variable $W$ takes values in $(0,1]$ and $x_0:=\sup\{x\in \R|\P{W\leq x}<1\}=1$.
			\item[\namedlabel{ass:weightzero}{$($C$2)$}]  
				There exist $C,\rho>0,x_0\in(0,1)$ such that for all $x\in[0,x_0]$ it holds that  $\P{W \leq x} \leq C x^\rho .$
		\end{itemize}
		Furthermore, the vertex-weights satisfy one of the following conditions:
		\begin{enumerate}[labelindent = 1cm, leftmargin = 2.2cm]
			\item[\namedlabel{ass:weightatom}{$($\textbf{Atom}$)$}] The vertex weights follow a distribution that has an atom at one, i.e. there exists a $q_0\in(0,1]$ such that $\P{W=1}=q_0$. (Note that $q_0=1$ recovers the RRT model)
			\item[\namedlabel{ass:weightweibull}{$($\textbf{Weibull}$)$}] The vertex-weights follow a distribution that belongs to the Weibull maximum domain of attraction (MDA). This implies that there exist some $\alpha>1$ and a positive function $\ell$ which is slowly varying at infinity, such that 
			\be 
			\hspace{2cm}\P{W\geq 1-1/x}=\P{(1-W)^{-1}\geq x}=\ell(x)x^{-(\alpha-1)},\qquad  x>0.
			\ee
			\item[\namedlabel{ass:weightgumbel}{$($\textbf{Gumbel}$)$}] The distribution belongs to the Gumbel maximum domain of attraction (MDA) (and $x_0=1$). This implies that there exist sequences $(a_n,b_n)_{n\in\N}$, such that 
			\be 
			\frac{\max_{\inn}\F_i-b_n}{a_n}\toindis \Lambda,
			\ee 
			where $\Lambda$ is a Gumbel random variable. 
			
			\noindent Within this class, we further distinguish the following two sub-classes:
			\begin{enumerate}[labelindent = 1cm, leftmargin = 1.3cm]
				\item[\namedlabel{ass:weighttaufin}{$($\textbf{RV}$)$}]
				There exist $a,c_1,\tau>0$, and $b\in\R$ such that
				\[ \hspace{1.5cm}\P{W>1-1/x}=\P{ (1-\F)^{-1} > x} \sim a x^b 	\e^{-( x/c_1)^\tau} \quad \mbox{as } x \rightarrow \infty .\] 
				\item[\namedlabel{ass:weighttau0}{$($\textbf{RaV}$)$}] 
				There exist $a,c_1>0, b\in\R,$ and $\tau>1$ such that 
				\[ \hspace{1.5cm}\P{W>1-1/x}=\P{ (1-\F)^{-1} > x} \sim a (\log x)^b 	\e^{- (\log (x)/c_1)^\tau} \quad \mbox{as } x \rightarrow \infty .\] 
			\end{enumerate}
		\end{enumerate}  
	\end{assumption}
	
	\begin{remark}
		Condition~\ref{ass:weightzero} is required only for a very specific part of the proof of Proposition~\ref{lemma:degprobasymp}. Though we were unable to omit this assumption, it covers a wide range of vertex-weight distributions. Still, we believe it is a mere technicality that can be overcome.
	\end{remark} 	
	
	Throughout, we will write 
	\[  	\zni := \mbox{ in-degree of vertex } i \mbox{ in } T_n . \]
	Working with the in-degree allows us to (in principle) generalise our methods to graphs with random out-degree, as mentioned in Remark~\ref{remark:def}. Obviously, if the out-degree is fixed, we can recover the results for the degree from our results on the $\zni$.
	
	In~\cite{LodOrt21}, the following results are obtained for the WRG model: if we let $\theta_m:=1+\E{W}/m$,
	\be\label{eq:pk} 
	p_k(m):=\E{\frac{\theta_m-1}{\theta_m-1+W}\Big(\frac{W}{\theta_m-1+W}\Big)^k},\  p_{\geq k}(m):=\sum_{j=k}^\infty p_k(m) =\E{\Big(\frac{W}{\theta_m-1+W}\Big)^k},
	\ee  
	then almost surely for any $k\in\N$ fixed,
	\be \label{eq:pkconv}
	\lim_{n\to\infty}\frac 1n \sum_{i=1}^n \ind_{\{\zni=k\}}=p_k(m)\qquad \text{and} \qquad \lim_{n\to\infty}\frac 1n\sum_{i=1}^n \ind_{\{\zni\geq k\}}=p_{\geq k}(m),
	\ee
	whenever $W$ follows a distribution with a finite mean. In particular, the above is satisfied for all cases in Assumption~\ref{ass:weights}. Moreover, if the vertex-weights are bounded almost surely (without loss of generality $x_0=1$), 
	\be \label{eq:bddmaxconv}
	\max_{\inn}\frac{\zni}{\log_{\theta_m}n}\toas 1.
	\ee 
	In this paper we improve these results when considering the WRT model with almost surely bounded weights. That is, we consider the case $m=1$. For ease of writing, we let 
	\be 
	\theta:=\theta_1=1+\E W\text{ and }p_k:=p_k(1),p_{\geq k}:=p_{\geq k}(1).
	\ee
	First, we are able to extend the result in~\eqref{eq:pkconv} to the case when $k=k(n)$ diverges with $n$ in the sense that the difference between both sides converges to zero in mean, under certain constraints on $k(n)$, and we obtain a convergence rate as well. Combining this result with techniques developed by Addario-Berry and Eslava in~\cite{AddEsl18} for random recursive trees we are then able to identify the higher-order asymptotic behaviour of the maximum degree depending on the cases in Assumption~\ref{ass:weights}. Additionally, in certain cases we are able to derive an asymptotic tail distribution for the maximum degree and obtain an asymptotic normality result for the number of vertices with `near-maximal' degrees (in certain cases). These results can be extended to the model with a \emph{random out-degree} as mentioned in Remark~\ref{remark:def} as well. 
	
	Recall that $\theta:=1+\E W$ and define
	\be\ba \label{eq:xni}
	X^{(n)}_i&:=|\{j\in[n]:\Zm_n(j)= \lfloor \log_\theta n\rfloor +i\}|,\\
	X^{(n)}_{\geq i}&:=|\{j\in[n]:\Zm_n(j)\geq \lfloor \log_\theta n\rfloor +i\}|.
	\ea\ee 	
	For certain classes of vertex-weight distributions, we can prove the distributional convergence of these quantities along subsequences, as is the case for the RRT model in~\cite{AddEsl18}. This result can be formulated in terms of convergence of point processes. Let $\Z^*:=\Z\cup \{\infty\}$ and endow $\Z^*$ with the metric $d(i,j)=|2^{-i}-2^{-j}|$ and $d(i,\infty)=2^{-i}$ for any $i,j\in\Z$, so that $[i,\infty]$ is a compact set. Following the notation in~\cite{DaleyVere-JonesII}, let $\mathcal M^\#_{\Z^*}$ be the space of boundedly finite measures on $\Z^*$ (which in this case corresponds to  locally finite measures) equipped with the vague topology. If we let $\mathcal P$ be a Poisson point process on $\R$ with intensity measure $\lambda(\d x):=q_0\theta^{-x}\log \theta\, \d x,q_0\in(0,1]$, and define 
	\be \label{eq:peps}
	\mathcal P^\eps:=\sum_{x\in \mathcal P}\delta_{\lfloor x+\eps\rfloor},\qquad \mathcal P^{(n)}:=\sum_{\inn}\delta_{\zni-\lfloor \log_\theta n\rfloor},\quad \text{ and } \quad \eps_n:=\log_\theta n-\lfloor \log_\theta n\rfloor,
	\ee 
	then $\mathcal P^\eps$ and $\mathcal P^{(n)}$ are random elements in $\mathcal M^\#_{\Z^*}$ and we can provide conditions such that $\mathcal P^{(n_\ell)}$ converges weakly to $\mathcal P^\eps$, for subsequences $(n_\ell)_{\ell\in\N}$ such that $\eps_{n_\ell}\to\eps$ as $\ell\to\infty$. We abuse notation to write $\mathcal P^\eps(i)=\mathcal P^\eps(\{i\})=|\{x\in\mathcal P: \lfloor x+\eps\rfloor =i\}|=|\{x\in\mathcal P: x\in [i-\eps,i+1-\eps)\}|$.
	The choice of the metric on $\Z^*$ is convenient since, by Theorem 11.1.VII of~\cite{DaleyVere-JonesII}, weak convergence in $\mathcal M^\#_{\Z^*}$ is equivalent to convergence of finite dimensional distributions, see Definition 11.1.IV of~\cite{DaleyVere-JonesII}. In particular, the weak convergence of $\cP^{(n_\ell)}$ implies the convergence in distribution of $X_{\ge i}^{(n_l)} = \cP^{(n_\ell)}[i,\infty)$.
	
	We now state our main results, which we split into several theorems based on the cases in Assumption~\ref{ass:weights}.
	
	\begin{theorem}[High degrees in WRTs \ref{ass:weightatom} case]\label{thrm:mainatom}
		Consider the WRT model in Definition~\ref{def:WRT} with vertex-weights $(W_i)_{\inn}$ that satisfy the \ref{ass:weightatom} case in Assumption~\ref{ass:weights} for some $q_0\in(0,1]$. Fix $\eps\in[0,1]$. Let $(n_\ell)_{\ell\in\N}$ be a positive integer sequence such that $\eps_{n_\ell}\to \eps$ as $\ell\to\infty$. Then $\mathcal P^{(n_\ell)}$ converges weakly in $\mathcal M^\#_{\Z^*}$ to $\mathcal P^\eps$ as $\ell\to \infty$. Equivalently, for any $i<i'\in\Z$, jointly as $\ell\to \infty$,
		\be 
		(X^{(n_\ell)}_i,X^{(n_\ell)}_{i+1},\ldots,X^{(n_\ell)}_{i'-1},X^{(n_\ell)}_{\geq i'})\toindis (\mathcal P^\eps(i),\mathcal P^\eps(i+1),\ldots,\mathcal P^\eps (i'-1),\mathcal P^\eps([i',\infty)).
		\ee 
	\end{theorem}
	
	We note that this result recovers and extends~\cite[Theorem $1.2$]{AddEsl18}, in which an equivalent result is presented which only holds for the random recursive tree, i.e.\ the particular case of the weighted recursive tree in which $q_0:=\P{W=1}=1$. In Theorem~\ref{thrm:mainatom} we allow $q_0\in(0,1)$ as well, under the additional condition~\ref{ass:weightzero}.
	
	When the vertex-weight distribution belongs to the Weibull MDA, we can prove convergence in probability under a deterministic second-order scaling, but are unable to obtain what we conjecture to be a random third-order term similar to the result in Theorem~\ref{thrm:mainatom}:
	
	\begin{theorem}[High degrees in WRTs, \ref{ass:weightweibull} case]\label{thrm:mainweibull}
		Consider the WRT model in Definition~\ref{def:WRT} with vertex-weights $(W_i)_{\inn}$ that satisfy the \ref{ass:weightweibull} case in Assumption~\ref{ass:weights} for some $\alpha>1$. Then, 
		\be 
		\max_{\inn}\frac{\zni-\log_\theta n}{\log_\theta\log_\theta n}\toinp -(\alpha-1).
		\ee
	\end{theorem}
	
	Finally, when the vertex-weight distribution belongs to the Gumbel MDA, we have similar results compared to the Weibull MDA case in the above theorem. Here we are also able to obtain a deterministic second-order scaling for both the~\ref{ass:weighttaufin} and~\ref{ass:weighttau0} sub-cases, as well as a third- and fourth-order scaling for the~\ref{ass:weighttau0} sub-case:
	
	\begin{theorem}[High degrees in WRTs, \ref{ass:weightgumbel} case]\label{thrm:maingumbel}
		Consider the WRT model in Definition~\ref{def:WRT} with vertex-weights $(W_i)_{\inn}$ that satisfy the \ref{ass:weightgumbel} case in Assumption~\ref{ass:weights}.\\
		In the~\ref{ass:weighttaufin} sub-case, let $\gamma:=1/(1+\tau)$. Then, 
		\be \label{eq:gumbrv2nd}
		\max_{\inn}\frac{\zni-\log_\theta n}{(\log_\theta n)^{1-\gamma}}\toinp -\frac{\tau^\gamma}{(1-\gamma)\log \theta}\Big(\frac{1-\theta^{-1}}{c_1}\Big)^{1-\gamma}=:-C_{\theta,\tau,c_1}.
		\ee
		In the~\ref{ass:weighttau0} sub-case, 
		\be \label{eq:gumbrav2nd}
		\max_{\inn}\frac{\zni-\log_\theta n+C_1(\log_\theta\log_\theta n)^\tau-C_2(\log_\theta \log_\theta n)^{\tau-1}\log_\theta \log_\theta\log_\theta n}{(\log_\theta\log_\theta n)^{\tau-1}}\toinp C_3,
		\ee 
		where 
		\be\ba \label{eq:c123}
		C_1&:=(\log \theta )^{\tau-1}c_1^{-\tau},\qquad C_2:=(\log \theta )^{\tau-1}\tau(\tau-1)c_1^{-\tau}, \\ 
		C_3&:=\big(\log_\theta(\log\theta) (\tau-1)\log\theta-\log(\e c_1^\tau(1-\theta^{-1})/\tau)\big)(\log \theta )^{\tau-2}\tau c_1^{-\tau}.
		\ea \ee
	\end{theorem}
	
	We see that only in the~\ref{ass:weightatom} case we are able to obtain the higher-order asymptotics up to random order. This is due to the fact that, in this particular case, the vertices with high degree all have vertex-weight one. In the other classes covered in Theorems~\ref{thrm:mainweibull} and~\ref{thrm:maingumbel} vertices with high degrees have a vertex-weight close to one, which causes their degrees to grow slightly slower. This results in the higher-order asymptotics as observed in these theorems.
	
	We are able to obtain more precise results related to the maximum and near-maximum degree vertices in the~\ref{ass:weightatom} case as well, which again recover and extend the results in~\cite{AddEsl18}.
	
	\begin{theorem}[Asymptotic tail distribution for maximum degree in~\ref{ass:weightatom} case]\label{thrm:maxtail}
		$\,$ \\ Consider the WRT model in Definition~\ref{def:WRT} with vertex-weights $(W_i)_{\inn}$ that satisfy the \ref{ass:weightatom} case in Assumption~\ref{ass:weights} for some $q_0\in(0,1]$ and recall $\eps_n$ from~\eqref{eq:peps}. Then, for any integer-valued $i_n=i(n)$ with $i_n+\log_\theta n <(\theta/(\theta-1)) \log n$ and $\liminf_{n\to\infty}i_n>-\infty$,
		\be 
		\P{\max_{j\in[n]}\Zm_n(j)\geq \lfloor \log_\theta n\rfloor +i_n}=\big(1-\exp(-q_0\theta^{-i_n+\eps_n})\big)(1+o(1)).
		\ee 
		Moreover, let $\M_n\subseteq [n]$ denote the $($random$)$ set of vertices that attain the maximum degree in $T_n$, fix $\eps\in[0,1]$, and let $(n_\ell)_{\ell\in\N}$ be a positive integer sequence such that $\eps_{n_\ell}\to\eps$ as $\ell\to\infty$. Then, $|\M_{n_\ell}|\toindis M_\eps$, where $M_\eps$ has distribution
		\be\label{eq:meps}
		\P{M_\eps=k}=\sum_{j\in\Z}\frac{1}{k!}\big(q_0(1-\theta^{-1})\theta^{-j+\eps}\big)^k \e^{-q_0\theta^{-j+\eps}},\qquad k\in\N.
		\ee 
	\end{theorem} 
	
	Finally, we establish an asymptotic normality result for the number of vertices which have `near-maximum' degrees. For a precise definition of `near-maximum', we define sequences $(s_k,r_k)_{k\in\N}$ as 
	\be\ba \label{eq:ek} 
	s_k&:=\inf\big\{x\in(0,1): \P{W\in(x,1)}\leq \exp(-(1-\theta^{-1})(1-x)k)\big\},\\
	r_k&:=\exp(-(1-\theta^{-1})(1-s_k)k).
	\ea\ee
	As a result, $r_k$ can be used as the error term in the asymptotic expression of $p_{\geq k}$ (as in~\eqref{eq:pk}) when the weight distribution satisfies the~\ref{ass:weightatom} case (see Theorem~\ref{thrm:pkasymp}) and is essential in quantifying how much smaller `near-maximum' degrees are relative to the maximum degree of the graph in this case. We note that $r_k$ is decreasing and converges to zero with $k$ (see Lemma~\ref{lemma:rk}), and that in the definition of $s_k$ and $r_k$ we can allow the index to be continuous rather than just an integer (the proof of Lemma~\ref{lemma:rk} can be adapted to still hold in this case). We can then formulate the following theorem:
	
	\begin{theorem}[Asymptotic normality of near-maximum degree vertices, \ref{ass:weightatom} case]\label{thrm:asympnormal}
		$\,$ \\ Consider the WRT model in Definition~\ref{def:WRT} with vertex-weights $(W_i)_{\inn}$ that satisfy the \ref{ass:weightatom} case in Assumption~\ref{ass:weights} for some $q_0\in(0,1]$. Then, for an integer-valued $i_n=i(n)\to-\infty$ such that $i_n=o(\log n \wedge |\log r_{\log_\theta n}|)$,
		\be 
		\frac{X^{(n)}_{i_n}-q_0(1-\theta^{-1})\theta^{-i_n+\eps_n}}{\sqrt{q_0(1-\theta^{-1})\theta^{-i_n+\eps_n}}}\toindis N(0,1).
		\ee 
	\end{theorem} 
	
	\begin{remark}\label{rem:asympnorm}
		The constraint $i_n=o(\log n\wedge \log r_{\log_\theta n})$ can be simplified by providing more information on the tail of the weight distribution. Only when $W$ has an atom at one and support bounded away from one do we have that $o(\log n\wedge \log r_{\log_\theta n})=o(\log n)$. That is, when there exists an $s\in(0,1)$ such that $\P{W\in(s,1)}=0$. In that case, we can set $s_k= s$ and $r_k= \exp(-(1-\theta^{-1})(1-s)k)$ for all $k$ large, so that 
		\be
		\log r_{ \log_\theta n}= -(1-\theta^{-1})(1-s)\log_\theta n,
		\ee 
		so that indeed $o(\log n\wedge \log r_{\log_\theta n})=o(\log n)$. In all other cases it follows that $s_k\uparrow 1$, so that $\log r_{\log_\theta n}=o(\log n)$ and the constraint simplifies to $i_n=o(\log r_{\log_\theta n})$. 
	\end{remark} 
	
	\textbf{Outline of the paper}\\
	In Section~\ref{sec:overview} we provide a short overview and intuitive idea of the proofs of Theorems~\ref{thrm:mainatom},~\ref{thrm:mainweibull}, \ref{thrm:maingumbel},~\ref{thrm:maxtail} and~\ref{thrm:asympnormal}. In Section~\ref{sec:exres} we discuss two examples of vertex-weight distributions which satisfy the~\ref{ass:weightweibull} and~\ref{ass:weightgumbel} cases, respectively, for which more precise results can be obtained. We then provide the key concepts and results that are used in the proofs of the main theorems in Section~\ref{sec:taildeg}. We use these results to prove the main theorems in Section~\ref{sec:mainproof}. Finally, in Section~\ref{sec:ex} we provide the necessary techniques and results, comparable to what is presented in Section \ref{sec:taildeg}, to prove the statements regarding the examples of Section \ref{sec:exres}.
	
	\section{Intuitive idea of (the proof of) the main theorems}\label{sec:overview}
	
	We provide a short intuitive idea as to why the results stated in Section~\ref{sec:results} hold.
	
	The main elements in obtaining a more precise understanding of the behaviour of the maximum degree of the WRT are the following:
	\begin{enumerate}
		\item[\namedlabel{item:i}{$(i)$}] A precise expression of the tail distribution of the in-degree of uniformly at random selected vertices $(v_\ell)_{\ell\in[k]}$, for any $k\in\N$. That is,
		\be 
		\P{\Zm_n(v_\ell)\geq m_\ell\text{ for all } \ell\in[k]}=\prod_{\ell=1}^k p_{\geq m_\ell}(1+o(n^{-\beta})),
		\ee 
		for some $\beta>0$ and where the $m_\ell\in\N$ are such that $m_\ell\leq c \log n$ for some $c\in(0,\theta/(\theta-1))$. This extends~\eqref{eq:pkconv} in the sense of convergence in mean to $k\in\N$ many uniformly at random selected vertices rather than just one, and allows the $m_\ell$ to grow with $n$ rather than being fixed. Moreover, the error term $1+o(n^{-\beta})$ extends previously known results as well, for which no convergence rate was known.
		\item[\namedlabel{item:ii}{$(ii)$}] The asymptotic behaviour of $p_{\geq k}$, as defined in~\eqref{eq:pk}, as $k\to\infty$ for each case in Assumption~\ref{ass:weights}. 
	\end{enumerate}
	Element \ref{item:i}, which is proved in Proposition~\ref{lemma:degprobasymp}, allows us to obtain bounds on the probability of the event $\{\max_{\jnn}\znj\geq k_n\}$ for any sequence $k_n\to\infty $ as $n\to\infty$. These probabilities can be expressed in terms of $n p_{\geq k_n}$ (using union bounds and 
	the second moment method in the form of the Chung-Erd\H os inequality, as is shown in Lemma~\ref{lemma:maxdegwhp}). By~\ref{item:ii} we can then precisely quantify $k_n$ such that these bounds either tend to zero or one, which implies whether $\{\max_{\jnn}\zni\geq k_n\}$ does or does not hold with high probability. This is the main approach for Theorems~\ref{thrm:mainweibull} and~\ref{thrm:maingumbel}.
	
	To obtain the random limits described in terms of the Poisson process $\mathcal P^\eps$, as in Theorem~\ref{thrm:mainatom}, we use a similar approach as in~\cite{AddEsl18}. Both~\ref{item:i} and~\ref{item:ii} are still essential, but are now used to obtain factorial moments of the quantities $X^{(n)}_i$ and $X^{(n)}_{\geq i}$, defined in~\eqref{eq:xni}, as shown in Proposition~\ref{prop:factmean}. More specifically, for any $i<i'\in\Z$ and $a_i,\ldots,a_{i'}\in\N_0$, and recalling that $(x)_k:=x(x-1)\ldots (x-(k-1))$,
	\be\label{eq:factmean}
	\E{\Big(X^{(n)}_{\geq i'}\Big)_{a_{i'}}\prod_{k=i}^{i'-1}\Big(X^{(n)}_i\Big)_{a_i}}=\Big(q_0\theta^{-i'+\eps_n}\Big)^{a_{i'}}\prod_{k=i}^{i'-1}\Big(q_0(1-\theta^{-1})\theta^{-k+\eps_n}\Big)^{a_k}(1+o(1)).
	\ee 
	We stress that the specific form of the right-hand side is due to the underlying assumption in Theorem~\ref{thrm:mainatom} that the vertex-weight distribution has an atom at one, as in the~\ref{ass:weightatom} case of Assumption~\ref{ass:weights}. The error term can be specified in more detail, but we omit this as it serves no further purpose here. The result essentially follows directly from these estimates by observing that the right-hand side of~\eqref{eq:factmean} can be understood as an approximation of the factorial moment of the Poisson random variables $\mathcal P^\eps([i-\eps,i+1-\eps)),\ldots, \mathcal P^\eps([i'-\eps,\infty))$, when $\eps_n$ converges to some $\eps$.
	
	The equality in~\eqref{eq:factmean} follows from the fact that $X^{(n)}_i$ and $X^{(n)}_{\geq i}$ can be expressed as sums of indicator random variables of disjoint events, so that their factorial means can be understood via the probabilities in~\ref{item:i}. Then, again using the asymptotic behaviour of $p_{\geq k}$ (as in~\ref{item:ii}), allows us to obtain the right-hand side of~\eqref{eq:factmean}.
	
	Finally, Theorems~\ref{thrm:maxtail} and~\ref{thrm:asympnormal} are also a result of~\eqref{eq:factmean}. This is due to the fact that the events $\{\max_{\jnn}\znj \geq \lfloor\log_\theta n\rfloor +i\}$ can be understood via the events $\{X^{(n)}_{\geq i}>0\}$. Again using ideas similar to ones developed in~\cite{AddEsl18} then allow us to obtain the results.
	
	\section{Examples}\label{sec:exres}
	
	In this section we discuss some particular choices of distributions for the vertex-weights for which more precise statements can be made compared to those stated in Section~\ref{sec:results}. The reason we can improve on these more general results is due to a better understanding of the asymptotic behaviour of $p_k$ and $p_{\geq k}$ (see~\eqref{eq:pk}) as $k\to\infty$. As discussed in Section~\ref{sec:overview}, to understand the asymptotic behaviour of the (near-)maximum degree(s) up to random order a very precise asymptotic expression for $p_{\geq k}$ is required. Though not possible in general in the~\ref{ass:weightweibull} and~\ref{ass:weightgumbel} cases of Assumption~\ref{ass:weights}, certain choices of vertex-weight distributions do allow for a more explicit formulation of $p_{\geq k}$, yielding improved asymptotics. The proofs of the results presented here are deferred to Section \ref{sec:ex}.
	
	\begin{example}[Beta distribution]\label{ex:beta}
		We consider a random variable $W$ with a tail distribution 
		\be \label{eq:betacdf}
		\P{W\geq x}=\int_x^1 \frac{\Gamma(\alpha+\beta)}{\Gamma(\alpha)\Gamma(\beta)}s^{\alpha-1}(1-s)^{\beta-1}\,\d s, \qquad x\in[0,1),		
		\ee 
		for some $\alpha,\beta>0$. Clearly, $W$ belongs to the~\ref{ass:weightweibull} case of Assumption~\ref{ass:weights}. Moreover, $W$ satisfies condition~\ref{ass:weightzero} of Assumption~\ref{ass:weights}. Since for any $x_0\in(0,1)$ we can bound $(1-s)^{\beta-1}\leq 1\vee (1-x_0)^{\beta-1}$ for all $s\in(0,x_0)$, it follows that for any $x\in[0,x_0]$,
			\be 
			\P{W\leq x}=\int_0^x \frac{\Gamma(\alpha+\beta)}{\Gamma(\alpha)\Gamma(\beta)}s^{\alpha-1}(1-s)^{\beta-1}\,\d s\leq (1\vee (1-x_0)^{\beta-1})\frac{\Gamma(\alpha+\beta)}{\Gamma(\alpha+1)\Gamma(\beta)}x^\alpha, 
			\ee 
			so that condition~\ref{ass:weightzero} is satisfied with $C=(1\vee (1-x_0)^{\beta-1})\Gamma(\alpha+\beta)/(\Gamma(\alpha+1)\Gamma(\beta))$, the constant $x_0\in(0,1)$ arbitrary, and $\rho=\alpha$. 			
		
		We set, for $\theta:=1+\E W\in(1,2),$
		\be\ba\label{eq:epsnbeta}
		X^{(n)}_i&:=|\{\jnn: \znj=\lfloor \log_\theta n-\beta\log_\theta\log_\theta n \rfloor +i\}|,\\ 
		X^{(n)}_{\geq i}&:=|\{\jnn: \znj\geq \lfloor \log_\theta n-\beta \log_\theta\log_\theta n\rfloor +i\}|,\\
		\eps_n&:=(\log_\theta n-\beta \log_\theta\log_\theta n)-\lfloor\log_\theta n-\beta\log_\theta\log_\theta n \rfloor,\\
		c_{\alpha,\beta,\theta}&:=\frac{\Gamma(\alpha+\beta)}{\Gamma(\alpha)}(1-\theta^{-1})^{-\beta}.
		\ea\ee 
		Then, we can formulate the following results.
		
		\begin{theorem}\label{thrm:betappp}
			Consider the WRT model in Definition~\ref{def:WRT} with vertex-weights $(W_i)_{\inn}$ whose distribution satisfies~\eqref{eq:betacdf} for some $\alpha,\beta>0$, and recall $\eps_n$ and $c_{\alpha,\beta,\theta}$ from~\eqref{eq:epsnbeta}. Fix $\eps\in[0,1]$, let $(n_\ell)_{\ell\in\N}$ be an increasing integer sequence such that $\eps_{n_\ell}\to \eps$ as $\ell \to \infty$ and let $\mathcal P$ be a Poisson point process on $\R$ with intensity measure $\lambda(x)=c_{\alpha,\beta,\theta}\theta^{-x}\log \theta\,\d x$. Define
			\be 
			\mathcal P^\eps:=\sum_{x\in \mathcal P}\delta_{\lfloor x+\eps\rfloor} \quad \text{ and } \quad \mathcal P^{(n)}:=\sum_{\inn}\delta_{\zni-\lfloor \log_\theta n-\beta \log_\theta\log_\theta n \rfloor}.
			\ee 
			Then in $\mathcal M^\#_{\Z^*}$ $($the space of boundedly finite measures on $\Z^*=\Z\cup\{\infty\}$ with the metric defined in Section~\ref{sec:results}$)$, $\mathcal P^{(n_\ell)}$ converges weakly to $\mathcal P^\eps$ as $\ell\to \infty$. Equivalently, for any $i<i'\in\Z$, jointly as $\ell\to \infty$,
			\be 
			(X^{(n_\ell)}_i,X^{(n_\ell)}_{i+1},\ldots,X^{(n_\ell)}_{i'-1},X^{(n_\ell)}_{\geq i'})\toindis (\mathcal P^\eps(i),\mathcal P^\eps(i+1),\ldots,\mathcal P^\eps (i'-1),\mathcal P^\eps([i',\infty))).
			\ee 
		\end{theorem}
		
		We remark that the second-order term $\beta \log_\theta\log_\theta n$ is established in Theorem~\ref{thrm:mainweibull} as well and that the above theorem recovers this result and extends it to the random third-order term, which is similar to the result in Theorem~\ref{thrm:mainatom}.
		
		\begin{theorem}\label{thrm:betamaxtail}
			Consider the WRT model in Definition~\ref{def:WRT} with vertex-weights $(W_i)_{\inn}$ whose distribution satisfies~\eqref{eq:betacdf} for some $\alpha,\beta>0$, and recall $\eps_n$ and $c_{\alpha,\beta,\theta}$ from~\eqref{eq:epsnbeta}. Then, for any integer-valued $i_n$ with $i_n= \delta\log_\theta n+o(\log n)$ for some $\delta\in[0,\theta\log\theta/(\theta-1)-1)$ and $\liminf_{n\to\infty}i_n>-\infty$,
			\be \ba
			\mathbb P{}&\Big(\max_{\jnn}\znj\geq \lfloor \log_\theta n-\beta \log_\theta\log_\theta n\rfloor +i_n\Big)=\big(1-\exp\big(-c_{\alpha,\beta,\theta}(1+\delta)^{-\beta}\theta^{-i_n+\eps_n}\big)\big)(1+o(1)).
			\ea \ee 
			Moreover, let $\M_n\subseteq [n]$ denote the (random) set of vertices that attain the maximum degree in $T_n$, fix $\eps\in[0,1]$ and let $(n_\ell)_{\ell\in\N}$ be a positive integer sequence such that $\eps_{n_\ell}\to\eps$ as $\ell\to\infty$. Then, $|\M_{n_\ell}|\toindis M_\eps$, where $M_\eps$ has distribution
			\be
			\P{M_\eps=k}=\sum_{j\in\Z}\frac{1}{k!}\big(c_{\alpha,\beta,\theta}(1-\theta^{-1})\theta^{-j+\eps}\Big)^k\exp\big(-c_{\alpha,\beta,\theta}\theta^{-j+\eps}\big),\qquad k\in\N.
			\ee
		\end{theorem}
		
		\begin{theorem}\label{thrm:betaasympnorm}
			Consider the WRT model in Definition~\ref{def:WRT} with vertex-weights $(W_i)_{\inn}$ whose distribution satisfies~\eqref{eq:betacdf} for some $\alpha,\beta>0$ and recall $\eps_n$ and $c_{\alpha,\beta,\theta}$ from~\eqref{eq:epsnbeta}. Then, for an integer-valued $i_n\to-\infty$ such that $i_n=o(\log\log n)$, 
			\be 
			\frac{X^{(n)}_{i_n}-c_{\alpha,\beta,\theta}(1-\theta^{-1})\theta^{-i_n+\eps_n}}{\sqrt{c_{\alpha,\beta,\theta}(1-\theta^{-1})\theta^{-i_n+\eps_n}}}\toindis N(0,1).
			\ee 
		\end{theorem}
		
		The three theorems are the analogue of Theorems~\ref{thrm:mainatom},~\ref{thrm:maxtail} and~\ref{thrm:asympnormal}, respectively, where we now consider vertex-weights distributed according to a beta distribution rather than a distribution with an atom at one.		
	\end{example} 
	
	\begin{example}[Fraction of `gamma' random variables]\label{ex:gumb}
		We consider a random variable $W$ with a tail distribution 
		\be \label{eq:gumbex}
		\P{W\geq x}=(1-x)^{-b}\e^{-x/(c_1(1-x))},\qquad x\in[0,1),
		\ee
		for some $b\in\R,c_1>0$ with $bc_1\leq 1$. The random variable $(1-W)^{-1}$ belongs to the Gumbel maximum domain of attraction, as 
		\be 
		\P{(1-W)^{-1}\geq x}=\P{W\geq 1-1/x}=\e^{1/c_1}x^b\e^{-x/c_1},\qquad x\geq 1,
		\ee 
		so that $W$ belongs to the Gumbel MDA as well by~\cite[Lemma $2.6$]{LodOrt21}, and satisfies the \ref{ass:weightgumbel}-\ref{ass:weighttaufin} sub-case with $a=\e^{1/c_1},b\in\R,c_1>0$, and $\tau=1$. As a result, $X:=(1-W)^{-1}$ is a `gamma' random variable in the sense that its tail distribution is asymptotically equal to that of a gamma random variable, up to constants. We can then write $W=(X-1)/X$, a fraction of these `gamma' random variables.
		
		Let $f_W$ denote the probability density function of $W$. It follows from~\eqref{eq:gumbex} that 
		\be 
		f_W(x)=(1/c_1 – b(1-x)) (1-x)^{-(b+2)} \e^{ - x/ (c_1 (1-x))}.
		\ee 
		Observe that the requirement $bc_1\leq 1$ ensures that $f_W(x)\geq 0$ for all $x\in[0,1)$. As $\lim_{x\uparrow 1}f_W(x)=0$ and the density is bounded from above in a neighbourhood of the origin, it follows that $f_W(x)\leq C$ for some $C>0$ and all $x\in[0,1)$. As a result, this distribution satisfies condition~\ref{ass:weightzero} of Assumption~\ref{ass:weights} with $\rho=1$, $x_0\in(0,1)$ arbitrary, and some sufficiently large $C>0$.
		
		Recall $C_{\theta,\tau,c_1}$ from~\eqref{eq:gumbrv2nd}. We set, for $\theta:=1+\E W\in(1,2),$
		\be \label{eq:c}
		C:=\e^{c_1^{-1}(1-\theta^{-1})/2}\sqrt{\pi}c_1^{-1/4+b/2}(1-\theta^{-1})^{1/4+b/2},\qquad \text{and}\quad c_{c_1,b,\theta}:=C\theta^{C_{\theta,1,c_1}^2/2},
		\ee
		and
		\be\ba\label{eq:epsngamma}
		X^{(n)}_i:={}&\big|\big\{\jnn: \znj=\big\lfloor \log_\theta n-C_{\theta,1,c_1}\sqrt{\log_\theta n}+(b/2+1/4) \log_\theta \log_\theta n \big\rfloor +i\big\}\big|,\\ 
		X^{(n)}_{\geq i}:={}&\big|\big\{\jnn: \znj\geq \big\lfloor \log_\theta n-C_{\theta,1,c_1}\sqrt{\log_\theta n}+(b/2+1/4) \log_\theta \log_\theta n \big\rfloor +i\big\}\big|,\\
		\eps_n:={}&\big(\log_\theta n-C_{\theta,1,c_1}\sqrt{\log_\theta n}+(b/2+1/4) \log_\theta \log_\theta n\big )\\
		&-\big\lfloor\log_\theta n-C_{\theta,1,c_1}\sqrt{\log_\theta n}+(b/2+1/4) \log_\theta \log_\theta n \big \rfloor.
		\ea\ee 
		Then, we can formulate the following results.
		
		\begin{theorem}\label{thrm:gumbppp}
			Consider the WRT model in Definition~\ref{def:WRT} with vertex-weights $(W_i)_{\inn}$ whose distribution satisfies~\eqref{eq:gumbex} for some $b\in\R,c_1>0$ such that $bc_1\leq 1$, and recall $C_{\theta,\tau,c_1}$ from~\eqref{eq:gumbrav2nd}. Then, 
			\be \label{eq:gumb3rd}
			\max_{\inn}\frac{\zni-\log_\theta n+C_{\theta,1,c_1}\sqrt{\log_\theta n}}{\log_\theta \log_\theta n}\toinp \frac{b}{2}+\frac 14. 
			\ee 
			Furthermore, recall $\eps_n$ from~\eqref{eq:epsngamma}, fix $\eps\in[0,1]$, and let $(n_\ell)_{\ell\in\N}$ be an increasing integer sequence such that $\eps_{n_\ell}\to \eps$ as $\ell \to \infty$. Let $\mathcal P$ be a Poisson point process on $\R$ with intensity measure $\lambda(x)=c_{c_1,b,\theta}\theta^{-x}\log \theta\, \d x$, where we recall $c_{c_1,b,\theta}$ from~\eqref{eq:c}. Define
			\be 
			\mathcal P^\eps:=\sum_{x\in \mathcal P}\delta_{\lfloor x+\eps\rfloor} \quad \text{ and }\quad \mathcal P^{(n)}:=\sum_{\inn}\delta_{\zni-\lfloor \log_\theta n-C_{\theta,1,c_1}\sqrt{\log_\theta n}+(b/2+1/4) \log_\theta \log_\theta n \rfloor}.
			\ee 
			Then in $\mathcal M^\#_{\Z^*}$ $($the space of boundedly finite measures on $\Z^*=\Z\cup\{\infty\}$ with the metric defined in Section~\ref{sec:results}$)$, $\mathcal P^{(n_\ell)}$ converges weakly to $\mathcal P^\eps$ as $\ell\to \infty$. Equivalently, for any $i<i'\in\Z$, jointly as $\ell\to \infty$,
			\be 
			(X^{(n_\ell)}_i,X^{(n_\ell)}_{i+1},\ldots,X^{(n_\ell)}_{i'-1},X^{(n_\ell)}_{\geq i'})\toindis (\mathcal P^\eps(i),\mathcal P^\eps(i+1),\ldots,\mathcal P^\eps (i'-1),\mathcal P^\eps([i',\infty))).
			\ee 
		\end{theorem}
		
		We remark that the second-order term in~\eqref{eq:gumb3rd} is established in Theorem~\ref{thrm:maingumbel},~\eqref{eq:gumbrv2nd}, as well. The above theorem recovers this former result and extends it to the third-order rescaling and to the random fourth-order term, which is similar to the result in Theorem~\ref{thrm:mainatom}.
		
		\begin{theorem}\label{thrm:gumbmaxtail}
			Consider the WRT model in Definition~\ref{def:WRT} with vertex-weights $(W_i)_{\inn}$ whose distribution satisfies~\eqref{eq:gumbex} for some $b\in \R,c_1>0$ such that $bc_1\leq 1$, and recall $C_{\theta,\tau,c_1}$, $c_{c_1,b,\theta}$, and $\eps_n$ from~\eqref{eq:gumbrav2nd}, \eqref{eq:c}, and~\eqref{eq:epsngamma}, respectively. Then, for any integer-valued $i_n$ with $i_n= \delta\sqrt{\log_\theta n}+o(\sqrt{\log n})$ for some $\delta\geq 0$ and $\liminf_{n\to\infty}i_n>-\infty$,
			\be \ba
			\mathbb P{}&\Big(\max_{\jnn}\znj\geq \lfloor \log_\theta n-C_{\theta,1,c_1}\sqrt{\log_\theta n}+(b/2+1/4) \log_\theta \log_\theta n \rfloor+i_n\Big)\\
			&=\big(1-\exp\big(-c_{c_1,b,\theta}\theta^{-i_n+\eps_n-\delta C_{\theta,1,c_1}/2}\big)\big)(1+o(1)).
			\ea \ee 
			Moreover, let $\M_n\subseteq [n]$ denote the $($random$)$ set of vertices that attain the maximum degree in $T_n$, fix $\eps\in[0,1]$ and let $(n_\ell)_{\ell\in\N}$ be a positive integer sequence such that $\eps_{n_\ell}\to\eps$ as $\ell\to\infty$. Then, $|\M_{n_\ell}|\toindis M_\eps$, where $M_\eps$ has distribution
			\be
			\P{M_\eps=k}=\sum_{j\in\Z}\frac{1}{k!}\big(c_{c_1,b,\theta}(1-\theta^{-1})\theta^{-j+\eps}\big)^k\exp\big(-c_{c_1,b,\theta}\theta^{-j+\eps}\big),\qquad k\in\N.
			\ee
		\end{theorem}
		
		\begin{theorem}\label{thrm:gumbasympnorm}
			Consider the WRT model in Definition~\ref{def:WRT} with vertex-weights $(W_i)_{\inn}$ whose distribution satisfies~\eqref{eq:gumbex} for some $b\in \R,c_1>0$ such that $bc_1\leq 1$, and recall $c_{c_1,b,\theta}$ and $\eps_n$ from~\eqref{eq:c} and~\eqref{eq:epsngamma}, respectively. Then, for an integer-valued $i_n\to-\infty$ such that $i_n=o(\log\log n)$, 
			\be 
			\frac{X^{(n)}_{i_n}-c_{c_1,b,\theta}(1-\theta^{-1})\theta^{-i_n+\eps_n}}{\sqrt{c_{c_1,b,\theta}(1-\theta^{-1})\theta^{-i_n+\eps_n}}}\toindis N(0,1).
			\ee 
		\end{theorem}
		
		The three theorems are the analogue of Theorems~\ref{thrm:mainatom},~\ref{thrm:maxtail} and~\ref{thrm:asympnormal}, respectively, where we now consider vertex-weights distributed according to a distribution as in~\eqref{eq:gumbex} rather than a distribution with an atom at one.
		
	\end{example}
	
	In both examples we see that a better understanding of the asymptotic behaviour of the tail of the degree distribution, $(p_{\geq k})_{k\in\N_0}$, allows us to identify the higher-order asymptotic behaviour of the (near-)maximum degree(s). It also shows that a higher-order random limit as in the sense of Theorems~\ref{thrm:betappp} and~\ref{thrm:gumbppp} is not expressed just by vertex-weights whose distribution has an atom at one, and we conjecture that this result is in fact universal for \emph{all} vertex-weights distributions with bounded support.
	
	\section{Degree tail distributions and factorial moments}\label{sec:taildeg}
	
	In this section we state and prove the key elements required to prove the main results as stated in Section~\ref{sec:results}. We stress that the results presented and proved in this section cover all the classes introduced in Assumption~\ref{ass:weights} (in fact, they cover any vertex-weight $W$ that satisfies conditions~\ref{ass:weightsup} and~\ref{ass:weightzero}) and that the distinction between the classes of Assumption~\ref{ass:weights} follows in Section~\ref{sec:mainproof}. The intermediate results presented here form the main technical contribution, and their proofs are the main body of the paper. 
	
	\subsection{Statement of intermediate results and main ideas}

	As discussed in Section~\ref{sec:overview}, to understand the asymptotic behaviour of the maximum degree and near-maximum degrees we require a more precise understanding of the convergence in mean of the empirical degree distribution. To that end, we present the following result:
	
	\begin{proposition}[Distribution of typical vertex degrees]\label{lemma:degprobasymp}
		Let $W$ be a positive random variable that satisfies conditions~\ref{ass:weightsup} and~\ref{ass:weightzero} of Assumption~\ref{ass:weights}. Consider the WRT model in Definition~\ref{def:WRT} with vertex-weights $(W_i)_{\inn}$ which are i.i.d.\ copies of $W$, fix $k\in\N$, and let $(v_\ell)_{\ell\in[k]}$ be $k$ vertices selected uniformly at random without replacement from $[n]$. For a fixed $c\in (0,\theta/(\theta-1))$, there exist $\beta\geq \beta'>0$ such that uniformly over non-negative integers $m_\ell< c\log n$ with $\ell\in[k]$,
		\begin{align}
			\P{\Zm_n(v_\ell)=m_\ell,\text{ for all }\ell\in[k]}&= \prod_{\ell=1}^k \E{\frac{\E{W}}{\E{W}+W}\Big(\frac{W}{\E{W}+W}\Big)^{m_\ell}}\big(1+o\big(n^{-\beta}\big)\big),\label{eq:degdist}
			\intertext{and}
			\P{\Zm_n(v_\ell)\geq m_\ell,\text{ for all }\ell\in[k]}&= \prod_{\ell=1}^k \E{\Big(\frac{W}{\E{W}+W}\Big)^{m_\ell}}\big(1+o\big(n^{-\beta'}\big)\big).	\label{eq:degtail}
		\end{align}
	\end{proposition}
	
	\begin{remark}\label{rem:degtail}
		$(i)$ In~\cite[Lemma $1$]{DevLu95}, it is proved that the degrees $(\znj)_{\jnn}$ are negative quadrant dependent when considering the random recursive tree model (the WRT with deterministic weights, all equal to $1$). That is, for any $k\in\N$ and any $j_1\neq \ldots \neq j_k\in[n]$ and $m_1,\ldots,m_k\in\N$, 
		\be
		\P{\bigcap_{\ell=1}^k\{\Zm_n(j_\ell)\geq m_\ell\}}\leq\prod_{j=1}^k \P{\Zm_n(j_\ell)\geq m_\ell}.
		\ee 
		This property only holds for the conditional probability measure $\mathbb P_W$ when considering the WRT (or, more generally, the WRG) model, as follows from~\cite[Lemma $7.1$]{LodOrt21}, and can be obtained `asymptotically' for the probability measure $\mathbb P$, as in the proof of~\cite[Theorem $2.8$, Bounded case]{LodOrt21}. Proposition~\ref{lemma:degprobasymp} improves on this by establishing \emph{asymptotic independence} under the non-conditional probability measure $\mathbb P$ of the degrees of typical vertices, which allows us to extend the results in~\cite{LodOrt21} to more precise asymptotics.
		
		$(ii)$ We note that the result only requires the two main conditions in Assumption~\ref{ass:weights} on the behaviour of the vertex-weight distribution close to zero and one. Hence, results for other vertex-weight distributions that do not satisfy any of the particular cases outlined in this assumption can be obtained as well using the methods presented in this paper.
		
		$(iii)$ The result in Proposition~\ref{lemma:degprobasymp} improves on known results, especially those in~\cite{FouIyerMaiSulz19,Iyer20}. In these papers similar techniques are used to prove a weaker result, in which the $m_\ell$ are not allowed to diverge with $n$ and where no convergence rate is provided.
		
		$(iv)$ The constraint $c\in(0,\theta/(\theta-1))$ arises from the following: when $m=c\log n+o(\log n)$, as we shall encounter later in the proof of Lemma~\ref{lemma:inepsterm}, see Equation~\eqref{eq:remarkref}, for $\eps>0$ small,
			\be 
			\frac1n \sum_{i<n^\eps}\P{\Zm_n(v_1)\geq m}\leq \exp\Big(\log n\Big(-(1-\eps)+(c+o(1))\Big(1-\frac{1}{c(\theta-1)}+\log\Big(\frac{1}{c(\theta-1)}\Big)\Big)\Big)\Big). 
			\ee 
			We want the contribution of left-hand side, compared to the left-hand side of~\eqref{eq:degtail} to be negligible. Hence, it needs to be $o(p_{\geq m})=o(\theta^{-m})$ (recall $p_{\geq m}$ from~\eqref{eq:pk}). We thus have the requirement that
			\be \label{eq:cineq}
			-(1-\eps)+c\Big(1-\frac{1}{c(\theta-1)}+\log\Big(\frac{1}{c(\theta-1)}\Big)\Big)>-c\log \theta	
			\ee 
			is satisfied. The inequality in~\eqref{eq:cineq} becomes an equality exactly when $c=\theta/(\theta-1)$ and $\eps=0$. Hence the need for $c$ to be strictly smaller than $\theta/(\theta-1)$. Also, a more technical reason is  that certain error terms no longer converge to zero with $n$ when $c\geq \theta/(\theta-1)$. Moreover, we observe that $\frac{1}{\log\theta}< \frac{\theta}{\theta-1}$ when $\theta\in(1,2]$, so that Proposition~\ref{lemma:degprobasymp} covers degrees that extend \emph{beyond} the value of the maximum degree (see~\eqref{eq:bddmaxconv}). This is a crucial element of Proposition~\ref{lemma:degprobasymp} and the reason that strong claims regarding the largest degrees in the tree can be proved.
	\end{remark}
	
	To use this (tail) distribution of $k$ typical vertices $v_1,\ldots,v_k$, a precise expression for the expected values on the right-hand side in Proposition~\ref{lemma:degprobasymp} is required. Recall $p_k$ from~\eqref{eq:pk}. The following theorem comes from~\cite[Theorem $2.7$]{LodOrt21}, in which the maximum degree of weighted recursive graphs is studied for a large class of vertex-weight distribution and in which asymptotic expressions of $p_k$ are presented.
	
	\begin{theorem}[\cite{LodOrt21}, Asymptotic behaviour of $p_k$]\label{thrm:pkasymp}
		Recall that $\theta:=1+\E W$. We consider the different cases with respect to the vertex-weights as in Assumption~\ref{ass:weights}, but do not require condition~\ref{ass:weightzero} to hold.
		\begin{enumerate}[labelindent = 1cm, leftmargin = 2.2cm]
			\item[\namedlabel{thrm:pkatom}{$(\mathrm{\mathbf{Atom}})$}] Recall that $q_0=\P{W=1}>0$ and recall $r_k$ from~\eqref{eq:ek}. Then,
			\be \label{eq:pkatom}
			\hspace{2cm} p_k=q_0(1-\theta^{-1})\theta^{-k}\big(1+\mathcal O(r_k)\big).
			\ee 
			\item[\namedlabel{thrm:pkweibull}{$(\mathrm{\mathbf{Weibull}})$}] Recall that $\alpha>1$ is the power-law exponent. Then, for all $k>1/\E W$,
			\be\label{eq:pkbddweibull}
			\underline L(k)k^{-(\alpha-1)} \theta^{-k}\leq p_k\leq \overline L(k)k^{-(\alpha-1)}\theta^{-k},
			\ee
			where $\underline L$ and $\overline L$ are slowly varying at infinity.\\
			\item[\namedlabel{thrm:pkgumbel}{$(\mathrm{\mathbf{Gumbel}})$}]
			\begin{enumerate}
				\item[$(i)$] If $W$ satisfies the~\ref{ass:weighttaufin} sub-case with parameter $\tau>0$, set $\gamma:=1/(\tau+1)$. Then,
				\be\ba\label{eq:pkbddgumbelrv}
				p_k=\exp\Big(-\frac{\tau^\gamma}{1-\gamma}\Big(\frac{(1-\theta^{-1})k}{c_1}\Big)^{1-\gamma}(1+o(1))\Big)\theta^{-k}.
				\ea\ee
				\item[$(ii)$] If $W$ satisfies the~\ref{ass:weighttau0} sub-case with parameter $\tau>1$,
				\be\ba\label{eq:pkbddgumbelrav}
				\hspace{1.8cm}p_k=\exp\Big(-\Big(\frac{\log k}{c_1}\Big)^\tau\Big(1-\tau(\tau-1)\frac{\log\log k}{\log k}+\frac{K_{\tau,c_1,\theta}}{\log k}(1+o(1))\Big)\Big)\theta^{-k}.
				\ea\ee
				where $K_{\tau,c_1,\theta}:=\tau\log(\e c_1^\tau(1-\theta^{-1})/\tau)$.
			\end{enumerate}
		\end{enumerate}
	\end{theorem}
	
	\begin{remark}\label{rem:pgeqk}
		Equivalent upper and lower bounds can be obtained for $p_{\geq k}$ as in~\eqref{eq:pk} (with $m=1$), by adjusting the multiplicative constants in front of the asymptotic expressions by a factor $(1-\theta^{-1})^{-1}$ only. This is due to the fact that 
		\be 
		(1-\theta^{-1})^{-1}p_k=\E{\frac{1}{1-\theta^{-1}}\frac{\theta-1}{\theta-1+W}\Big(\frac{W}{\theta-1+W}\Big)^k}\geq \E{\Big(\frac{W}{\theta-1+W}\Big)^k}=p_{\geq k}\geq p_k.
		\ee 
		Equivalently, the proof of~\cite[Theorem $2.7$]{LodOrt21} can be adapted to work for $p_{\geq k}$ in which case the same asymptotic behaviour is established for $p_{\geq k}$, but with different constants in the front.
	\end{remark}
	
	We also provide less precise but more general bounds on the degree distribution.
	
	\begin{lemma}\label{lemma:pkbound} 
		Let $W$ be a positive random variable with $x_0:=\sup\{x>0:\P{W\leq x}<1\}=1$. Then, for any $\xi>0$ and $k$ sufficiently large,
		\be 
		(\theta+\xi)^{-k}\leq p_k\leq p_{\geq k}\leq \theta^{-k}.
		\ee 
	\end{lemma} 
	
	\begin{proof} 
		The upper bound on $p_{\geq k}$ directly follows from the fact that $x\mapsto (x/(\theta-1+ x))^k$ is increasing in $x$, so that 
		\be 
		p_{\geq k}=\E{\Big(\frac{W}{\theta-1+W}\Big)^k}\leq \Big(\frac{1}{\theta-1+1}\Big)^k=\theta^{-k}.
		\ee 
		For the lower bound, let us take some $\delta\in(0,\xi/(\theta-1+\xi))$ and define
		\be \label{eq:f}
		f_k(\theta,x):=\frac{\theta-1}{\theta-1+x}\Big(\frac{x}{\theta-1+x}\Big)^k.
		\ee 
		Note that $p_k=\E{f_k(\theta,W)}$. Then, since $f_k(\theta,x)$ is increasing in $x$ on $(0,1]$ for $k$ sufficiently large, 
		\be 
		\E{f_k(\theta,W)}\geq \E{f_k(\theta,W)\ind_{\{W>1-\delta\}}}\geq \P{W>1-\delta}\frac{\theta-1}{\theta-\delta}\Big(\frac{1-\delta}{\theta-\delta}\Big)^k.
		\ee 
		We note that, since $x_0=1$, it holds that $\P{W>1-\delta}>0$ for any $\delta>0$. Now, by the choice of $\delta$, we have $(\theta+\xi)(1-\delta)/(\theta-\delta)>1$, so we can find some $\zeta>0$ sufficiently small so that 
		\be 
		\E{f_k(\theta,W)}(\theta+\xi)^k \geq \P{W>1-\delta}\frac{\theta-1}{\theta-\delta}\Big(\frac{(\theta+\xi)(1-\delta)}{\theta-\delta}\Big)^k\geq (1+\zeta)^k\geq 1,
		\ee 
		as required.	
	\end{proof} 
	Recall the definition of $X_i^{(n)},X_{\geq i}^{(n)}$ and $\eps_n$ from~\eqref{eq:xni} and~\eqref{eq:peps}, respectively. Proposition~\ref{lemma:degprobasymp} combined with Theorem~\ref{thrm:pkasymp} then allows us to obtain the following result.
	
	\begin{proposition}[Factorial moments when vertex-weights satisfy the~\ref{ass:weightatom} case]\label{prop:factmean}
		$\,$\\ Consider the WRT model as in Definition~\ref{def:WRT} with vertex-weights $(W_i)_{\inn}$ that satisfy the~\ref{ass:weightatom} case in Assumption~\ref{ass:weights} for some $q_0\in(0,1]$. Recall $r_k$ from~\eqref{eq:ek}, recall that $\theta:=1+\E W$, and that $(x)_k:=x(x-1)\cdots (x-(k-1))$ for $x\in\R,k\in\N$, and $(x)_0:=1$. For fixed $K\in\N$ and $c\in(0,\theta/(\theta-1))$, there exists $\beta>0$ such that the following holds. For any integer-valued $i_n$ and $i_n',$ such that $0<i_n+\log_\theta n<i_n'+\log_\theta n< c\log n$ and any $(a_j)_{i_n\le j\le i_n'}$ non-negative integer sequence such that $\sum_{j=i_n}^{i_n'}a_j=K$,
		\be \ba
		\E{\Big(X^{(n)}_{\geq i_n'}\Big)_{a_{i_n'}}\prod_{k=i_n}^{i_n'-1}\Big(X_k^{(n)}\Big)_{a_k}}={}&\Big(q_0\theta^{-i_n'+\eps_n}\Big)^{a_{i_n'}}\prod_{k=i_n}^{i_n'-1}\Big(q_0(1-\theta^{-1})\theta^{-k+\eps_n}\Big)^{a_k}\\
		&\times\big(1+\mathcal O\big(r_{\lfloor \log_\theta n\rfloor +i_n}\vee n^{-\beta}\big)\big).
		\ea\ee 
	\end{proposition}
	
	\begin{remark}
		$(i)$ Intuitively, the result of this proposition tells us that the random variables $(X^{(n)}_k)_{k=i_n}^{i_n'-1}$ and $X(n)_{\geq i_n'}$ are, jointly, approximately Poisson distributed, since their factorial moments approximate those of Poisson random variables. This result hence plays an essential role in proving the weak convergence to a Poisson point process in Theorem~\ref{thrm:mainatom}.
		
		$(ii)$ Related to Remark~\ref{rem:asympnorm}, the error term decays polynomially only if $W$ has an atom at one and support bounded away from one and $\log_\theta n+i_n>\eta \log n$ for some $\eta>0$. That is, when there exists an $s\in(0,1)$ such that $\P{W\in(s,1)}=0$. In that case, $s_k\leq s$ and $r_k\leq \exp(-(1-\theta^{-1})(1-s)k)$ for all $k$ large, so that 
		\be
		r_{\lfloor \log_\theta n\rfloor +i_n}\vee n^{-\beta}\leq \exp(-(1-\theta^{-1})(1-s)\eta \log n)\vee n^{-\beta}=n^{-\min\{\eta(1-\theta^{-1})(1-s),\beta\}}.
		\ee 
		In all other cases, the error term decays slower than polynomially. 
	\end{remark} 
	
	\begin{proof}[Proof of Proposition~\ref{prop:factmean} subject to Proposition~\ref{lemma:degprobasymp}]
		We closely follow the approach in~\cite[Proposition $2.1$]{AddEsl18}, where an analogue result is presented for the case $q_0=1$, i.e.\ for the random recursive tree. Set $K':=K-a_{i_n'}$ and for each $i_n\leq k\leq i_n'$ and each $u\in\N$ such that $\sum_{\ell=i_n}^{k-1}a_\ell<u\leq \sum_{\ell=i_n}^k a_\ell$, let $m_u=\lfloor \log_\theta n \rfloor +k$. We note that $m_u<\log_\theta n+i_n'<c\log n$, so that the results in Proposition~\ref{lemma:degprobasymp} can be used. Also, let $(v_u)_{u\in[K]}$ be $K$ vertices selected uniformly at random without replacement from $[n]$, and define $I:=[K]\backslash[K']$. Then, as the $X^{(n)}_{\geq k}$ and $X^{(n)}_{k}$ can be expressed as sums of indicators,
		\be \ba \label{eq:meanex}
		\hspace{-0.2cm}\mathbb E\bigg[\!\Big(X^{(n)}_{\geq i_n'}\Big)_{a_{i_n'}}\!\prod_{k=i_n}^{i_n'-1}\!\Big(X_k^{(n)}\Big)_{a_k}\!\bigg]\!&=(n)_K\P{\Zm_n(v_u)=m_u, \Zm_n(v_w)\geq m_w\text{ for all } u\in[K'],w\in I}\\
		&=(n)_K\sum_{\ell=0}^{K'}\sum_{\substack{S\subseteq [K']\\ |S|=\ell}}\!\!(-1)^\ell \P{\Zm_n(v_u)\geq m_u+\ind_{\{u\in S\}}\text{ for all } u\in [K]},
		\ea \ee
		where the second step follows from~\cite[Lemma $5.1$]{AddEsl18} and is based on an inclusion-exclusion argument. We can now use Proposition~\ref{lemma:degprobasymp}. First, we note that there exists a $\beta>0$ such that for non-negative integers $m_1,\ldots, m_K< c\log n$, 
		\be \label{eq:tailprob}
		\P{\Zm_n(v_u)\geq m_u+\ind_{\{u\in S\}}\text{ for all } u\in [K]}=\prod_{u=1}^K \mathbb E\bigg[\Big(\frac{W}{\E{W}+W}\Big)^{m_u+\ind_{\{u\in S\}}}\bigg]\big(1+o\big(n^{-\beta}\big)\big).
		\ee 
		Now, by Theorem~\ref{thrm:pkasymp} and the definition of $r_k$ in~\eqref{eq:ek} and as $r_k$ is decreasing by Lemma~\ref{lemma:rk} in the \hyperref[sec:appendix]{Appendix}, when $|S|=\ell$,
		\be\label{eq:pkprod}
		\prod_{u=1}^K \E{\Big(\frac{W}{\E{W}+W}\Big)^{m_u+\ind_{\{u\in S\}}}}=q_0^K\theta^{-\ell-\sum_{u=1}^K m_u}\big(1+\mathcal O\big(r_{\lfloor \log_\theta n\rfloor+i_n}\vee n^{-\beta'}\big)\big), 
		\ee  
		as the smallest $m_u$ equals $\lfloor \log_\theta n\rfloor+i_n$. We have 
		\be \ba \label{eq:finstep}
		(n)_K \sum_{\ell=0}^{K'}\sum_{\substack{S\subseteq [K']\\ |S|=\ell}}(-1)^\ell q_0^K \theta^{-\ell-\sum_{u=1}^K m_u}&=(n)_Kq_0^K \theta^{-K'-\sum_{u=1}^K m_u}\sum_{\ell=0}^{K'}\binom{K'}{\ell}(-1)^\ell \theta^{K'-\ell}\\
		&=(n)_Kq_0^K(1-\theta^{-1})^{K'}\theta^{-\sum_{u=1}^K m_u}.
		\ea \ee
		We then observe that $(n)_K=\theta^{K\log_\theta n}(1+\mathcal O(1/n))$. Moreover, we recall that $K=\sum_{k=i_n}^{i_n'}a_k$ and $K'=\sum_{k=i_n}^{i_n'-1}a_k$, while $m_u=\lfloor \log_\theta n\rfloor +k$ if $\sum_{\ell=i_n}^{k-1}a_\ell\leq u<\sum_{\ell=i_n}^k a_\ell$ for $i_n\leq \ell\leq i_n'$; in addition, recall $\eps_n$ from~\eqref{eq:peps}. Using~\eqref{eq:finstep} combined with~\eqref{eq:tailprob} and~\eqref{eq:pkprod} in~\eqref{eq:meanex}, we obtain 
		\be \ba
		\mathbb E\bigg[\!\Big(\!X^{(n)}_{\geq i_n'}\!\Big)_{a_{i_n'}}\!\prod_{k=i_n}^{i_n'-1}\!\Big(\!X_k^{(n)}\!\Big)_{a_k}\!\bigg]\!={}&	q_0^K(1-\theta^{-1})^{K'}\theta^{K\log_\theta n-\sum_{u=1}^K m_u}\big(1+\mathcal O\big(r_{\lfloor \log_\theta n\rfloor +i_n}\vee n^{-\beta}\big)\big)\\
		={}&\Big(q_0\theta^{-i_n'+\eps_n}\Big)^{a_{i_n'}}\prod_{k=i_n}^{i_n'-1}\Big(q_0(1-\theta^{-1})\theta^{-k+\eps_n}\Big)^{a_k}\\
		&\times\big(1+\mathcal O\big(r_{\lfloor \log_\theta n\rfloor +i_n}\vee n^{-\beta}\big)\big),
		\ea\ee  
		as desired.
	\end{proof}
	
	The next lemma builds on~\cite[Lemma $7.1$]{LodOrt21} and~\cite[Lemma $1$]{DevLu95} and provides bounds on the maximum degree that hold with high probability.
	
	\begin{lemma}\label{lemma:maxdegwhp}
		Let $W$ be a positive random variable that satisfies conditions~\ref{ass:weightsup} and~\ref{ass:weightzero} of Assumption~\ref{ass:weights}. Consider the WRT model in Definition~\ref{def:WRT} with vertex-weights $(W_i)_{\inn}$ which are i.i.d.\ copies of $W$. Fix $c\in(0,\theta/(\theta-1))$ and let $(k_n)_{n\in\N}$ be a non-negative, diverging integer sequence such that $k_n< c\log n$ and let $v_1$ be a vertex selected uniformly at random from $[n]$. If $\lim_{n\to\infty}n\P{\Zm_n(v_1)\geq k_n}=0$, then
		\be
		\lim_{n\to\infty}\P{\max_{\inn}\zni\geq k_n}=0.
		\ee 
		Similarly, when instead $\lim_{n\to\infty}n\P{\Zm_n(v_1)\geq k_n}=\infty$,
		\be
		\lim_{n\to\infty}\P{\max_{\inn}\zni \geq k_n}=1.
		\ee 
	\end{lemma}
	
	\begin{remark}
		Similar to what is discussed in Remark \ref{rem:degtail}$(i)$, the result in this lemma is stronger than the results presented in~\cite[Lemma $7.1$]{LodOrt21} and~\cite[Lemma $1$]{DevLu95}. It extends the latter to the WRT model rather than just the RRT model, and improves the former as the result holds for the non-conditional probability measure $\mathbb P$ rather than $\mathbb P_W$, which is what is used in~\cite{LodOrt21}. Due to the difficulties of working with the conditional probability measure, only a first order asymptotic result can be proved there. With the improved understanding of the degree distribution, as in Proposition~\ref{lemma:degprobasymp}, the above result can be obtained, which allows for finer asymptotics to be proved.
	\end{remark}
	
	\begin{proof}[Proof of Lemma~\ref{lemma:maxdegwhp} subject to Proposition~\ref{lemma:degprobasymp}]
		The first result immediately follows from a union bound and the fact that
		\be \label{eq:uarprob}
		n\P{\Zm_n(v_1)\geq k_n}=\sum_{i=1}^n \P{\zni\geq k_n}.
		\ee 
		For the second result, let $A_{n,i}:=\{\zni\geq k_n\},i\in[n]$. Then, by the Chung-Erd{\H o}s inequality (which is one way of formalizing the second moment method), 
		\be 
		\P{\max_{\inn}\zni \geq k_n}=\P{\cup_{i=1}^n A_{n,i}}\geq \frac{\Big(\sum_{i=1}^n \P{A_{n,i}}\Big)^2}{\sum_{i\neq j}\P{A_{n,i}\cap A_{n,j}}+\sum_{i=1}^n \P{A_{n,i}}}.
		\ee 
		By~\eqref{eq:uarprob} it follows that $\sum_{i=1}^n\P{A_{n,i}}=n\P{\Zm_n(v_1)\geq k_n}$. Furthermore, by Proposition~\ref{lemma:degprobasymp}, 
		\be 
		\sum_{i\neq j}\P{A_{n,i}\cap A_{n,j}}=n(n-1)\P{A_{n,v_1}\cap A_{n,v_2}}=(n\P{A_{n,v_1}})^2(1+o(1)),
		\ee
		where $v_2$ is another vertex selected uniformly at random, unequal to $v_1$. Note that the condition that $k_n<c\log n$ is required for this to hold. Together with the above lower bound, these two observations yield
		\be 
		\P{\cup_{i=1}^n A_{n,i}}\geq \frac{(n\P{A_{n,v_1}})^2}{(n\P{A_{n,v_1}})^2(1+o(1))+n\P{A_{n,v_1}}}=\frac{n\P{A_{n,v_1}}}{n\P{A_{n,v_1}}(1+o(1))+1}. 
		\ee 
		Hence, when $n\P{A_{n,v_1}}=n\P{\Zm_n(v_1)\geq k_n}$ diverges with $n$, we obtain the desired result.
	\end{proof}
	
	\subsection{Non-rigorous proof of Proposition~\ref{lemma:degprobasymp} when $k=1$}\label{sec:nonrigour}
	What remains for this section is to prove Proposition~\ref{lemma:degprobasymp}. As the proof is rather long and involved, we first provide a non-rigorous step-by-step proof of the case $k=1$. That is, we provide an asymptotic estimate for $\P{\Zm_n(v)=m}$, where $v$ is a vertex selected uniformly at random. Though this part is not essential for the proof of Proposition~\ref{lemma:degprobasymp}, it helps the reader in understanding the complete proof of the proposition later in this section. In the next sub-section, we discuss the strategy of the full proof and take care of certain parts of the proof in separate lemmas. Whilst proving these lemmas, we refer back to the simpler non-rigorous proof provided here for guidance.
		The main part of the proof is dedicated to proving~\eqref{eq:degdist}. Once this is established,~\eqref{eq:degtail} follows without much effort. We thus focus on discussing the proof of~\eqref{eq:degdist} here first. 
	
	To provide a non-rigorous proof of the asymptotic estimate of $\P{\Zm_n(v)=m}$, we first make the following simplification: we assume that $S_\ell=(\ell+1)\E W$ for all $\ell\in\N$. Naturally, the law of large numbers implies that $S_\ell/\ell\toas \E W$, so that, up to lower-order error terms, this equality holds. Using this simplification, however, allows us to omit many details and focus on the most important steps of the proof.
		
		Let us fix an $\eps\in(0,1)$, which we shall choose sufficiently small later. We then first condition  on the value of the uniform vertex $v$, and distinguish between its values being smaller or larger than~$n^\eps$. That is, 
		\be \ba\label{eq:vprobsplit}
		\P{\Zm_n(v)=m}&=\frac1n \sum_{j=1}^n \E{\Pf{\Zm_n(j)=m}}\\
		&=\frac1n \sum_{1\leq j< n^\eps}\E{\Pf{\Zm_n(j)=m}}+\frac1n \sum_{n^\eps \leq j\leq n}\E{\Pf{\Zm_n(j)=m}}.
		\ea \ee  
		We first take care of the second sum, and start by determining $\Pf{\Zm_n(j)=m}$ when $j\geq n^\eps$, under the simplification that $S_\ell=(\ell+1)\E W$. Since $\Zm_n(j)$ is a sum of independent indicator random variables, we will see that $\Zm_n(j)$ is well-approximated by a Poisson random variable, whose rate equals the sum of the success probabilities of the indicator random variables. By the simplification $\cS_\ell=(\ell+1)\E W$, we can thus determine that 
		\be 
		\Zm_n(j)=\sum_{i=j+1}^n \ind_{\{i\to j\}}=\sum_{i=j+1}^n \text{Ber}\Big(\frac{W_j}{S_{i-1}}\Big)\approx \text{Poi}\Big(\sum_{j=i+1}^n \frac{W_j}{i\E W}\Big)\approx \text{Poi}\Big(\frac{W_j}{\E W}\log(n/j)\Big)=:P_j.
		\ee
		It thus follows that $\Pf{\Zm_n(j)=m}\approx \Pf{P_j=m}$. We now provide the necessary steps to obtain this intuitive result. The degree of vertex $j$ in $T_n$ equals exactly $m$ when there exist vertices $i_1,\ldots ,i_m$ that connect to $j$, and all other vertices do not. Since these connections are independent,
		\be \ba\label{eq:probexplicit}
		\Pf{\Zm_n(j)=m}&=\sum_{j<i_1<\ldots<i_m\leq n}\prod_{s=1}^m \frac{W_j}{S_{i_s-1}}\prod_{\substack{s=j+1\\ s\neq i_\ell, \ell \in[m]}}^n \Big(1-\frac{W_j}{S_{s-1}}\Big)\\
		&=\sum_{j<i_1<\ldots<i_m\leq n}\prod_{s=1}^m \frac{W_j}{i_s\E W}\prod_{\substack{s=j+1\\ s\neq i_\ell, \ell \in[m]}}^n \Big(1-\frac{W_j}{s\E W}\Big)
		\ea \ee 
		Here, we sum over all possible choices of vertices $i_1<\ldots <i_m$ that connect to $j$, and the final step follows from the simplifying assumption. We now include the terms $s \in \{i_1, \ldots, i_m\}$
		in the second product, by changing the denominator in the fraction in the first product to $i_s\E W- W_j$. This yields
		\be \label{eq:connectsplit}
		\Pf{\Zm_n(j)=m}=\sum_{j<i_1<\ldots<i_m\leq n}\prod_{s=1}^m \frac{W_j}{i_s\E W-W_j}\prod_{s=j+1}^n \Big(1-\frac{W_j}{s\E W}\Big).
		\ee 
		Since $W_j\in(0,1]$ almost surely, $i_s>j\geq n^\eps$, and $m<c\log n$, it follows that for any $\xi\in(0,1)$, 
		\be\label{eq:prodasymp} 
		\prod_{s=1}^m \frac{1}{i_s\E W-W_j}=\prod_{s=1}^m \frac{1}{i_s\E W(1+\cO(n^{-\eps}))}=(1+\cO(n^{-(1-\xi)\eps}))\prod_{s=1}^m \frac{1}{i_s\E W}.
		\ee
		Similarly, 
		\be
		\prod_{s=j+1}^n \Big(1-\frac{W_j}{s\E W}\Big)=\exp\Big(\sum_{s=j+1}^n \log\Big(1-\frac{W_j}{s\E W}\Big)\Big)=\exp\Big(-(1+\cO(n^{-\eps}))\sum_{s=j+1}^n \frac{W_j}{s\E W}\Big).
		\ee 
		As  
		\be 
		\sum_{s=j+1}^n \frac1s\approx \log(n/j),
		\ee
		we arrive at
		\be \label{eq:nonconnectasymp}
		\prod_{s=j+1}^n \Big(1-\frac{W_j}{s\E W}\Big)\approx \Big(\frac jn\Big)^{(1+\zeta_n)W_j/\E W}, 
		\ee 
		where $\zeta_n=\cO(n^{-\eps})$. Combined with~\eqref{eq:prodasymp} in~\eqref{eq:connectsplit}, this yields
		\be 
		\Pf{\Zm_n(j)=m}=(1+\cO(n^{-(1-\xi)\eps}))\Big(\frac jn\Big)^{(1+\zeta_n)W_j/\E W}\Big( \frac{W_j}{\E W}\Big)^m\sum_{j<i_1<\ldots<i_m\leq n}\prod_{s=1}^m \frac{1}{i_s}.
		\ee 
		We can approximate the sum by the multiple integrals 
		\be 
		\int_j^n \int_{x_1}^n \cdots \int_{x_{m-1}}^n \prod_{s=1}^m x_s^{-1}\,\d x_m\ldots \d x_1. 
		\ee 
		By Lemma~\ref{lemma:logints} we obtain that this equals $(\log(n/j))^m/m!$. So, we finally have 
		\be \ba
		\Pf{\Zm_n(j)=m}&\approx \Big(\frac jn\Big)^{W_j/\E W}\Big(\frac{W_j}{\E W}\log(n/j)\Big)^m\frac{1}{m!}\\
		&=\exp\Big(-\frac{W_j}{\E W}\log(n/j)\Big)\Big(\frac{W_j}{\E W}\log(n/j)\Big)^m\frac{1}{m!}\\
		&=\Pf{P_j=m},
		\ea \ee 
		as desired. Since all vertex-weights are i.i.d., when we take an expectation over the weights, we can in fact omit the index~$j$ from $W_j$ and use a general random variable $W$ instead, so that $P_j$ has mean $\log(n/j)W/\E W$. Using this expression in the second sum in~\eqref{eq:vprobsplit} approximately yields
		\be \label{eq:poiprob}
		\frac{1}{n} \sum_{n^\eps \leq j\leq n}\E{ \Pf{P_j=m}}.
		\ee 
		To avoid confusion between different parametrisations, we say that a random variable $X$ has a Gamma$(\alpha,\beta)$ distribution for some $\alpha,\beta>0$ when it has a probability density function $f:\R_+\to \R_+$ with 
		\be \label{eq:gammapdf}
		f(x)=\frac{\beta^\alpha}{\Gamma(\alpha)}x^{\alpha-1}\e^{-\beta x},\qquad x>0.
		\ee 
		We now use the following duality between Poisson and gamma random variables. Let \\$G\sim \text{Gamma}(m,1)$ be a gamma random variable for some integer $m$. Note that we can also interpret $G$ as a sum of $m$ independent rate one exponential random variables. Then, conditionally on $W$, the event $\{P_j=m\}$ can be thought of as the event that in a rate one Poisson process exactly $m$ particles have arrived before time $\log(n/j)W/\E W$. This is equivalent to the sum of the first $m$ inter-arrival times (which are rate one exponentially distributed) being at most $\log(n/j)W/\E W$, and the sum of the first $m+1$ inter-arrival times exceeding this quantity. As we mentioned, this sum of $m$ rate one exponential random variables is, in law, identical to $G$. So, if we let $V$ be the $m+1^{\text{st}}$ inter-arrival time, independent of $G$, then
		\be\ba
		\Pf{P_j=m}&=\Pf{G\leq \log(n/j)W/\E W, G+V>\log(n/j)W/\E W}\\
		&=\Pf{Y\leq \log(n/j), Y+\wt V>\log(n/j)}.
		\ea\ee 
		where, conditionally on $W$, we have $Y\sim \text{Gamma}(m,W/\E W)$ and $ \wt V\sim \text{Exp}(W/\E W)$. In~\eqref{eq:poiprob}, this yields
		\be 
		\frac1n \E{\sum_{n^\eps\leq j\leq n}\Pf{Y\leq \log(n/j),Y+\wt V>\log(n/j)}}\approx \E{\Pf{Y\leq \log(n/v),Y+\wt V>\log (n/v)}}, 
		\ee 
		where we recall that $v$ is a uniform element of $[n]$. Observe that $\log(n/v)\approx T$, where $T\sim \text{Exp}(1)$. As a result, the conditional probability can be expressed, using that $Y$ can be viewed as a sum of $m$ i.i.d.\ copies of $\wt V$ and the memoryless property of exponential random variables, as 
		\be 
		\E{\Pf{Y\leq T, Y+\wt V\geq T} }=\E{\frac{\E W}{W+\E W}\Big(\frac{W}{W+\E W}\Big)^m},
		\ee 
		as desired. We provide some more details for this derivation. First, we write the sum as 
		\be \label{eq:gammadif}
		\frac{1}{n}\E{ \sum_{n^\eps \leq j\leq n}\Pf{Y\leq \log(n/j)}-\Pf{Y+\wt V\leq \log(n/j)} }.
		\ee 
		Since $Y+\wt V\sim \text{Gamma}(m+1,W/\E W)$, it suffices to deal with the first conditional probability only. We can approximate the sum by
		\be  \ba
		\sum_{n^\eps \leq j\leq n}\Pf{Y\leq \log(n/j)}&\approx \int_{n^\eps}^n \Pf{Y\leq \log(n/x)}\,\d x=n\int_0^{(1-\eps)\log n}\e^{-y}\Pf{Y\leq y}\,\d y.
		\ea\ee  
		The second step follows from a variable transformation $y=\log(n/x)$. We write $\Pf{Y\leq y}=\int_0^y f_{Y|W}(x)\d x$, where $f_{Y|W}(x)$ is as in~\eqref{eq:gammapdf} with $\alpha=m-1, \beta=W/\E W$, the conditional probability density function of $Y$, conditionally on $W$. Changing the order of integration  yields
		\be\ba 
 		n{}&\int_0^{(1-\eps)\log n}\e^{-x} f_{Y|W}(x)\,\d x-n^\eps \Pf{Y\leq (1-\eps)\log n}\\
		&=n\Big(\frac{W}{\E W+W}\Big)^m \Pf{Y'\leq (1-\eps)\log n}-n^\eps \Pf{Y\leq (1-\eps)\log n}, 
		\ea \ee 
		where $Y'\sim \text{Gamma}(m,1+W/\E W)$, conditionally on $W$. A similar result with $m+1$ and $Y''\sim\text{Gamma}(m+1,1+W/\E W)$ instead of $m$ and $Y'$ follows for the probability of the event $\{Y+\wt V\leq \log(n/j)\}$. We remark that, via the construction of the random variables $Y'$ and $Y''$ using $Y$ and $Y+\wt V$, respectively, they are defined on the same probability space and that in fact $Y''$ stochastically dominates $Y'$. As such, combining both in~\eqref{eq:gammadif} approximately provides
			\be \ba 
			\mathbb E\bigg[\Big(\frac{W}{\E W+W}{}&\Big)^m \Big(\Pf{Y'\leq (1-\eps)\log n}-\frac{W}{\E W+W}\Pf{Y''\leq (1-\eps)\log n}\Big)\bigg]\\
			={}&(1+\cO(n^{-(1-\xi)\eps}))\E{\frac{\E W}{\E W+W}\Big(\frac{W}{\E W+W}\Big)^m\Pf{Y'\leq (1-\eps)\log n}}\\
			&+\E{\Big(\frac{W}{\E W+W}\Big)^m\frac{W}{\E W+W}\Pf{Y'\leq (1-\eps)\log n, Y''>(1-\eps)\log n}}.
			\ea \ee 
			Here, the second step follows from the fact that $Y'\preceq Y''$. We finally show that this quantity equals 
		\be 
		\E{\frac{\E W}{\E W+W}\Big(\frac{W}{\E W+W}\Big)^m}(1+o(n^{-\beta})),
		\ee 
		for some small $\beta>0$, when $\eps$ is sufficiently small, by proving that $\Pf{Y'\leq (1-\eps)\log n}$ and $\Pf{Y''>(1-\eps)\log n}$ are sufficiently close to one and zero, respectively. We thus conclude that 
		\be \ba \label{eq:1sttermasymp}
		\frac 1n \sum_{n^\eps \leq j\leq n}\E{\Pf{\Zm_n(j)=m}}&=\E{\frac{\E W}{\E W+W}\Big(\frac{W}{\E W+W}\Big)^m}(1+o(n^{-\beta}))\\
		&=p_m(1+o(n^{-\beta})),
		\ea \ee 
		where $\beta$ and $\eps$ are sufficiently small and where we recall $p_m$ from~\eqref{eq:pk}. 
		
		It remains to show that the first term on the right-hand side of~\eqref{eq:vprobsplit} can be included in the $o(n^{-\beta})$ term. The main difficulty here for the actual proof is controlling the partial sums $S_\ell$, which, for `too small' $\ell$, do not concentrate around the expected value with probability sufficiently close to one. In this simplified non-rigorous proof, however, the assumption that $S_\ell=(\ell+1)\E W$ allows us to ignore this technical difficulty for now and focus on the main probabilistic arguments. The aim is to show that
		\be \label{eq:2ndtermbound}
		\frac 1n \sum_{1\leq j<n^\eps}\E{\Pf{\Zm_n(j)=m}}=o(p_mn^{-\beta}), 
		\ee  
		for some small $\beta$ when $\eps$ is sufficiently small. 
		
		We first focus on the case that $m=\wt c\log n+o(\log n)$ with $\wt c<1/\log \theta$. In this case, we can simply bound the conditional probability from above by one to obtain the upper bound $n^{-(1-\eps)}$. Then, by Lemma~\ref{lemma:pkbound}, we  have that $p_m\geq (\theta+\xi)^{-m}=n^{-\wt c\log(\theta+\xi)}$ for any $\xi>0$. Since $\wt c<1/\log \theta$, it follows that we can choose $\xi, \beta$, and $\eps$ sufficiently small such that $n^{-(1-\eps)}=o(n^{-\beta-\wt c\log(\theta+\xi)})=o(p_mn^{-\beta})$, from which the claim in~\eqref{eq:2ndtermbound} follows. 
		
		When $\wt c\in[1/\log \theta,c)$ instead (where we recall that $c<\theta/(\theta-1)$), a more careful approach is required. 
		In this case, we bound $\Pf{\Zm_n(j)=m}\leq \Pf{\Zm_n(j)\geq m}$ and use a Chernoff bound on the right-hand side of the inequality. This bound follows from the proof of~\cite[Proposition $7.2$]{LodOrt21} and yields
		\be \label{eq:chernbound}
		\frac1n \sum_{1\leq j<n^\eps}\E{\Pf{\Zm_n(j)=m}}\leq \frac1n \sum_{1\leq j<n^\eps}\E{\exp(m(1-u_j+\log u_j))}, 
		\ee 
		where 
		\be
		u_j:=\frac{1}{m}\sum_{\ell=j}^{n-1}\frac{W_j}{S_\ell}=\frac1m\sum_{\ell=j+1}^n\frac{W_j}{\ell\E W}\leq \frac{\log(n/j)}{m\E W}\leq \frac{\log n}{m\E W}=\frac{1+o(1)}{\wt c \E W}.
		\ee 
		Note that the second step follows from our simplifying assumption on $S_\ell$ and we bound $W_j$ from above by one in the inequality. Using that $x\mapsto 1-x+\log x$ is increasing on $(0,1)$, that $1/(\wt c\E W)\leq \log \theta/(\theta-1)<1$, and that we have a bound for $u_j$ uniformly in $1\leq j\leq n^\eps$, \eqref{eq:chernbound} yields
		\be \ba \label{eq:uiluse}
		\frac1n \sum_{1\leq j<n^\eps}\!\!\!\!\E{\Pf{\Zm_n(j)=m}}&\leq \frac{1}{n^{1-\eps}}\exp\Big((1+o(1))\wt c\log n \Big(1-\frac{1}{\wt c(\theta-1)}+\log\Big(\frac{1}{\wt c (\theta-1)}\Big)\Big)\Big)\\
		&=  n^{-(1-\eps)+(1+o(1))\wt c(1-1/(\wt c(\theta-1))+\log(1/(\wt c (\theta-1))))}.
		\ea \ee 
		The aim is to show that this last expression is $o(p_mn^{-\beta})$ for some sufficiently small $\beta>0$. Again using that $p_mn^{-\beta}\geq n^{-\wt c\log(\theta+\xi)-\beta}$, it suffices to show that 
		\be 
		-(1-\eps)+\wt c\Big(1-\frac{1}{\wt c(\theta-1)}+\log\Big(\frac{1}{\wt c (\theta-1)}\Big)\Big)<-\wt c\log(\theta+\xi)-\beta
		\ee 
		holds when we choose $\xi,\eps,$ and $\beta$ sufficiently small. Since the right-hand side is continuous in $\xi$ and $\beta$, and the left-hand side is continuous in $\eps$, it suffices to prove that 
		\be 
		-1+\wt c\Big(1-\frac{1}{\wt c(\theta-1)}+\log\Big(\frac{1}{\wt c (\theta-1)}\Big)\Big)<-\wt c\log\theta
		\ee 
		holds for any $\theta\in(1,2]$ and any $\wt c<\theta/(\theta-1)$. As the second derivative of the mapping $x\mapsto -1+x-1/(\theta-1)+x\log(1/(x(\theta-1)))$ is negative and the mapping $x\mapsto -x\log \theta$ is tangent to the former mapping at $x=\theta/(\theta-1)$, the inequality follows for all $\wt c<\theta/(\theta-1)$. This completes the proof of~\eqref{eq:2ndtermbound} for some sufficiently small $\eps$ and $\beta$. Combined with~\eqref{eq:1sttermasymp}, this yields~\eqref{eq:degdist} for $k=1$. 
	
	\subsection{Complete proof of Proposition~\ref{lemma:degprobasymp}}
	
	Before we provide the complete proof, we first discuss the proof strategy. Though this is similar to the approach of the non-rigorous proof provided in the previous sub-section, dealing with $k>1$ many uniformly selected vertices provides technical challenges worth addressing. Let us start by introducing the following notation: for two sequences $f(n)$ and $g(n)$, we let $f(n)\leq j_1\neq \ldots \neq j_k\leq g(n)$ denote the set 
		\be\label{eq:setnot}
		\{(j_1,\ldots, j_k)\in\N^k: j_\ell\in[f(n),g(n)]\text{ for all }\ell\in[k], \text{ and }j_{\ell_1}\neq j_{\ell_2}\text{ for any }1\leq \ell_1<\ell_2\leq k\}.
		\ee 
		In words, summing over all indices $f(n)\leq j_1\neq \ldots \neq j_k\leq g(n)$ denotes summing over all distinct indices $j_1,\ldots, j_k$ which are at least $f(n)$ and at most $g(n)$.
	
	The left-hand side of~\eqref{eq:degdist} can be expressed by conditioning on the values of the typical vertices, and splitting between cases of young and old vertices. That is, 
	\be\ba\label{eq:splitsum2}
	\P{\Zm_n(v_\ell)=m_\ell\text{ for all }\ell\in[k]}={}&\frac{1}{(n)_k}\sum_{1\leq j_1\neq \ldots \neq j_k\leq n}\!\!\!\P{\Zm_n(j_\ell)=m_\ell\text{ for all }\ell\in[k]}\\
	={}&\frac{1}{(n)_k}\sum_{n^{ \eps}\leq j_1\neq \ldots \neq j_k\leq n}\!\!\!\!\!\!\!\!\!\!\!\!\!\P{\Zm_n(j_\ell)=m_\ell\text{ for all }\ell\in[k]}\\
	&+\frac{1}{(n)_k}\sum_{\textbf{j}\in I_n(\eps)}\!\!\P{\Zm_n(j_\ell)=m_\ell\text{ for all }\ell\in[k]},
	\ea \ee 
	where
	\be\label{eq:ineps}
	I_n(\eps):=\{\textbf{j}=(j_1,\ldots,j_k):1\leq j_1\neq\ldots\neq j_k\leq n,\ \exists i\in[k]\ j_i<n^{ \eps}\}
	\ee 
	for any $\eps\in(0,1)$. Splitting the sum on the first line into the two sums on the second and third line allows us to deal with them in a different way, and is similar to the distinction made in~\eqref{eq:vprobsplit}. In the sum on the second line, in which all indices are at least $n^{ \eps}$, we can apply the law of large numbers on sums of vertex-weights to gain more control over the conditional probability of the event $\{\Zm_n(j_\ell)=m_\ell\text{ for all }\ell\in[k]\}$. The aim is to show that this first sum has the desired form, as on the right-hand side of~\eqref{eq:degdist}. This uses the same ideas as in the non-rigorous proof in the previous sub-section, but increases in difficulty due to multiple summations, integrals, and careful `book-keeping' of a large number of distinct indices.
	
	The sum on the third line, in which at least one of the indices takes on values strictly smaller than $n^{\eps}$ can be shown to be negligible compared to the first sum. Especially when $m_\ell$ is large, this is non-trivial. To do this, we consider the tail events $\{\Zm_n(j_\ell)\geq m_\ell\text{ for all }\ell\in[k]\}$ and use the negative quadrant dependence of the degrees (see Remark~\ref{rem:degtail} and~\cite[Lemma $7.1$]{LodOrt21}), so that we can deal with the more tractable probabilities $\Pf{\Zm_n(j_\ell)\geq m_\ell}$ for all $\ell\in[k]$, rather than the probability of the intersection of all tail degree events. We observe that this is not required when $k=1$ (that is, using the negative quadrant dependence), and significantly increases the technical difficulty of this part of the proof compared to the non-rigorous proof provided in the previous sub-section. Depending on whether the indices in $I_n(\eps)$ are at most or at least $n^{\eps}$, we then use bounds similar to one developed in the proof of~\cite[Lemma $7.1$]{LodOrt21} or use an approach similar to what we use to bound the sum on the second line of~\eqref{eq:splitsum2}, respectively.
	
	In the following lemma, we deal with the sum on the second line of~\eqref{eq:splitsum2}.
	
	\begin{lemma}\label{lemma:in0eps}
		Let $W$ be a positive random variable that satisfies condition~\ref{ass:weightsup} of Assumption~\ref{ass:weights}. Consider the WRT model in Definition~\ref{def:WRT} with vertex-weights $(W_i)_{\inn}$ which are i.i.d.\ copies of $W$ and fix $k\in\N$ and $c\in (0,\theta/(\theta-1))$. Then, there exist $\beta>0$ and an $\eps\in(0,1)$ that can be made arbitrarily small, such that uniformly over non-negative integers $m_\ell< c\log n, \ell\in[k]$,
		\be 
		\frac{1}{(n)_k}\sum_{n^{\eps}\leq j_1\neq \ldots \neq j_k\leq n}\!\!\!\!\!\!\!\!\!\!\!\!\P{\Zm_n(j_\ell)=m_\ell\text{ for all } \ell\in[k]}=\prod_{\ell=1}^k \mathbb E\bigg[\frac{\E W}{\E W+W}\Big(\frac{W}{\E W+W}\Big)^{m_\ell}\bigg]\big(1+o\big(n^{-\beta}\big)\big).
		\ee 
	\end{lemma} 
	
	We note that condition~\ref{ass:weightzero} of Assumption~\ref{ass:weights} is not required for this result to hold. To prove this lemma, we sum over all possible $m_\ell$ vertices that connect to $j_\ell$ for each $\ell\in[k]$ and use the fact that the $j_1,\ldots, j_k$ are at least $n^{\eps}$ to precisely control the connection probabilities and to evaluate the sums over all the possible $m_\ell$ vertices for all $\ell\in[k]$, as well as the sum over the indices $j_1,\ldots, j_k$.
	
	In the following lemma, we show the sum on the third line of~\eqref{eq:splitsum2} is negligible compared to the sum on the second line.
	
	\begin{lemma}\label{lemma:inepsterm}
		Let $W$ be a positive random variable that satisfies conditions~\ref{ass:weightsup} and~\ref{ass:weightzero} of Assumption~\ref{ass:weights}. Consider the WRT model in Definition~\ref{def:WRT} with vertex-weights $(W_i)_{\inn}$ which are i.i.d.\ copies of $W$. Fix $k\in\N,c\in (0,\theta/(\theta-1))$ and recall $I_n(\eps)$ from~\eqref{eq:ineps}. There exist $\eps\in(0,1)$ and a $\beta>0$ such that uniformly over non-negative integers $m_\ell< c\log n, \ell\in[k]$,
		\be 
		\frac{1}{(n)_k}\sum_{\textbf{j}\in I_n(\eps)}\!\!\P{\Zm_n(j_\ell)=m_\ell\text{ for all }\ell\in[k]}=o\bigg(\prod_{\ell=1}^k \E{\frac{\E W}{\E W+W}\Big(\frac{W}{\E W+W}\Big)^{m_\ell}} n^{-\beta}\bigg).
		\ee
	\end{lemma} 
	
	Note that condition~\ref{ass:weightzero} of Assumption~\ref{ass:weights} is required only in this lemma, where it was not necessary in Lemma~\ref{lemma:in0eps}. In fact, it is required for one inequality in the proof only, which convinces us that it could possibly be avoided.
	
	It is clear that~\eqref{eq:degdist} in Proposition~\ref{lemma:degprobasymp} immediately follows from using the results of Lemmas~\ref{lemma:in0eps} and~\ref{lemma:inepsterm} in~\eqref{eq:splitsum2}. Namely, in Lemma~\ref{lemma:in0eps} we can take $\eps $ arbitrarily small and the left-hand side of the equality in Lemma~\ref{lemma:inepsterm} is monotone decreasing as $\eps$ decreases, so that we can take $\eps$ as small as required. In what follows we first prove Lemma~\ref{lemma:in0eps} in Section~\ref{sec:proof5.10}, prove Lemma~\ref{lemma:inepsterm} in Section~\ref{sec:proof5.11}, and finally complete the proof of Proposition~\ref{lemma:degprobasymp} in Section~\ref{sec:proof5.1}. After each of the proofs we discuss the required adaptations to prove the results of Lemmas~\ref{lemma:in0eps} and~\ref{lemma:inepsterm} for the model with \emph{random out-degree}, as introduced in Remark~\ref{remark:def}$(ii)$.
	
	\subsection{Proof of Lemma~\ref{lemma:in0eps}\label{sec:proof5.10}}
	
	\begin{proof}[Proof of Lemma~\ref{lemma:in0eps}]
		We provide a matching upper bound and lower bound for
		\be
		\frac{1}{(n)_k}\sum_{n^{\eps}\leq j_1\neq \ldots \neq j_k\leq n}\!\!\!\!\!\!\!\!\!\!\!\!\P{\Zm_n(j_\ell)=m_\ell\text{ for all } \ell\in[k]}.
		\ee
		\paragraph{\textbf{Upper bound.}} Let  $\zeta_n=n^{-\delta\eps}/\E{W}$ for some $\delta\in(0,1/2)$. We define the event
		\be \label{eq:en}
		E^{(1)}_n:=\bigg\{ \sum_{\ell=1}^j W_\ell \in ((1-\zeta_n)\E{W}j,(1+\zeta_{n})\E{W}j),\text{ for all } n^{\eps}\leq j\leq n\bigg\},
		\ee 
		We use $E^{(1)}_n$ to mimic the simplified assumption that $S_j=(j+1)\E W$ in the non-rigorous proof in Section~\ref{sec:nonrigour}. We know that $\mathbb P((E^{(1)}_n)^c)=o(n^{-\gamma})$ for any $\gamma>0$ (and thus of smaller order than $\prod_{\ell=1}^k p_{m_\ell}n^{-\beta}$ for any $\beta>0$ and uniformly in $m_1,\ldots m_k<(\theta/(\theta-1))\log n$) from Lemma~\ref{lemma:weightsumbounds} in the~\hyperref[sec:appendix]{Appendix}. By also conditioning on the vertex-weights, this yields for any $\gamma>0$ the upper bound
		\be \ba \label{eq:enbound}
		\frac{1}{(n)_k}{}&\sum_{n^{\eps}\leq j_1\neq \ldots \neq j_k\leq n}\E{\Pf{\Zm_n(j_\ell)=m_\ell\text{ for all }\ell\in[k]}}\\
		\leq {}&\frac{1}{(n)_k}\sum_{n^{\eps}\leq j_1\neq \ldots \neq j_k\leq n}\!\!\!\!\!\!\!\!\!\!\!\!\!\!\mathbb E[\Pf{\Zm_n(j_\ell)=m_\ell\text{ for all }\ell\in[k]}\ind_{E_n^{(1)}}]+o(n^{-\gamma}),
		\ea \ee 
		Now, to express the first term on the right-hand side of~\eqref{eq:enbound}, we consider ordered indices rather than unordered ones. We provide details for the case $n^{\eps}\leq j_1<j_2<\ldots<j_k\leq n$ and discuss later on how the other permutations of $j_1,\ldots,j_k$ can be dealt with. Moreover, for every $\ell\in[k]$, we introduce the ordered indices $j_{\ell}<i_{1,\ell}<\ldots<i_{m_\ell,\ell}\leq n,\ell\in[k]$, which denote the steps at which vertex $\ell$ increases it degree by one. Note that for every $\ell\in[k]$ these indices are distinct by definition, but we also require that $i_{s,\ell}\neq i_{t,j}$ for any $\ell,j\in[k],s\in[m_\ell]$ and $t\in[m_j]$ (equality is allowed only when $\ell=j$ and $s=t$). Indeed, a new vertex can only connect to one already present vertex. We denote this constraint by adding a $*$ on the summation symbol. Finally, we define $j_{k+1}:=n$. Combining these additional steps, we arrive at
		\be\ba \label{eq:ubfirst}
		\frac{1}{(n)_k}{}&\sum_{n^{\eps}\leq j_1< \ldots< j_k\leq n}\E{\Pf{\Zm_n(j_\ell)=m_\ell\text{ for all } \ell\in[k]}\ind_{E^{(1)}_n}}\\
		={}&\frac{1}{(n)_k}\sum_{n^{\eps}\leq j_1<\ldots<j_k\leq n}\ \, \sideset{}{^*}\sum_{\substack{j_\ell<i_{1,\ell}<\ldots<i_{m_\ell,\ell}\leq n,\\\ell\in[k]}}\mathbb E\Bigg[\prod_{t=1}^k\prod_{s=1}^{m_t}\frac{W_{j_t}}{\sum_{\ell=1}^{i_{s,t}-1}W_\ell}\\ &\times \prod_{u=1}^k\!\!\prod_{\substack{s=j_u+1\\s\neq i_{\ell,t},\ell\in[m_t],t\in[k]}}^{j_{u+1}}\!\!\!\!\bigg(1-\frac{\sum_{\ell=1}^u W_{j_\ell}}{\sum_{\ell=1}^{s-1}W_{\ell}}\bigg)\ind_{E^{(1)}_n}\Bigg].
		\ea \ee   
		This step is equivalent to~\eqref{eq:probexplicit} in the non-rigorous proof. The terms in the first double product denote the probabilities that vertices $i_{s,t}$ connect to $j_t$, whereas the terms in the second double product denote the probabilities that vertices $i_{s,t}$ do \emph{not} connect to the vertices $j_\ell$ such that $j_\ell<i_{s,t}$. Similar to~\eqref{eq:connectsplit} in the non-rigorous proof, we then include the terms where $s=i_{\ell,t}$ for all $\ell\in[m_t]$ and $t\in[k]$ in the second double product. To do this, we change the first double product to
		\be \label{eq:fracbound}
		\prod_{t=1}^k \prod_{s=1}^{m_t}\frac{W_{j_t}}{\sum_{\ell=1}^{i_{s,t}-1}W_\ell-\sum_{\ell=1}^k W_{j_\ell}\ind_{\{i_{s,t}>j_\ell\}}}\leq\prod_{t=1}^k \prod_{s=1}^{m_t} \frac{W_{j_t}}{\sum_{\ell=1}^{i_{s,t}-1}W_\ell-k},
		\ee 
		that is, we subtract the vertex-weight $W_{j_\ell}$ in the numerator when the vertex  $j_\ell$ has already been introduced by step $i_{s,t}$. In the upper bound we use that the weights are bounded from above by one. We thus arrive at the upper bound		
		\be\label{eq:tailprobmiddle}
		\frac{1}{(n)_k}\sum_{n^{\eps}\leq j_1<\ldots<j_k\leq n}\ \, \sideset{}{^*}\sum_{\substack{j_\ell<i_{1,\ell}<\ldots<i_{m_\ell,\ell}\leq n,\\\ell\in[k]}}\!\!\!\mathbb E\Bigg[\prod_{t=1}^k\prod_{s=1}^{m_t}\frac{W_{j_t}}{\sum_{\ell=1}^{i_{s,t}-1}W_\ell-k}\prod_{u=1}^k\prod_{s=j_u+1}^{j_{u+1}}\bigg(1-\frac{\sum_{\ell=1}^u W_{j_\ell}}{\sum_{\ell=1}^{s-1}W_\ell}\bigg)\ind_{E^{(1)}_n}\Bigg].
		\ee  
		For ease of writing, we omit the first sum until we actually intend to sum over the indices $j_1,\ldots,j_k$. We use the bounds from the event $E^{(1)}_n$ to bound 
		\be
		\sum_{\ell=1}^{i_{s,t}-1}W_\ell\geq (i_{s,t}-1)\E W(1-\zeta_n),\qquad \sum_{\ell=1}^{s-1}W_\ell \leq s\E W (1+\zeta_n).
		\ee
		For $n$ sufficiently large, we observe that $(i_{s,t}-1)\E W(1-\zeta_n)-k\geq i_{s,t}\E W(1-2\zeta_n)$, which yields 
		\be  
		\frac{1}{(n)_k}\  \sideset{}{^*}\sum_{\substack{j_\ell<i_{1,\ell}<\ldots<i_{m_\ell,\ell}\leq n,\\\ell\in[k]}}\!\!\!\!\mathbb E\Bigg[\prod_{t=1}^k\prod_{s=1}^{m_t}\frac{W_{j_t}}{i_{s,t}\E W (1-2\zeta_n)} \prod_{u=1}^k\prod_{s=j_u+1}^{j_{u+1}}\!\!\bigg(1-\frac{\sum_{\ell=1}^u W_{j_\ell}}{s\E W(1+\zeta_n)}\bigg)\ind_{E^{(1)}_n}\Bigg].
		\ee
		Moreover, relabelling the vertex-weights $W_{j_t}$ to $W_t$ for $t\in[k]$ does not change the distribution of the terms within the expected value, so that the expected value remains unchanged. We can also bound the indicator from above by one, to arrive at the upper bound
		\be 
		\frac{1}{(n)_k}\  \sideset{}{^*}\sum_{\substack{j_\ell<i_{1,\ell}<\ldots<i_{m_\ell,\ell}\leq n,\\\ell\in[k]}}\mathbb E\Bigg[\prod_{t=1}^k\prod_{s=1}^{m_t}\frac{W_t}{i_{s,t}\E W (1-2\zeta_n)} \prod_{u=1}^k\prod_{s=j_u+1}^{j_{u+1}}\bigg(1-\frac{\sum_{\ell=1}^u W_\ell}{s\E W(1+\zeta_n)}\bigg)\Bigg].
		\ee 
		We bound the final product from above by
		\be \ba \label{eq:expbound}
		\prod_{s=j_u+1}^{j_{u+1}}\bigg(1-\frac{\sum_{\ell=1}^u W_\ell}{s\E W(1+\zeta_n)}\bigg)&\leq \exp\bigg(-\frac{1}{\E W(1+\zeta_n)} \sum_{s=j_u+1}^{j_{u+1}}\frac{\sum_{\ell=1}^u W_\ell}{s}\bigg)\\
		&\leq \exp\bigg(-\frac{1}{\E W(1+\zeta_n)} \sum_{\ell=1}^u W_\ell \log\Big(\frac{j_{u+1}}{j_u+1}\Big)\bigg)\\
		&= \Big(\frac{j_{u+1}}{j_u+1}\Big)^{-\sum_{\ell=1}^u W_\ell/(\E W(1+\zeta_n))}.
		\ea\ee 
		As the weights are almost surely bounded by one, we thus find
		\be
		\prod_{s=j_u+1}^{j_{u+1}}\bigg(1-\frac{\sum_{\ell=1}^u W_\ell}{s\E W(1+\zeta_n)}\bigg)\leq \Big(\frac{j_{u+1}}{j_u}\Big)^{-\sum_{\ell=1}^u W_\ell/(\E W(1+\zeta_n))}\Big(1+\mathcal O\big(n^{-\eps}\big)\Big),
		\ee 
		which is equivalent to~\eqref{eq:nonconnectasymp} in the non-rigorous proof.
		As a result, we obtain the upper bound
		\be\ba 
		\frac{1}{(n)_k}\, \sideset{}{^*}\sum_{\substack{j_\ell<i_{1,\ell}<\ldots<i_{m_\ell,\ell}\leq n,\\\ell\in[k]}}\!\!\!\!\!{}&\mathbb E\Bigg[\prod_{t=1}^k\! \bigg(\Big(\frac{W_t}{\E{W}}\Big)^{m_t}\prod_{s=1}^{m_t}\frac{1}{i_{s,t}(1-2\zeta_n)}\bigg)\!\prod_{u=1}^k\!\Big(\frac{j_{u+1}}{j_u}\Big)^{-\sum_{\ell=1}^u W_\ell/(\E{W}(1+\zeta_n))}\Bigg]\\
		&\times \Big(1+\mathcal O\big(n^{-\eps}\big)\Big).
		\ea\ee
		As $j_{k+1}=n$, we can simplify the final product to obtain 
		\be\ba
		\frac{1}{(n)_k}\,  \sideset{}{^*}\sum_{\substack{j_\ell<i_{1,\ell}<\ldots<i_{m_\ell,\ell}\leq n,\\\ell\in[k]}}\!\!\!\!\!\!\!(1-2\zeta_n)^{-\sum_{h=1}^k m_h}\mathbb E\Bigg[{}&\prod_{t=1}^k \Big(\frac{W_t}{\E{W}}\Big)^{m_t}\prod_{t=1}^k \Big(j_t^{W_t/(\E{W}(1+\zeta_n))}\prod_{s=1}^{m_t}i_{s,t}^{-1}\Big)\\
		&\times  n^{-\sum_{\ell=1}^k W_\ell/(\E{W}(1+\zeta_n))}\Bigg]\Big(1+\mathcal O\big(n^{-\eps}\big)\Big).
		\ea \ee 		
		We bound this from above even further by no longer constraining the indices $i_{s,t}$ to be distinct. That is, for distinct $t_1$ and $t_2\in[k]$, we allow $i_{s_1,t_1}=i_{s_2,t_2}$ to hold for any $s_1\in[m_{t_1}]$ and $s_2\in[m_{t_2}]$. This yields
		\be\ba\label{eq:ubexp} 
		\frac{1}{(n)_k}\, \sum_{\substack{j_\ell<i_{1,\ell}<\ldots<i_{m_\ell,\ell}\leq n,\\\ell\in[k]}}\!\!\!\!\!\!\!(1-2\zeta_n)^{-\sum_{h=1}^k m_h}\mathbb E\Bigg[{}&\prod_{t=1}^k \Big(\frac{W_t}{\E{W}}\Big)^{m_t}\prod_{t=1}^k \Big(j_t^{W_t/(\E{W}(1+\zeta_n))}\prod_{s=1}^{m_t}i_{s,t}^{-1}\Big)\\
		&\times  n^{-\sum_{\ell=1}^k W_\ell/(\E{W}(1+\zeta_n))}\Bigg]\Big(1+\mathcal O\big(n^{-\eps}\big)\Big).
		\ea\ee 		
		We set
		\be \label{eq:at}
		a_t:=W_t/(\E{W}(1+\zeta_n)), 
		\ee
		and look at the terms 
		\be \label{eq:simplify}
		\frac{n^{-\sum_{t=1}^k a_t}}{(n)_k}\ \sum_{\substack{j_\ell<i_{1,\ell}<\ldots<i_{m_\ell,\ell}\leq n,\\\ell\in[k]}}\prod_{t=1}^k \bigg( (a_t(1+\zeta_n))^{m_t}j_t^{a_t} \prod_{s=1}^{m_t}i_{s,t}^{-1}\bigg) .
		\ee 
		We bound the sums from above by multiple integrals, almost surely, which yields 
		\be \ba\label{eq:intbound}
		\frac{n^{-\sum_{t=1}^k a_t}}{(n)_k}\prod_{t=1}^k (a_t(1+\zeta_n))^{m_t}j_t^{a_t} \int_{j_t}^n \int _{x_{1,t}}^n\cdots \int_{x_{m_t-1,t}}^n \prod_{s=1}^{m_t}x_{s,t}^{-1}\,\d x_{m_t,t}\ldots\d x_{1,t}.
		\ea \ee 
		Using~\eqref{eq:logint} in Lemma~\ref{lemma:logints}, we obtain that this equals
			\be 
			\frac{n^{-\sum_{t=1}^k a_t}}{(n)_k}\prod_{t=1}^k (a_t(1+\zeta_n))^{m_t}j_t^{a_t} \frac{(\log(n/j_t))^{m_t}}{m_t!}.
			\ee 
		Substituting this in~\eqref{eq:simplify} and reintroducing the sum over the indices $j_1,\ldots, j_k$, we arrive at
		\be \label{eq:intstep1}
		\frac{(1+\zeta_n)^{\sum_{h=1}^k m_h}}{(n)_k}\sum_{n^{\eps}\leq j_1<\ldots<j_k\leq n}\prod_{t=1}^k\Big(\frac{j_t}{n}\Big)^{a_t}\frac{(a_t\log(n/j_t))^{m_t}}{m_t!}.
		\ee 
		We observe that switching the order of the indices $j_1,\ldots,j_k$ achieves the same result as permuting the $m_1,\ldots,m_k$ and $a_1,\ldots, a_k$. Hence, if we let $\pi:[k]\to[k]$ be a permutation, then considering the indices $n^{\eps}\leq j_{\pi(1)}<j_{\pi(2)}<\ldots<j_{\pi(k)}\leq n$ yields a similar result as in \eqref{eq:intstep1} but with a term $j_{\pi(t)}^{a_{\pi(t)}}(\log(n/j_{\pi(t)}))^{m_{\pi(t)}}/m_{\pi(t)}!$ in the final product. Since this product is invariant to such permutations of the $m_t$ and $a_t$, the only thing that would change is the summation order of the indices $j_1,\ldots,j_k$.
			
			By using~\eqref{eq:intstep1} in~\eqref{eq:ubexp} and incorporating all permutations of the indices, i.e.\ $n^\eps \leq j_{\pi(1)}<j_{\pi(2)}<\ldots<j_{\pi(k)}\leq n$ for all $\pi\in P_k$ (where $P_k$ denotes the set of all permutations on $[k]$), we arrive at 
			\be \ba \label{eq:permutations}
			\frac{1}{(n)_k}{}&\sum_{n^{\eps}\leq j_1\neq \ldots \neq j_k\leq n}\!\!\!\!\!\!\!\!\!\!\!\mathbb E[\Pf{\Zm_n(j_\ell)=m_\ell,\ell\in[k]}\ind_{E^{(1)}_n}]\\
			\leq{}& \frac{1}{(n)_k}\Big(\frac{1+\zeta_n}{1-2\zeta_n}\Big)^{\sum_{h=1}^k m_h}\E{\sum_{\pi\in P_k}\sum_{n^{\eps}\leq j_{\pi(1)}<\ldots<j_{\pi(k)}\leq n}\prod_{t=1}^k\Big(\frac{j_{\pi(t)}}{n}\Big)^{a_t}\frac{(a_t\log(n/j_{\pi(t)}))^{m_t}}{m_t!}}\\
			={}&\frac{1}{(n)_k}\Big(\frac{1+\zeta_n}{1-2\zeta_n}\Big)^{\sum_{h=1}^k m_h}\E{\sum_{n^{\eps}\leq j_1\neq\ldots\neq j_k\leq n}\prod_{t=1}^k\Big(\frac{j_t}{n}\Big)^{a_t}\frac{(a_t\log(n/j_t))^{m_t}}{m_t!}}.
			\ea \ee 
			We bound the last expression from above even further by allowing the indices $j_1,\ldots, j_k$ to take \emph{any} integer value in $[n^\eps,n]$. That is, $j_{\ell_1}=j_{\ell_2}$ is allowed for $\ell_1\neq \ell_2$. As this introduces more non-negative terms, it yields the upper bound 
			\be \ba
			\frac{1}{(n)_k}\Big(\frac{1+\zeta_n}{1-2\zeta_n}\Big)^{\sum_{h=1}^k m_h}\mathbb E\Bigg[\sum_{n^{\eps}\leq j_1\leq n}\cdots{}& \sum_{n^\eps \leq j_k\leq n}\prod_{t=1}^k\Big(\frac{j_t}{n}\Big)^{a_t}\frac{(a_t\log(n/j_t))^{m_t}}{m_t!}\Bigg]\\
			=\frac{1}{(n)_k}\Big(\frac{1+\zeta_n}{1-2\zeta_n}\Big)^{\sum_{h=1}^k m_h}\prod_{t=1}^k {}&\E{\sum_{n^{\eps}\leq j\leq n}\Big(\frac{j}{n}\Big)^{a_t}\frac{(a_t\log(n/j))^{m_t}}{m_t!}},
			\ea \ee 
			where we note that the product can be taken out in the second line, since the weights are independent. We now observe two things: First, we can redefine $a_t=a:=W/(\E W(1+\zeta_n))$. After all, the index of the vertex-weights $W_1,\ldots, W_k$ is irrelevant, since they are independent and identically distributed. Second, the summand is equal to $\Pf{P(a)=m_t}$, where $P(a)\sim \text{Poi}(a\log(n/j))$, conditionally on $W$. So we obtain, similar to~\eqref{eq:poiprob} in the non-rigorous proof,
			\be 
			\frac{1}{(n)_k}\Big(\frac{1+\zeta_n}{1-2\zeta_n}\Big)^{\sum_{h=1}^k m_h}\prod_{t=1}^k\E{\sum_{n^{\eps}\leq j\leq n}\Pf{P(a)=m_t}}.
			\ee 
			We now use the duality between Poisson and gamma random variables, as also explained in the intuitive proof after~\eqref{eq:poiprob}. That is, let $G_t\sim \text{Gamma}(m_t,1)$ and $V\sim \text{Exp}(1)$ be independent random variables. Then, $\{P(a)=m_t\}=\{G_t\leq a\log(n/j), G_t+V>a\log(n/j)\}$, so that
			\be\ba\label{eq:poitogamma}
			\Pf{P(a)=m_t}&=\Pf{G_t\leq a\log(n/j), G_t+V>a\log(n/j)}\\
			&=\Pf{Y_t\leq \log(n/j), Y_t+\wt V>\log(n/j)}\\
			&=\Pf{Y_t\leq \log(n/j)}-\mathbb P_\F(Y_t+\wt V\leq \log(n/j)),
			\ea \ee
			where $Y_t\sim \text{Gamma}(m_t,a)$ and $\wt V\sim \text{Exp}(a)$, conditionally on $W$ (and observe that $Y_t+\wt V\sim \text{Gamma}(m_t+1,a)$). We thus obtain 
			\be \label{eq:gammadiffprob}
			\frac{1}{(n)_k}\Big(\frac{1+\zeta_n}{1-2\zeta_n}\Big)^{\sum_{h=1}^k m_h}\prod_{t=1}^k\E{\sum_{n^{\eps}\leq j\leq n}\big(\Pf{Y_t\leq \log(n/j)}-\mathbb P_\F(Y_t+\wt V\leq \log(n/j))\big)},
			\ee 
			similar to~\eqref{eq:gammadif} in the non-rigorous proof. We now compare the sum of the gamma probabilities with an integral. That is, since both probabilities are decreasing in $j$, we can bound the sum from above by
			\be \ba \label{eq:sumtoint}
			\int_{\lfloor n^\eps \rfloor}^n{}& \Pf{Y_t\leq \log(n/x)}\,\d x-\int_{\lceil n^\eps\rceil}^n \mathbb P_\F(Y_t+\wt V\leq \log(n/x))\,\d x\\
			&=n\int_0^{\log(n/\lfloor n^\eps \rfloor)} \e^{-y}\Pf{Y_t\leq y}\,\d y-n\int_0^{\log(n/\lceil n^\eps\rceil)}\e^{-y}\mathbb P_\F(Y_t+\wt V\leq y)\,\d y,
			\ea \ee  
			where we use a variable substitution $y=\log(n/x)$ to obtain the second line. Writing the cumulative density function as an integral from zero to $y$ and switching the order of integration, we find
			\be\ba  
			n{}&\Ef{}{\e^{-Y_t}\ind_{\{Y_t\leq \log(n/\lfloor n^\eps\rfloor)\}}}-n\Ef{}{\e^{-(Y_t+\wt V)}\ind_{\{Y_t+\wt V\leq \log( n/\lceil n^\eps\rceil)\}}}\\
			&+\mathbb P_W(Y_t+\wt V\leq \log(n/\lceil n^\eps\rceil))+\lfloor n^\eps\rfloor \big(\mathbb P_W(Y_t+\wt V\leq \log(n/\lceil n^\eps\rceil))-\Pf{Y_t\leq \log(n/\lfloor n^\eps\rfloor)}\big).
			\ea \ee 
			As the last term is negative, we can omit it to obtain an upper bound. The first term on the second line can be bounded from above by
			\be 
			\mathbb P_W(Y_t+\wt V\leq \log(n/\lceil n^\eps\rceil))\leq \mathbb P_W(\e^{-(Y_t+\wt V)}\geq n^{-(1-\eps)})\leq n^{1-\eps}\E{\e^{-(Y_t+\wt V)}}=n^{1-\eps}\Big(\frac{a}{1+a}\Big)^{m_t+1}.
			\ee 
			Combined, this yields the upper bound
			\be \label{eq:gammaexp}
			n\Ef{}{\e^{-Y_t}\ind_{\{Y_t\leq \log(n/\lfloor n^\eps\rfloor)\}}}-n\Ef{}{\e^{-(Y_t+\wt V)}\ind_{\{Y_t+\wt V\leq \log( n/\lceil n^\eps\rceil)\}}}+n^{1-\eps}\Big(\frac{a}{1+a}\Big)^{m_t+1}.
			\ee 
			We now distinguish two cases: 
			\begin{enumerate}
				\item[\namedlabel{item:1}{$(1)$}] $m_t=c_t\log n+o(\log n)$ with $c_t\in[0,1/(\theta-1)]$, for all $t\in[k]$. 
				\item[\namedlabel{item:2}{$(2)$}] $m_t=c_t\log n+o(\log n)$ with $c_t\in(1/(\theta-1),c)$, for some $t\in[k]$.
			\end{enumerate}
			In case~\ref{item:1}, we bound the difference of the expected values from above by 
			\be\ba \label{eq:expdiff}
			n{}&\Ef{}{\e^{-Y_t}}-n\Ef{}{\e^{-(Y_t+\wt V)}}+n\Ef{}{\e^{-(Y_t+\wt V)}\ind_{\{Y_t+\wt V>\log(n/\lceil n^\eps\rceil)\}}}+n^{1-\eps}\Big(\frac{a}{1+a}\Big)^{m_t+1}\\
			&\leq n\frac{1}{1+a}\Big(\frac{a}{1+a}\Big)^{m_t}(1+\cO(n^{-\eps}))+\cO(n^\eps).
			\ea\ee  
			Substituting this for the argument of the expected value for each $t\in[k]$ in~\eqref{eq:gammadiffprob} yields
			\be \label{eq:expprod2}
			\Big(\frac{1+\zeta_n}{1-2\zeta_n}\Big)^{\sum_{h=1}^k m_h}(1+\cO(n^{-\eps}))\prod_{t=1}^k\bigg(\E{\frac{1}{1+a}\Big(\frac{a}{1+a}\Big)^{m_t}}+\cO(n^{-(1-\eps)})\bigg).
			\ee 
			We recall that $a:=W/(\E W(1+\zeta_n))$ and that $\zeta_n:=n^{-\delta\eps}/\E W$. Since $W\in[0,1]$ almost surely and $m_t=\cO(\log n)$, 
			\be \label{eq:aapprox}
			\E{\frac{1}{1+a}\Big(\frac{a}{1+a}\Big)^{m_t}}=\E{\frac{\E W}{\E W+W}\Big(\frac{W}{\E W+W}\Big)^{m_t}}(1+o(n^{-\beta}))=p_{m_t}(1+o(n^{-\beta})), 
			\ee 
			for some small $\beta>0$. We have by Lemma~\ref{lemma:pkbound} that $p_{m_t}\geq (\theta+\xi)^{-m_t}$ for any $\xi>0$. Now we use that in case~\ref{item:1} $m_t=c_t\log n+o(\log n)$ with $c_t<1/(\theta-1)$. As a result, for $\xi$, $\eps,$ and some $\beta>0$ sufficiently small and since $\log\theta<\theta-1$ for $\theta>1$, we then have $n^{-(1-\eps)}=o( n^{-(c_t+o(1))\log(\theta+\xi)-\beta})=o(p_{m_t}n^{-\beta})$. Together with~\eqref{eq:aapprox}, this implies we can write the argument of the product in~\eqref{eq:expprod2} as $p_{m_t}(1+o(n^{-\beta}))$ whenever $m_t$ satisfies case~\ref{item:1}. 
			
			In case~\ref{item:2}, these bounds do not suffice. First of all, we note that  because of the product structure in~\eqref{eq:gammadiffprob}, we can deal with any 
				terms that satisfy $c_t \in [0,1/(\theta -1)]$ in the same way as in~\ref{item:1}. Without loss of generality we thus assume that $c_t \in (1/(\theta-1), c)$ for all $t  \in[k]$.	In this case, we use that for any $N>0$, 
			\be \ba\label{eq:expre}
			\Ef{}{\e^{-Y_t}\ind_{\{Y_t\leq N\}}}&=\int_0^N \frac{a^{m_t}}{\Gamma(m_t)}x^{m_t-1}\e^{-(1+a)x}\,\d x\\
			&=\Big(\frac{a}{1+a}\Big)^{m_t}\int_0^N\frac{(1+a)^{m_t}}{\Gamma(m_t)}x^{m_t-1}\e^{-(1+a)x}\,\d x\\
			&=\Big(\frac{a}{1+a}\Big)^{m_t}\Pf{Y_t'\leq N},
			\ea \ee 
			where $Y_t'\sim \text{Gamma}(m_t,1+a)$, conditionally on $W$. As $Y_t+\wt V\sim \text{Gamma}(m_t+1,a)$, we also obtain a similar result for the second expected value in~\eqref{eq:gammaexp} with a random variable $Y''_t\sim \text{Gamma}(m_t+1,1+a)$. As in the non-rigorous proof, we can assume that $Y_t'$ and $Y''_t$ are defined on the same probability space and we have that $Y_t' \preceq Y_t''$. Using this in~\eqref{eq:gammaexp}, we thus obtain 
			\be \ba\label{eq:mainterm}
			n{}&\Big(\frac{a}{1+a}\Big)^{m_t}\bigg[\Pf{Y_t'\leq \log(n/\lfloor n^\eps\rfloor)}-\frac{a}{1+a}\Pf{Y_t''\leq \log(n/\lceil n^\eps\rceil)}\bigg]+n^{1-\eps}\Big(\frac{a}{1+a}\Big)^{m_t+1}\\
			&\leq n\Big(\frac{a}{1+a}\Big)^{m_t}\bigg[\frac{1}{1+a}\Pf{Y'_t\leq \log(n/\lfloor n^\eps\rfloor)}+\Pf{Y'_t\leq \log(n/\lceil n^\eps\rceil), Y''_t>\log(n/\lceil n^\eps\rceil)}\\
			&\hphantom{\leq n\Big(\frac{a}{1+a}\Big)^{m_t}\bigg[}\ +\Pf{\log(n/\lceil n^\eps\rceil)\leq Y'_t\leq \log(n/\lfloor n^\eps\rfloor)}\bigg]+n^{1-\eps}\Big(\frac{a}{1+a}\Big)^{m_t+1}.
			\ea \ee 
			By omitting the probability from the first term and bounding the other two probabilities from above by $\Pf{Y''_t>\log(n/\lceil n^\eps\rceil)}$, we obtain an upper bound
			\be 
			n\Big(\frac{a}{1+a}\Big)^{m_t}\Big(\frac{1}{1+a}(1+\cO(n^{-\eps}))+2\Pf{Y''_t>\log(n/\lceil n^\eps\rceil)}\Big).
			\ee 
			Again, we substitute this for the argument in the expected value in~\eqref{eq:gammadiffprob} to obtain for some large constant $C>0$,
			\be \ba\label{eq:expsum}
			\Big(\frac{1+\zeta_n}{1-2\zeta_n}\Big)^{\sum_{h=1}^k m_h}(1+\cO(n^{-\eps}))\prod_{t=1}^k\bigg({}&\E{\Big(\frac{1}{1+a}\Big(\frac{a}{1+a}\Big)^{m_t}}\\
			+{}&C\E{\Big(\frac{a}{1+a}\Big)^{m_t}\big(\Pf{Y''_t>\log(n/\lceil n^\eps\rceil)}+n^{-\eps}\big)}\bigg).
			\ea\ee  
			It thus remains to show that the second expected value in~\eqref{eq:expsum} is of sufficiently smaller order compared to the first. We obtain this by using the indicators $\ind_{\{(1+a)<c_t/(1-\mu)\}}$ and $\ind_{\{1+a\geq c_t/(1-\mu)\}}$ to split the expected value in two parts, for some $\mu>0$ small enough such that $c_t/(1-\mu)<\theta/(\theta-1)$ (which is possible since $c_t<c<\theta/(\theta-1)$). For the first indicator, we can bound
			\be\ba 
			\E{\Big(\frac{a}{1+a}\Big)^{m_t}\big(\Pf{Y''_t>\log(n/\lceil n^\eps\rceil)}+n^{-\eps}\big)\ind_{\{1+a<c_t/(1-\mu)\}}}&\leq \Big(\frac{c_t/(1-\mu)-1}{c_t/(1-\mu)}\Big)^{m_t}\\
			&=\Big(1-\frac{1-\mu}{c_t}\Big)^{m_t}<(\theta+\xi)^{-m_t}, 
			\ea \ee  
			for some small $\xi>0$ and $n$ sufficiently large. Since the first expected value in~\eqref{eq:expsum} equals $p_{m_t}\geq (\theta+\xi/2)^{-m_t}$, where the inequality follows from Lemma~\ref{lemma:pkbound}, we thus obtain that 
			\be \label{eq:ind}
			\mathbb E\bigg[\Big(\frac{a}{1+a}\Big)^{m_t}\Pf{Y''_t>\log(n/\lceil n^\eps\rceil)}\ind_{\{1+a<c_t/(1-\mu)\}}\bigg]=o\bigg(\mathbb E\bigg[\Big(\frac{1}{1+a}\Big(\frac{a}{1+a}\Big)^{m_t}\bigg]n^{-\beta}\bigg),
			\ee 
			for some small $\beta>0$, since $m_t=c_t\log n+o(\log n)$ with $c_t\in[1/\log \theta,c)$. For the second indicator, we observe that $\{Y''_t>\log(n/\lceil n^\eps\rceil)\}=\{Z_t>(1+a)(1-\eps)\log n(1+o(1))\}$, where $Z_t\sim \text{Gamma}(m_t+1, 1)$. We then have 
			\be \ba
			\mathbb E{}&\bigg[\Big(\frac{a}{1+a}\Big)^{m_t}\Pf{Z_t>(1+a)(1-\eps)\log n(1+o(1))}\ind_{\{1+a\geq c_t/(1-\mu)\}}\bigg]\\
			&\leq \mathbb E\bigg[\Big(\frac{a}{1+a}\Big)^{m_t}\bigg]\P{Z_t>\frac{c_t(1-\eps)}{1-\mu}\log n(1+o(1))}.
			\ea \ee 
			It thus suffices to show that the probability on the right-hand side is $o(n^{-\beta})$ for some small $\beta>0$. By choosing $\eps\in(0,\mu)$ we can apply a standard large deviation bound. Let $(V_i)_{i\in\N}$ be i.i.d.\ exponential rate $1$ random variables and let $I(a):=a-1-\log(a)$ be their rate function. Then, as we can think of $Z_t$ as the sum of $V_1,\ldots,V_{m_t+1}$,
			\be \ba 
			\mathbb P\Big(Z_t\geq \frac{c_t (1-\eps)}{1-\mu}\log n(1+o(1))\Big)&=\P{\sum_{i=1}^{m_t+1}V_i\geq (m_t+1) \frac{c_t  (1-\eps)\log n(1+o(1))}{(1-\mu)(m_t+1)}}\\
			&\leq\exp\Big(-(m_t+1)I\Big(\frac{c_t (1-\eps)\log n(1+o(1))}{(1-\mu)(m_t+1)}\Big)\Big).
			\ea \ee 
			In the first step, we express the upper bound within the probability in terms of the mean of the sum of random variables, which equals $m_t+1$. We then use the large deviations bound in the second step, which we can do as the argument of $I$ is strictly greater than $1$ when $n$ is sufficiently large (as $m_t+1\sim c_t\log n$) and $\eps\in(0,\mu)$ is sufficiently small. Since $I((c_t+o(1)) (1-\eps) \log n/((1-\mu)(m_t+1)))=  (1-\eps)/(1-\mu)-1-\log(  (1-\eps)/(1-\mu))+o(1)$, we thus arrive at 
			\be \ba \label{eq:ztbound}
			\mathbb P\Big(Z_t\geq \frac{c_t (1-\eps)}{1-\mu}\log n(1+o(1))\Big)&\leq \e^{-c_t\log n( (1-\eps)/(1-\mu)-1-\log( (1-\eps)/(1-\mu))+o(1))}\\
			&=n^{-c_{t,\mu,\eps}+o(1)},
			\ea\ee 
			where $c_{t,\mu,\eps}:=c_t(  (1-\eps)/(1-\mu)-1-\log(  (1-\eps)/(1-\mu)))>0$ as $\eps<\mu$ and $x-1-\log x>0$ holds for $x>1$. This yields~\eqref{eq:ind} when using $\ind_{\{1+a\geq c_t/(1-\mu)\}}$ for $\mu$ and $\eps$ sufficiently small as well. Together with~\eqref{eq:aapprox}, this yields the desired result.
			
			Combined with the analysis in case~\ref{item:1}, and the fact that 
			\be 
			\Big(\frac{1+\zeta_n}{1-2\zeta_n}\Big)^{\sum_{h=1}^k m_h}(1+\cO(n^{-\eps}))=1+o(n^{-\beta}), 
			\ee 
			for some small $\beta>0$, by the choice of $\zeta_n$, and the fact that $m_t=\cO(\log n)$ for all $t\in[k]$, it follows that~\eqref{eq:expprod2} is at most $(o(n^{-\beta})\prod_{t=1}^k p_{m_t})$ for some small $\beta>0$ in both cases~\ref{item:1} and~\ref{item:2}. Finally, using this in~\eqref{eq:enbound} yields 
				\be 
				\frac{1}{(n)_k}{}\sum_{n^{\eps}\leq j_1\neq \ldots \neq j_k\leq n}\P{\Zm_n(j_\ell)=m_\ell\text{ for all } \ell\in[k]}\leq \prod_{t=1}^k \big(p_{m_t}(1+o(n^{-\beta}))\big)+o(n^{-\gamma}).
				\ee 
				As we can choose $\gamma$ arbitrarily large (as discussed prior to~\eqref{eq:enbound}), this term can incorporated in the $o(n^{-\beta})$ term regardless of the choice of $m_t$, which completes the proof of the upper bound.		
		
		\paragraph{\textbf{Lower bound.}} 
		We then focus on proving a similar lower bound. We define the event 
		\be \label{eq:wten}
		E^{(2)}_n:=\Big\{\sum_{\ell=k+1 }^j W_\ell\in(\E{W}(1-\zeta_n)j,\E{W}(1+\zeta_n)j),\text{ for all } n^{\eps}\leq j\leq n\Big\}.
		\ee
		We again have from Lemma~\ref{lemma:weightsumbounds} in the~\hyperref[sec:appendix]{Appendix} that $\mathbb P((E^{(2)}_n)^c)=o(n^{-\gamma})$ for any $\gamma>0$. We obtain a lower bound for the probability of the event $\{\Zm_n(v_\ell)=m_\ell,\ell\in[k]\}$ by omitting the second term in~\eqref{eq:enbound}. This yields, together with the rearrangements in~\eqref{eq:ubfirst},
		\be \ba 
		\frac{1}{(n)_k}{}&\sum_{n^{\eps}\leq j_1\neq \ldots \neq j_k\leq n}\E{\Pf{\Zm_n(j_\ell)=m_\ell\text{ for all } \ell\in[k]}}\\
		\geq{}&\frac{1}{(n)_k}\sum_{n^{\eps}<j_1\neq \ldots\neq j_k\leq n}\ \sideset{}{^*}\sum_{\substack{j_\ell<i_{1,\ell}<\ldots<i_{m_\ell,\ell}\leq n,\\\ell\in[k]}}\!\!\!\!\!\mathbb E\Bigg[\!\prod_{t=1}^k\prod_{s=1}^{m_t}\frac{W_{j_t}}{\sum_{\ell=1}^{i_{s,t}-1}W_\ell}\prod_{u=1}^k\!\!\!\!\!\!\!\prod_{\substack{s=j_u+1\\s\neq i_{\ell,t},\ell\in[m_t],t\in[k]}}^{j_{u+1}}\!\!\!\!\!\!\!\!\!\!\!\!\!\bigg(1-\frac{\sum_{\ell=1}^u W_{j_\ell}}{\sum_{\ell=1}^{s-1}W_\ell}\bigg)\Bigg].
		\ea \ee  
		We again start by only considering the ordered indices $n^{\eps}<j_1<\ldots <j_k$ and also omit this sum for now for ease of writing. We also omit the constraint $s\neq i_{\ell,t},\ell\in[m_t],t\in[k]$ in the final product. As this introduces more terms smaller than one, we obtain a lower bound. Then, in the two denominators, we bound the vertex-weights $W_{j_1},\ldots, W_{j_k}$ from above and below by one and zero, respectively, to obtain a lower bound 
		\be 
		\frac{1}{(n)_k}\ \sideset{}{^*}\sum_{\substack{j_\ell<i_{1,\ell}<\ldots<i_{m_\ell,\ell}\leq n,\\\ell\in[k]}}\!\!\!\!\!\mathbb E\Bigg[\prod_{t=1}^k\prod_{s=1}^{m_t}\frac{W_{j_t}}{\sum_{\ell=1}^{i_{s,t}-1}W_\ell\ind_{\{\ell\neq j_t,t\in[k]\}}+k} \prod_{u=1}^k\prod_{s=j_u+1}^{j_{u+1}}\!\!\!\bigg(1-\frac{\sum_{\ell=1}^u W_{j_\ell}}{\sum_{\ell=1}^{s-1}W_\ell\ind_{\{\ell\neq j_t,t\in[k]\}}}\bigg)\Bigg].
		\ee 
		As a result, we can now swap the labels of $W_{j_t}$ and $W_t$ for each $t\in[k]$, which again does not change the expected value, but it changes the value of the two denominators to $\sum_{\ell=k+1}^{i_{s,t}}W_\ell+k$ and $\sum_{\ell=k+1}^{i_{s,t}}W_\ell$, respectively. After this we introduce the indicator $\ind_{E^{(2)}_n}$ and use the bounds in $E^{(2)}_n$ on these sums in the expected value to obtain a lower bound. Finally, we note that the (relabelled) weights $W_t,t\in[k],$ are independent of $E^{(2)}_n$ so that we can take the indicator out of the expected value. Combining all of the above steps, we arrive at the lower bound
		\be \ba \label{eq:lbas}
		\frac{1}{(n)_k}{}&\, \sideset{}{^*}\sum_{\substack{j_\ell<i_{1,\ell}<\ldots<i_{m_\ell,\ell}\leq n,\\\ell\in[k]}}\mathbb E\Bigg[\prod_{t=1}^k \Big(\frac{W_t}{\E{W}}\Big)^{m_t} \prod_{s=1}^{m_t}\frac{1}{i_{s,t}(1+2\zeta_n)}\\ &\times\prod_{u=1}^k\prod_{s=j_u+1}^{j_{u+1}}\bigg(1-\frac{\sum_{\ell=1}^u W_\ell}{(s-1)\E W(1-\zeta_n)}\bigg)\Bigg]\mathbb P(E^{(2)}_n).\\
		\ea \ee 
		The $1+2\zeta_n$ in the fraction on the first line arises from the fact that, for $n$ sufficiently large, $(i_{s,t}-1)(1+\zeta_n)+k\leq i_{s,t}(1+2\zeta_n)$. As stated above, $\mathbb P(E^{(2)}_n)=1-o(n^{-\gamma})$ for any $\gamma>0$. Similar to the calculations in \eqref{eq:expbound} and using $\log(1-x)\geq -x-x^2$ for $x$ small, we obtain an almost sure lower bound for the final product for $n$ sufficiently large of the form 
		\be\ba
		\prod_{s=j_u+1}^{j_{u+1}}\bigg(1-\frac{\sum_{\ell=1}^u W_\ell}{(s-1)\E W(1-\zeta_n)}\bigg)&\geq \exp\bigg(-\frac{1}{\E W(1-\zeta_n)}\sum_{\ell=1}^u W_\ell\sum_{s=j_u+1}^{j_{u+1}} \frac{1}{s-1}\\
		&\qquad\quad\ \ \, -\Big(\frac{1}{\E W(1-\zeta_n)}\sum_{\ell=1}^u W_\ell\Big)^2\!\sum_{s=j_u+1}^{j_{u+1}}\!\frac{1}{(s-1)^2}\bigg)\\
		&\geq \Big(\frac{j_{u+1}}{j_u}\Big)^{-\sum_{\ell=1}^u W_\ell/(\E W(1-\zeta_n))}\Big(1-\mathcal O\big(n^{-\eps}\big)\Big).
		\ea \ee 
		Using this in~\eqref{eq:lbas} yields the lower bound
		\be\label{eq:lbstep}
		\frac{1}{(n)_k}\  \sideset{}{^*}\sum_{\substack{j_\ell<i_{1,\ell}<\ldots<i_{m_\ell,\ell}\leq n,\\\ell\in[k]}}\!\!\!\!\!\!\!\!\!\!(1+2\zeta_n)^{-\sum_{h=1}^k m_h} \mathbb E\Bigg[\prod_{t=1}^k \Big(\frac{W_t}{\E W}\Big)^{m_t} \Big(\frac{j_t}{n}\Big)^{\wt a_t}\prod_{s=1}^{m_t}i_{s,t}^{-1} \Big)\Bigg]\Big(1-\mathcal O\big(n^{-\eps}\big)\Big),
		\ee  
		where $\wt a_t:=W_t/(\E W(1-\zeta_n))$. As in the upper bound, we can take the product over $t\in[k]$ out of the expected value by the independence of the vertex-weights. This yields
		\be 
		\frac{1}{(n)_k}\  \sideset{}{^*}\sum_{\substack{j_\ell<i_{1,\ell}<\ldots<i_{m_\ell,\ell}\leq n,\\\ell\in[k]}}\Big(\frac{1-\zeta_n}{1+2\zeta_n}\Big)^{\sum_{h=1}^k m_h}\prod_{t=1}^k \mathbb E\Bigg[ \wt a_t^{m_t} \Big(\frac{j_t}{n}\Big)^{\wt a_t}\prod_{s=1}^{m_t}i_{s,t}^{-1}\Bigg]\Big(1-\mathcal O\big(n^{-\eps}\big)\Big).
		\ee 
		As $\zeta_n=n^{-\delta\eps}/\E W$ and $m_h=\cO(\log n)$ for all $h\in[k]$, the fraction and the $1-\cO(n^{-\eps})$ can be combined to yield a $1-o(n^{-\beta})$, for some small $\beta>0$. We now reintroduce the sum over the indices $n^{\eps}\leq j_1<\ldots<j_k\leq n$ and bound the sum over the  indices $i_{s,\ell}$ from below. We note that the expression in the expected value is  decreasing in $i_{s,\ell}$ and we restrict the range of the indices to $j_\ell+\sum_{h=1}^k m_h<i_{1,\ell}<\ldots< i_{m_\ell,\ell}\leq n,\ell \in[k]$, but no longer constrain the indices to be distinct (so that we can drop the $*$ in the sum). In the distinct sums and the suggested lower bound, the number of values the $i_{s,\ell}$ take on equal
		\be 
		\prod_{\ell=1}^k \binom{n-(j_\ell-1)-\sum_{h=1}^{\ell-1}m_t}{m_h} \quad \text{and} \quad
		\prod_{\ell=1}^k \binom{n-(j_\ell-1)-\sum_{h=1}^k m_h}{m_\ell},
		\ee 
		respectively. It is straightforward to see that the former allows for more possibilities than the latter, as $\binom{b}{c}> \binom{a}{c}$ when $b> a\geq c$. As we omit the largest values of the expected value (since it decreases in $i_{s,\ell}$ and we omit the largest values of $i_{s,\ell}$), we thus arrive at the lower bound
		\be\ba \label{eq:lbstep2}
		((n)_k)^{-1}\!\!\!\!\!\!\!\!\!\!\!\!\!\!\!\!\!\!\!\!\sum_{n^{\eps}<j_1<\ldots<j_k\leq n-\sum_{h=1}^k m_h} \sum_{\substack{j_\ell+\sum_{h=1}^k m_h<i_{1,\ell}<\ldots<i_{m_\ell,\ell}\leq n,\\\ell\in[k]}}\! \prod_{t=1}^k\mathbb E\bigg[\wt a_t^{m_t} \Big(\frac{j_t}{n}\Big)^{\wt a_t}\!\prod_{s=1}^{m_t}\frac{1}{i_{s,t}}\bigg]\big(1-o(n^{-\beta})\big),
		\ea \ee 
		where we also restrict the range of indices in the upper bound of the outer sum, as otherwise there would be a contribution of zero from these values of $j_1,\ldots,j_k$. We now use similar techniques compared to the upper bound of the proof to switch from summation to integration. First, we observe that we can switch the inner summation and the product over $t$. Second, we restrict the upper bound on the outer sum to $n-1-2\sum_{h=1}^k m_h$. This will prove useful later. Combined, this yields
			\be \ba\label{eq:lbsum}
			((n)_k)^{-1}\!\!\!\!\!\!\!\!\!\!\!\!\!\!\!\!\!\!\!\!\!\!\!\sum_{n^{\eps}<j_1<\ldots<j_k\leq n-1-2\sum_{h=1}^k m_h}\prod_{t=1}^k \sum_{\substack{j_\ell+\sum_{h=1}^k m_h<i_{1,\ell}<\ldots<i_{m_\ell,\ell}\leq n,\\\ell\in[k]}}\!\!\!\!\!\!\mathbb E\bigg[\wt a_t^{m_t} \Big(\frac{j_t}{n}\Big)^{\wt a_t}\!\prod_{s=1}^{m_t}\frac{1}{i_{s,t}}\bigg](1-o(n^{-\beta})).
			\ea\ee 
			Then, we take the inner summation in the expected value and, for now, focus on 
			\be 
			\sum_{\substack{j_t+\sum_{h=1}^k m_h<i_{1,t}<\ldots <i_{m_\ell,\ell}\leq n\\ \ell\in[k]}} \prod_{s=1}^{m_t}i_{s,t}^{-1}\geq \int_{j_t+\sum_{h=1}^k m_h+1}^{n+1}\int_{x_{1,t}+1}^{n+1}\cdots \int_{x_{m_t-1,t}+1}^{n+1}\prod_{s=1}^{m_t}x_{s,t}^{-1}\,\d x_{m_t,t}\ldots \d x_{1,t}.
			\ee 
			Since $j_\ell\leq n-1-2\sum_{h=1}^k m_h$ for each $\ell\in[k]$, it follows in particular that $j_\ell+1+\sum_{h=1}^k m_h\leq n$. It thus follows that we can use~\eqref{eq:loglb} in Lemma~\ref{lemma:logints} with $a=j_\ell+1+\sum_{h=1}^k m_h, b=n+1$, and $k=m_{\ell}$ to obtain the lower bound
			\be \label{eq:lbist}
			\prod_{t=1}^k \frac{1}{m_t!}\log\Big(\frac{n+1}{j_t +\sum_{h=1}^k m_h+m_t}\Big)^{m_t}\geq \prod_{t=1}^k \frac{1}{m_t!}\Big(\log\Big(\frac{n}{j_t +2\sum_{h=1}^k m_h}\Big)\Big)^{m_t}.
			\ee 
			Substituting this in \eqref{eq:lbsum} yields the lower bound
			\be
			\frac{1}{(n)_k}\sum_{n^{\eps}<j_1<\ldots<j_k\leq n-1-2\sum_{h=1}^k m_h} \prod_{t=1}^k \mathbb E\Bigg[ \Big(\frac{j_t}{n}\Big)^{\wt a_t}\frac{\wt a_t^{m_t}}{m_t!}\Big(\log\Big(\frac{n}{j_t +2\sum_{h=1}^k m_h}\Big)\Big)^{m_t} \bigg]\big(1-o(n^{-\beta})\big).
			\ee 
		Note that, due to the independence of the vertex-weights $W_1,\ldots, W_k$, and hence of $\wt a_1, \ldots, \wt a_k$, we can interchange the product and the sum with the expected value. To simplify the summation over $j_1,\ldots,j_k$, we write the summand as 
		\be
		\prod_{t=1}^k\Big(j_t+2\sum_{h=1}^k m_h\Big)^{\wt a_t}\frac{n^{-\wt a_t}}{m_t!}\Big(\log\Big(\frac{n}{j_t +2\sum_{h=1}^k m_h}\Big)\Big)^{m_t}\bigg(1-\frac{2\sum_{h=1}^k m_h}{j_t +2\sum_{h=1}^k m_h}\bigg)^{\wt a_t}.
		\ee
		Using that $m_t<c \log n,j_t\geq n^{\eps}$ and almost surely, $x^{\wt a_t}\geq x^{1/(\E W(1-\zeta_n))}$ for $x\in[0,1]$, we obtain the lower bound
		\be
		\prod_{t=1}^k \Big(j_t+2\sum_{h=1}^k m_h\Big)^{\wt a_t}\frac{n^{-\wt a_t}}{m_t!}\Big(\log\Big(\frac{n}{j_t +2\sum_{h=1}^k m_h}\Big)\Big)^{m_t}\big(1-o(n^{-\beta})\big).
		\ee 
		We can then shift the bounds on the range of the sum to $n^{\eps}+2\sum_{h=1}^k m_h$ and $n-1$. As $j_k=n$ yields a contribution of zero to the sum, we can increase the upper range on the summation from $n-1$ to $n$. We thus obtain the lower bound
		\be 
		\frac{1}{(n)_k}\E{\sum_{n^{\eps}+2\sum_{h=1}^k m_h<j_1<\ldots<j_k\leq n}\prod_{t=1}^k \Big(\frac{j_t}{n}\Big)^{\wt a_t}\frac{1}{m_t!}(\log(n/j_t))^{m_t} }\big(1-o(n^{-\beta})\big).
		\ee		
			We can now use a similar approach as in~\eqref{eq:permutations} through~\eqref{eq:gammaexp}. By considering all permutations of the indices $j_1, \ldots, j_k$, we obtain
			\be\ba
			\frac{1}{(n)_k}{}&\sum_{n^{\eps}\leq j_1\neq \ldots \neq j_k\leq n}\mathbb E[\Pf{\Zm_n(j_t)=m_t\text{ for all } t\in[k]}]\\
			&\geq \frac{1-o(n^{-\beta})}{(n)_k}\mathbb E\Bigg[\sum_{n^{\eps}+2\sum_{h=1}^k m_h<j_1\neq \ldots\neq j_k\leq n}\prod_{t=1}^k (j_t/n)^{\wt a_t}\frac{1}{m_t!}(\wt a_t\log(n/j_t))^{m_t}\Bigg].
			\ea\ee 
			We now aim to again exchange the order of summation and the product in the expected value at the expense of some additional terms. To this end, we allow the indices $j_1, \ldots, j_k$ to take on the same values, but also subtract the largest terms from the sum to still obtain a lower bound. We are thus required to determine which terms are the largest. The following holds for any $t\in[k]$:  $j_t\mapsto (j_t/n)^{\wt a_t}(\wt a_t\log(n/j_t))^{m_t}/m_t!$ is increasing on $[1,n\exp(-m_t/\wt a_t))$ and decreasing on $(n\exp(-m_t/\wt a_t), n]$, with maximum value $\e^{-m_t}m_t^{m_t}/m_t!$. As such, we again distinguish the two cases~\ref{item:1} and~\ref{item:2} as in the upper bound (i.e.\ $m_t=c_t\log n+o(\log n)$ with either $c_t\leq 1/(\theta-1)$ for all $t \in [k]$ or $c_t>1/(\theta-1)$ for some $t \in [k]$). In case~\ref{item:2}, we know that $n\exp(-m_t/\wt a_t)\leq n^{1-c_t(\theta-1)(1+o(1))}=o(1)$, so that $(j_t/n)^{\wt a_t}(\wt a_t\log(n/j_t))^{m_t}/m_t!$ is maximised for the smallest value of $j_t$. Omitting these terms thus yields a lower bound. In case~\ref{item:1}, this is not the case, and hence we omit the largest terms at $n\exp(-m_t/\wt a_t)$.
			
			Let us deal with case~\ref{item:1} first. Since the sums contain
			\be 
			\Big(n-\Big\lceil n^\eps +\sum_{h=1}^k m_h\Big \rceil +1\Big)_k
			\ee 
			many terms, we bound these sums from below by 
			\be \ba
			\sum_{n^\eps+\sum_{h=1}^k m_h\leq j_1\leq n}{}&\!\!\cdots \!\!\sum_{n^\eps+\sum_{h=1}^k m_h\leq j_k\leq n}\prod_{t=1}^k \frac{(j_t/n)^{\wt a_t}}{m_t!}(\wt a_t\log(n/j_t))^{m_t}\\
			& - \cO\Big(\Big(n-\Big\lceil n^\eps -\sum_{h=1}^k m_h\Big \rceil +1\Big)^{k-1}\Big) \prod_{t=1}^k\e^{-m_t}\frac{m_t^{m_t}}{m_t!}\\
			={}&\prod_{t=1}^k \bigg(\sum_{n^\eps+\sum_{h=1}^k m_h\leq j\leq n}\frac{(j/n)^{\wt a_t}}{m_t!}(\wt a_t \log(n/j))^{m_t}\bigg)-\cO(n^{k-1}), 
			\ea\ee 
			where the last line follows from the fact that $\e^{-x} x^x/\Gamma(x+1)\leq 1$ for all $x>0$. We now identify the argument of the sum as $\Pf{P(\wt a_t,j)=m_t}$, where $P(\wt a_t)\sim \text{Poi}(\wt a_t \log(n/j))$. Again using the Poisson-gamma duality, we can rewrite this sum as 
			\be \ba
			\sum_{n^\eps+\sum_{h=1}^k m_h\leq j\leq n}{}&\Pf{Y_t\leq \log (n/j), Y_t+V>\log(n/j)}\\
			&=\sum_{n^\eps+\sum_{h=1}^k m_h\leq j\leq n}\big(\Pf{Y_t\leq  \log (n/j)}-\Pf{Y_t+V\leq \log(n/j)}\big), 
			\ea \ee  
			where $Y_t\sim \text{Gamma}(m_t, \wt a_t)$ and $V_t\sim \text{Exp}(\wt a_t)$ and both random variables are independent, conditionally on $W_t$. With a similar approach as in~\eqref{eq:sumtoint} and~\eqref{eq:gammaexp}, we can bound this sum from below by 
			\be 
			n\Ef{}{\e^{-Y_t}\ind_{\{Y_t\leq \log(n/f(n))\}}}-n\Ef{}{\e^{-(Y_t+V_t)}\ind_{\{Y_t+V_t\leq \log(n/\wt f(n))\}}}+\cO(n^\eps),
			\ee 
			with $f(n):=\lceil n^\eps +\sum_{h=1}^k m_h\rceil$, $\wt f(n):=f(n)-1$. As $f(n)\sim \wt f(n)\sim n^\eps$, the desired lower bound follows in an analogous way to the approach used in the upper bound in~\eqref{eq:expdiff} through the paragraph following~\eqref{eq:aapprox}. 
			
			In case~\ref{item:2}, we can as in the upper bound assume without loss of generality that the condition on $c_t$ holds for all $t \in [k]$. Then, the largest values of $\Pf{P(\wt a_t,j)=m_t}$ are attained for $j_t$ as small as possible. So, by omitting these small $j_t$, we obtain the lower bound 
			\be\ba
			\sum_{n^\eps+\sum_{h=1}^k m_h\leq j_1\leq n}{}&\sum_{n^\eps+1+\sum_{h=1}^k m_h\leq j_1\leq n}\!\!\cdots \!\!\sum_{n^\eps+(k-1)+\sum_{h=1}^k m_h\leq j_k\leq n}\prod_{t=1}^k \Pf{P(\wt a_t,j_t)=m_t}\\
			& =\prod_{t=1}^k \bigg(\sum_{n^\eps+(t-1)+\sum_{\ell=1}^k m_\ell\leq j\leq n}\Pf{P(\wt a_t,j)=m_t}\bigg).
			\ea \ee 
			As in case~\ref{item:1}, we can write each sum as a difference of expected values. Then, following the same steps as in~\eqref{eq:expre} through~\eqref{eq:ztbound}, we obtain the desired matching lower bound in case~\ref{item:2} for $\eps,\beta>0$ sufficiently small. The only difference is that we cannot bound the  conditional probability in the second line~\eqref{eq:mainterm} from above by one. Instead, the lower bound we obtain is
			\be\ba 
			n\Big(\frac{a}{1+a}\Big)^{m_t}\bigg[{}&\frac{1}{1+a}\Pf{Y'_t\leq \log\Big(n\Big/\Big\lfloor n^\eps+(t-1)+\sum_{\ell=1}^k m_\ell\Big\rfloor\Big)}\\
			&-\Pf{Y''_t\geq \log\Big(n\Big/\Big\lceil n^\eps+(t-1)+\sum_{\ell=1}^k m_\ell\Big\rceil\Big)}-n^{-\eps}\frac{a}{1+a}\bigg]\\
			\geq n\Bigg[{}&\frac{1}{1+a}\Big(\frac{a}{1+a}\Big)^{m_t}\Bigg(\Pf{Y'_t\leq \log\Big(n\Big/\Big\lfloor n^\eps+(t-1)+\sum_{\ell=1}^k m_\ell\Big\rfloor\Big)}-\cO(n^{-\eps})\Bigg)\\
			&-\cO\bigg(\Big(\frac{a}{1+a}\Big)^{m_t}\bigg)\Bigg].
			\ea \ee 
			From here on out, the same approach used to bound~\eqref{eq:expsum} from above can be used to yield the lower bound
			\be
			\frac{1}{(n)_k}\sum_{n^{\eps}\leq j_1\neq \ldots \neq j_k\leq n}\E{\Pf{\Zm_n(j_\ell)=m_\ell\text{ for all } \ell\in[k]}}\geq \prod_{t=1}^k p_{m_t}(1-o(n^{-\beta})).
			\ee 
			Combined with the upper bound, this yields the desired result and concludes the proof.
	\end{proof} 
	
	We now discuss the necessary changes to the proof for it to hold for the model with \emph{random out-degree}, as introduced in Remark~\ref{remark:def}$(ii)$, as well. The main difference between these two models is the fact that the events $\{n\to i\}$ and $\{n\to j\}$ are no longer disjoint, where $u\to v$ denotes that vertex $u$ connects to vertex $v$ with an edge directed towards $v$. As a result, the probability, conditional on the vertex-weights, that vertex $n$ does \emph{not} connect to vertices $j_1,\ldots, j_k\in[n-1]$ now equals
	\be \label{eq:notconnectprob}
	\prod_{\ell=1}^k \Big(1-\frac{W_{j_\ell}}{S_{n-1}}\Big)\qquad \text{instead of}\qquad \Big(1-\frac{\sum_{\ell=1}^k W_{j_\ell}}{S_{n-1}}\Big).
	\ee 
	Moreover, as a new vertex can connect to multiple vertices at once, it is no longer necessary to restrict ourselves to indices $j_\ell<i_{1,\ell}<\ldots< i_{m_\ell,\ell}\leq n,\ell\in[k]$ with $i_{s,\ell}\neq i_{t,j}$ for any $\ell,j\in[k],s\in[m_\ell],t\in[m_j]$. Instead, when $\ell\neq j$, $i_{s,\ell}= i_{t,j}$ is allowed for any $s\in[m_\ell],t\in[m_j]$. That is, the steps at which two vertices increase their degree by one no longer need to be disjoint. 
	
	In the proof of the upper bound, this changes the right-hand side of~\eqref{eq:ubfirst} into
	\be \ba
	\frac{1}{(n)_k}\sum_{n^{\eps}\leq j_1<\ldots<j_k\leq n}\ \sum_{\substack{j_\ell<i_{1,\ell}<\ldots<i_{m_\ell,\ell}\leq n,\\\ell\in[k]}}\mathbb E\Bigg[&\prod_{t=1}^k\prod_{s=1}^{m_t}\frac{W_{j_t}}{\sum_{\ell=1}^{i_{s,t}-1}W_\ell}\\ &\times \prod_{u=1}^k\!\!\prod_{\substack{s=j_u+1\\s\neq i_{\ell,t},\ell\in[m_t],t\in[k]}}^{j_{u+1}}\prod_{\ell=1}^u\bigg(1-\frac{W_{j_\ell}}{\sum_{r=1}^{s-1}W_r}\bigg)\ind_{E_n}\Bigg],
	\ea \ee 
	where we change the last term in the expected value due to~\eqref{eq:notconnectprob} and omit the $*$ in the inner sum due to the degree increments no longer being disjoint. We can then again use the event $E_n^{(1)}$ to bound the denominators with deterministic quantities. Then, the exact same steps to obtain the upper bound as in~\eqref{eq:expbound} can be performed. Most importantly, this yields that the difference in the upper bound due to~\eqref{eq:notconnectprob} is no longer present. Furthermore, as we omit the constraint of disjoint indices $i_{s,t}$ in~\eqref{eq:ubexp}, we obtain the exact same upper bound as we would here. As a result, the remaining steps of the proof of the upper bound carry through as well.
	
	The same is true for the lower bound. In fact, in the proof of the lower bound of Lemma~\ref{lemma:in0eps} we argue why we can omit the constraint of disjoint indices $i_{s,t}$ in a certain way and still obtain a lower bound in~\eqref{eq:lbstep2}, which is no longer necessary here. Hence, we obtain~\eqref{eq:lbist} without the factor two in front of $\sum_{h=1}^k m_h$. This change carries through all the steps, so that we arrive at the same desired lower bound. Together with the upper bound, this establishes the desired result.
	
	\subsection{Proof of Lemma~\ref{lemma:inepsterm}}\label{sec:proof5.11}
	
	\begin{proof}[Proof of Lemma~\ref{lemma:inepsterm}]
		We aim to bound 
		\be \label{eq:errorterm}
		\frac{1}{(n)_k}\sum_{\textbf{j}\in I_n(\eps)}\P{\Zm_n(j_\ell)=m_\ell\text{ for all } \ell\in[k]},
		\ee 	
		where we recall that $I_n(\eps)$ from~\eqref{eq:ineps}. This is the equivalent of the steps following~\eqref{eq:2ndtermbound} in the non-rigorous proof in Section~\ref{sec:nonrigour}. We first assume that $m_\ell=c_\ell \log n+o(\log n)$ for some $c_\ell\in[0,1/\log \theta)$ for all $\ell\in[k]$. We define 
		\be 
		I_n(\eps,i):=\{\textbf{j}\in I_n(\eps): |\{\ell\in[k]: j_\ell<n^{\eps}\}|=i\},\qquad i\in[k],
		\ee
		that is, $I_n(\eps,i)$ denotes the set of indices $\textbf{j}=(j_1,\ldots, j_k)$ such that exactly $i$ of the indices are smaller than $n^{\eps}$, and note that $I_n(\eps)=I_n(\eps,1)\cup\cdots\cup I_n(\eps,k)$. We then write 
		\be 
		\frac{1}{(n)_k}\!\sum_{\textbf{j}\in I_n(\eps)}\!\!\!\P{\Zm_n(j_\ell)=m_\ell\text{ for all }\ell\in[k]}=\sum_{i=1}^k \frac{1}{(n)_k}\!\sum_{\textbf{j}\in I_n(\eps,i)}\!\!\!\!\P{\Zm_n(j_\ell)=m_\ell\text{ for all }\ell\in[k]},
		\ee  
		and bound the probability on the right-hand side from above by omitting all events $\{\Zm_n(j_\ell)=m_\ell\}$ whenever $j_\ell<n^{\eps}$. This leaves us with 
		\be \label{eq:Ssum}
		\sum_{i=1}^{k-1} \frac{1}{(n)_k}n^{i\eps}\sum_{\substack{S\subseteq [k]\\ |S|=k-i}}\ \ \sideset{}{^*}\sum_{\substack{n^{\eps}\leq j_\ell \leq n\\ \ell\in S}}\P{\Zm_n(j_\ell)=m_\ell\text{ for all }\ell \in S}+\frac{n^{k\eps}}{(n)_k},
		\ee 
		where we recall that the $*$ on the inner summation denotes that we only consider distinct values of $j_\ell,\ell\in S$. We isolate the case $i=k$ here as in this case no indices are larger than $n^{\eps}$ and we hence bound the probability from above by one, whereas $i=k$ would yield a contribution of zero in the triple sum. The inner sum can then be dealt with in the same manner as in the derivation of the upper bound in the proof of Lemma~\ref{lemma:in0eps}, to yield an upper bound
		\be 
		\sum_{i=1}^{k-1} n^{-i (1-\eps)}\sum_{\substack{S\subseteq [k]\\ |S|=k-i}}\prod_{\ell\in S}\E{\frac{\E W}{\E W+W}\Big(\frac{W}{\E W+W}\Big)^{m_\ell}}(1+o\big(n^{-\beta}\big)\big)+2n^{-k (1-\eps)},
		\ee 
		for some $\beta>0$. It thus remains to show that for any {$m_\ell=c_\ell \log n+o(\log n)$} with $c_\ell\in[0,1/\log \theta)$ we can take $\eps$ and a $\beta>0$ sufficiently small, such that 
		\be 
		n^{- (1-\eps)}=o\Big(\E{\frac{\E W}{\E W+W}\Big(\frac{W}{\E W+W}\Big)^{m_\ell}}n^{-\beta}\Big).
		\ee  
		By Lemma \ref{lemma:pkbound}, we have for any $\xi>0$ and $n$ sufficiently large, that 
		\be\label{eq:pkbound}
		\E{\frac{\E W}{\E W+W}\Big(\frac{W}{\E W+W}\Big)^{m_\ell}}\geq (\theta+\xi)^{- m_\ell}=n^{- c_\ell \log(\theta+\xi)+o(1)},
		\ee 
		and $n^{- (1-\eps)}=o(n^{-\eta-c_\ell \log(\theta+\xi)+o(1)})$ when we choose $\eta,\xi,$ and $\eps$ sufficiently small, since $c_\ell \log \theta <1$ for any $\ell\in[k]$. As a result, 
		\be 
		\frac{1}{(n)_k}\sum_{\textbf{j}\in I_n(\eps)}\P{\Zm_n(j_\ell)=m_\ell\text{ for all }\ell \in [k]}=o\Big(\prod_{\ell=1}^k\E{\frac{\E W}{\E W+W}\Big(\frac{W}{\E W+W}\Big)^{m_\ell}}n^{-\eta}\Big).
		\ee 
		We now assume (as before in the proof of Lemma~\ref{lemma:in0eps} without loss of generality) that $m_\ell=c_\ell \log n+o(\log n)$ with $c_\ell\in[1/\log \theta,\theta/(\theta-1))$ for all $\ell\in[k]$. In this case, the crude bound used above no longer suffices. Now, the aim is to combine the above with a similar approach as at the start of the proof of \cite[Theorem $2.9$, Bounded case]{LodOrt21} and also use condition~\ref{ass:weightzero} in Assumption~\ref{ass:weights}. First, we consider the set of indices $I_n(\eps,k)$. To make use of the negative quadrant dependence of the degrees $(\zni)_{\inn}$  (see Remark~\ref{rem:degtail} and~\cite[Lemma $7.1$]{LodOrt21}), we create an upper bound by considering the event $\{\Zm_n(j_\ell)\geq m_\ell\text{ for all }\ell\in[k]\}$. Then, using the tail distribution and the negative quadrant dependency of the degrees under the conditional probability measure $\mathbb P_W$ yields
		\be 
		\frac{1}{(n)_k}\sum_{\textbf{j}\in I_n(\eps,k)}\P{\Zm_n(j_\ell)=m_\ell\text{ for all }\ell \in [k]}\leq 	\frac{1}{(n)_k}\sum_{1\leq j_1\neq \ldots \neq j_k<n^{\eps}}\E{\prod_{\ell=1}^k\Pf{\Zm_n(j_\ell)\geq m_\ell}}.
		\ee 
		We then allow the indices $j_1,\ldots, j_k$ to take any value between $1$ and $n^{\eps}$, to obtain the upper bound
		\be 
		\frac{1}{(n)_k}\E{\prod_{\ell=1}^k \bigg( \sum_{i<n^{\eps}} \Pf{\zni\geq m_\ell}\bigg)}.
		\ee 
		As in the proof of \cite[Theorem $2.9$, Bounded case]{LodOrt21}, we apply a Chernoff bound to the conditional probability measure $\mathbb P_W$ to obtain 
		\be 
		\frac{1}{(n)_k}\mathbb E\Bigg[\prod_{\ell=1}^k \bigg( \sum_{i<n^{\eps}}\!\! \Pf{\zni\geq m_\ell}\bigg)\Bigg]\leq \frac{1}{(n)_k}\mathbb E\Bigg[\prod_{\ell=1}^k \bigg( \sum_{i<n^{\eps}}\!\!\exp(m_\ell(1-u_{i,\ell}+\log u_{i,\ell}))\bigg)\Bigg],
		\ee 
		where 
		\be \label{eq:uil}
		u_{i,\ell}:=\frac{1}{m_\ell}\sum_{j=i}^{n-1}\frac{W_i}{S_j}.
		\ee 
		This is the equivalent of~\eqref{eq:chernbound} in the non-rigorous proof in Section~\ref{sec:nonrigour}. We then introduce the constants $\delta\in(0,1/2),C>kc_\theta^{-1} \theta \log(\theta)/(\theta-1)$, and $\alpha>0$ (with $c_\theta:=1/(2\theta^2)$). Also, we define the sequences $\zeta'_n:=(C\log n )^{-\delta/(1-2\delta)}/ \E W$, $f(n):=\lceil \log n/\log\log\log n\rceil$, and $g(n):=\lceil \log\log n\rceil$, $n\in\N$. We then define the events 
			\be \ba 
			E^{(3)}_n&:=\bigg\{\sum_{\ell=1}^j W_\ell \geq j\E W(1-\zeta'_n),\text{ for all } (C\log n)^{1/(1-2\delta)}\leq j\leq n\bigg\},\\
			E^{(4)}_n&:=\{S_{\lceil \alpha \log n\rceil}\geq c\log n\}, \\
			E^{(5)}_n&:=\{S_{f(n)}\geq g(n)\}.
			\ea \ee 
		The event $E^{(3)}_n$ is similar to the event $E^{(1)}_n$ introduced in~\eqref{eq:en}, but considers a larger range of indices~$j$. The particular choice of the lower bound on the indices $j$ follows from the fact that we want as much control over the partial sums of the vertex-weights as possible, but need to ensure that $\mathbb P((E^{(3)}_n)^c)$ decays sufficiently fast. It follows that for all three events, the probability of their complement is $o(n^{-\gamma})$ for any $\gamma>0$ (in the case of $E^{(3)}_n$ and $E^{(4)}_n$ when we choose $C$ and $\alpha $ sufficiently large, respectively) by Lemma~\ref{lemma:weightsumbounds} in the~\hyperref[sec:appendix]{Appendix}.  
		
		We can use these events in the expected value to arrive at the upper bound 
		\be \label{eq:e'n}
		\frac{1}{(n)_k}\E{\prod_{\ell=1}^k \bigg( \sum_{i<n^{\eps}}\exp(m_\ell(1-u_{i,\ell}+\log u_{i,\ell}))\bigg)\ind_{E^{(3)}_n\cap E^{(4)}_n\cap E^{(5)}_n}}+\mathbb P((E^{(3)}_n\cap E^{(4)}_n\cap E^{(5)}_n)^c).
		\ee 
		By the proof of Corollary~\ref{cor:uil} in the~\hyperref[sec:appendix]{Appendix} and since $m_\ell=c_\ell \log n+o(\log n)$ with $c_\ell>0$, we can bound $u_{i,\ell}$ from above uniformly in $i\in[n^\eps]$ and for any $\ell\in[k]$ by $(1+o(1))/(c_\ell\E W)$. Recalling that $\E W=\theta-1$ and that $c_\ell\geq 1/\log \theta>1/(\theta-1)$ (as $\theta>1$), we find that $u_{i,\ell}<1$ for all $i\in[n]$ and $\ell\in[k]$ as well. Since $x\mapsto 1-x+\log x$ is increasing on $(0,1)$ we can thus use this upper bound in the first term of~\eqref{eq:e'n} to bound it from above by 
		\be \ba \label{eq:remarkref}
		\frac{1}{(n)_k}{}&\prod_{\ell=1}^k \bigg( \sum_{i<n^{\eps}}\exp\Big(c_\ell\log n\Big(1-\frac{1}{c_\ell \E W}+\log \Big(\frac{1}{c_\ell \E W}\Big)\Big )(1+o(1))\Big )\bigg)\\
		&\leq \exp\Big(\log n(1+o(1))\sum_{\ell=1}^k \Big(- (1-\eps)+c_\ell\Big(1-\frac{1}{c_\ell \E W}+\log\Big(\frac{1}{c_\ell \E W}\Big)\Big)\Big)\Big),
		\ea \ee 
		which is the equivalent of~\eqref{eq:uiluse} in the non-rigorous proof in Section~\ref{sec:nonrigour}. We then require that
		\be \label{eq:littleo}
		\exp\Big(\log n(1+o(1))\sum_{\ell=1}^k \Big(\!- (1-\eps)+c_\ell\Big(1-\frac{1}{c_\ell \E W}+\log\Big(\frac{1}{c_\ell \E W}\Big)\Big)\Big)\Big)=o\Big(\frac{1}{n^{\beta}}\prod_{\ell=1}^k p_{m_\ell}\Big),
		\ee 
		for some $\beta>0$. As, by Lemma~\ref{lemma:pkbound}, $p_{m_\ell}\geq (\theta+\xi)^{-m_\ell}=\exp(-\log n(1+o(1))c_\ell\log(\theta+\xi))$ for any $\xi>0$ and $n$ sufficiently large, it suffices to show that 
		\be \label{eq:csineq}
		\sum_{\ell=1}^k \Big(- (1-\eps)+c_\ell\Big(1-\frac{1}{c_\ell \E W}+\log\Big(\frac{1}{c_\ell \E W}\Big)\Big)\Big)<-\sum_{\ell=1}^k c_\ell\log(\theta+\xi)-\beta,
		\ee  
		when $\xi,\beta,$ and $\eps$ are sufficiently small. We show that this strict inequality can be achieved for each term individually, by arguing that we can choose $\eps\in(0,1)$ small enough such that 
		\be 
		(1-\eps)>c_\ell\Big(1-\frac{1}{c_\ell(\theta-1)}+\log\Big(\frac{\theta+\xi}{c_\ell(\theta-1)}\Big)\Big),\qquad \ell\in[k],
		\ee 
		where we note that we have written $\E W$ as $\theta-1$. The right-hand side is increasing in $c_\ell$ when $c_\ell\in[1/\log\theta,\theta/(\theta-1))$, so that all $k$ inequalities are satisfied when we solve 
		\be 
		(1-\eps)>\wt c\Big(1-\frac{1}{\wt c(\theta-1)}+\log\Big(\frac{\theta+\xi}{\wt c(\theta-1)}\Big)\Big),
		\ee 
		with $\wt c:=\max_{\ell\in[k]}c_\ell$. We now show that the right-hand side is strictly smaller than one when $\xi$ is sufficiently small. We write
		\be \ba
		\wt c\Big (1-\frac{1}{\wt c(\theta-1)}+\log\Big(\frac{\theta+\xi}{\wt c(\theta-1)}\Big)\Big)&=\wt c\Big(1-\frac{1}{\wt c(\theta-1)}+\log\Big(\frac{\theta}{\wt c(\theta-1)}\Big)\Big)+\wt c\log\Big(1+\frac\xi\theta\Big)\\
		&\leq \wt c\Big (1-\frac{1}{\wt c(\theta-1)}+\log\Big(\frac{\theta}{\wt c(\theta-1)}\Big)\Big)+\frac{\xi}{\theta-1},
		\ea \ee 
		where the final upper bound follows from the fact that $\log(1+x)\leq x$ for $x>-1$ and $\wt c<\theta/(\theta-1)$. We denote the first term on the right-hand side by $\kappa=\kappa(\wt c,\theta)$. As $\kappa$ is increasing in $\wt c$ when $\wt c< \theta/(\theta-1)$ (which is the case when $c_\ell<\theta/(\theta-1)$ for all $\ell\in[k]$), we have $\kappa<1$, as $\wt c<\theta/(\theta-1)$. Thus, setting $\xi<(1-\kappa)(\theta-1)/2$ we achieve the desired result. Now, taking $\eps\in(0,1-(\kappa+\xi/(\theta-1)))$, we arrive at~\eqref{eq:littleo} for some small $\eta>0$. Recalling that we can bound the second term in~\eqref{eq:e'n} by $o(n^{-\gamma})$ for any $\gamma>0$ and that $p_{m_\ell}\geq n^{-\nu}$ for some $\nu>0$ uniformly in $m_\ell<c\log n$, we obtain
		\be\label{eq:inepsk}
		\frac{1}{(n)_k}\sum_{\textbf{j}\in I_n(\eps,k)}\P{\Zm_n(j_\ell)=m_\ell\text{ for all }\ell\in[k]}=o\Big(n^{-\eta}\prod_{\ell=1}^k p_{m_\ell}\Big),
		\ee 
		for some small $\eta>0$.
		
		We now consider the remaining sets $I_n(\eps,1),\ldots, I_n(\eps,k-1)$ and aim to bound 
		\be 
		\frac{1}{(n)_k}\sum_{i=1}^{k-1}\sum_{\textbf{j} \in I_n(\eps,i)}\P{\Zm_n(j_\ell)\geq m_\ell\text{ for all }\ell \in[k]}. 
		\ee 
		Again, using the negative quadrant dependence and introducing the events $E^{(1)}_n,E^{(3)}_n,E^{(4)}_n$, and $E^{(5)}_n$ yields the upper bound
		\be 
		\frac{1}{(n)_k}\sum_{i=1}^{k-1}\sum_{\textbf{j} \in I_n(\eps,i)}\E{\ind_{E^{(1)}_n\cap E^{(3)}_n\cap E^{(4)}_n\cap E^{(5)}_n}\prod_{\ell=1}^k \Pf{\Zm_n(j_\ell)\geq m_\ell}}+\P{(E^{(1)}_n\cap E^{(3)}_n\cap E^{(4)}_n\cap E^{(5)}_n)^c}.
		\ee 
		The aim is to treat the probabilities of indices which are at most $n^{\eps}$ in the same way as when dealing with the indices in $I_n(\eps,k)$ to reach a bound as in~\eqref{eq:littleo}, for which we use the events $E^{(i)}_n, i=3,4,5$. For the indices which are larger than $n^{\eps}$ such an upper bound will not suffice. Instead, we aim to bound $\Pf{\Zm_n(j_\ell)\geq m_\ell}$ when $n^{\eps}\leq j_\ell\leq n$ in a similar way as we bounded $\Pf{\Zm_n(j_\ell)=m_\ell}$ from above in the proof of Lemma~\ref{lemma:in0eps}, for which we use the event  $E^{(1)}_n$. 
		
		First, we split the summation over $I_n(\eps,i)$ over all possible configurations of indices with are at most and at least $n^{\eps}$, similar to~\eqref{eq:Ssum}. That is, 
		\be\ba
		\frac{1}{(n)_k}{}&\sum_{i=1}^{k-1}\sum_{\textbf{j} \in I_n(\eps,i)}\E{\ind_{E^{(1)}_n\cap E^{(3)}_n\cap E^{(4)}_n\cap E^{(5)}_n}\prod_{\ell=1}^k \Pf{\Zm_n(j_\ell)\geq m_\ell}}\\
		={}&\frac{1}{(n)_k}\sum_{i=1}^{k-1}\sum_{\substack{S\subseteq [k]\\ |S|=i}}\ \sideset{}{^*}\sum_{\substack{1\leq j_\ell <n^{\eps}\\ \ell\in S}}\ \sideset{}{^*}\sum_{\substack{n^{\eps}\leq j_\ell \leq n\\ \ell \in[k]\backslash S}}\!\!\!\!\!\mathbb E\Bigg[\!\ind_{E^{(1)}_n\cap E^{(3)}_n\cap E^{(4)}_n\cap E^{(5)}_n}\!\!\prod_{\ell\in S} \Pf{\Zm_n(j_\ell)\geq m_\ell}\!\!\!\!\prod_{\ell\in [k]\backslash S}\!\!\!\!\Pf{\Zm_n(j_\ell)\geq m_\ell}\Bigg].
		\ea \ee 
		Using the events $E^{(i)}_n, i=3,4,5$, we can follow similar steps as above to bound the sum over the indices $j_\ell$ and the product of probabilities $\Pf{\Zm_n(j_\ell)\geq m_\ell}$ for $\ell\in S$ from above by the deterministic upper bound 
		\be 
		\exp\Big(\log n(1+o(1))\sum_{\ell\in S} \Big(- (1-\eps)+c_\ell\Big(1-\frac{1}{c_\ell \E W}+\log\Big(\frac{1}{c_\ell \E W}\Big)\Big)\Big)\Big)=:n^{C(S)(1+o(1))},
		\ee 
		which yields
		\be \label{eq:detbound}
		\sum_{i=1}^{k-1}\frac{n^i}{(n)_k}\sum_{\substack{S\subseteq [k]\\ |S|=i}}n^{C(S)(1+o(1))}\ \sideset{}{^*}\sum_{\substack{n^{\eps}\leq j_\ell \leq n\\ \ell \in[k]\backslash S}}\!\!\!\!\E{\ind_{E^{(1)}_n}\prod_{\ell\in [k]\backslash S}\!\!\Pf{\Zm_n(j_\ell)\geq m_\ell}}.
		\ee 
		We now proceed to bound each individual probability $\Pf{\Zm_n(j_\ell)\geq m_\ell}$ when $\ell \in [k]\backslash S$. This follows a similar approach to the upper bound of $\Pf{\Zm_n(j_\ell)=m_\ell}$ in the proof of Lemma~\ref{lemma:in0eps} (as well as the non-rigorous proof provided in Section~\ref{sec:nonrigour}), with a couple of modifications. Introducing indices $j_\ell<i_1<\ldots <i_{m_\ell}\leq n$, which denote the steps at which vertex $j_\ell$ increases its degree, we can write
		\be 
		\Pf{\Zm_n(j_\ell)\geq m_\ell}=\sum_{j_\ell<i_1<\ldots <i_{m_\ell}\leq n}\prod_{t=1}^{m_\ell} \frac{W_{j_\ell}}{\sum_{r=1}^{i_t-1}W_r}\prod_{\substack{s=j_\ell+1\\ s\neq i_t,t\in [m_\ell]}}^{i_{m_\ell}-1}\Big(1-\frac{W_{j_\ell}}{\sum_{r=1}^{s-1} W_r}\Big).
		\ee 
		The second product, in comparison to dealing with the event $\{\Zm_n(j_\ell)=m_\ell\}$, goes up to $i_{m_\ell}-1$ instead of $n$. This is due to the fact that we now only need to control the connections vertex $j$ does and does not make up to its $m_\ell^{\text{th}}$ connection. Using the same idea as in~\eqref{eq:fracbound} and using the event $E^{(1)}_n$, we obtain the upper bound
		\be \ba
		\sum_{j_\ell<i_1<\ldots <i_{m_\ell}\leq n}{}&\prod_{t=1}^{m_\ell} \frac{W_{j_\ell}}{(i_t-1)\E W(1-\zeta_n)-1}\prod_{s=j_\ell+1}^{i_{m_\ell}-1}\Big(1-\frac{W_{j_\ell}}{s\E W(1+\zeta_n)}\Big)\\
		\leq{}& \sum_{j_\ell<i_1<\ldots <i_{m_\ell}\leq n}\prod_{t=1}^{m_\ell} \frac{W_{j_\ell}}{i_t\E W(1-2\zeta_n)}\prod_{s=j_\ell+1}^{i_{m_\ell}-1}\Big(1-\frac{W_{j_\ell}}{s\E W(1+\zeta_n)}\Big).
		\ea \ee 
		The last step follows from the fact that $(i_t-1)(1-\zeta_n)\E W-1\geq i_t(1-2\zeta_n)\E W$ for $n$ sufficiently large. Using this in the expected value of~\eqref{eq:detbound} yields
		\be 
		\ \sideset{}{^*}\sum_{\substack{n^{\eps}\leq j_\ell \leq n\\ \ell \in[k]\backslash S}}\E{\prod_{\ell\in [k]\backslash S}\bigg(\sum_{j_\ell<i_1<\ldots <i_{m_\ell}\leq n}\prod_{t=1}^{m_\ell} \frac{W_{j_\ell}}{i_t\E W(1-2\zeta_n)}\prod_{s=j_\ell+1}^{i_{m_\ell}-1}\Big(1-\frac{W_{j_\ell}}{s\E W(1+\zeta_n)}\Big)\bigg)}.
		\ee 
		We can now relabel the vertex-weights $W_{j_\ell}$ as $W_\ell$, $\ell\in [k]\backslash S$. This does not change the expected value and is possible since the indices $j_\ell$ are distinct for all $\ell\in [k]\backslash S$. Directly after this, we omit the requirement that the indices $j_\ell$ are distinct, which is now of no consequence as the weights have been relabelled already. We hence arrive at the upper bound
		\be \label{eq:expprod}
		\prod_{\ell\in [k]\backslash S}\!\!\!\mathbb E\Bigg[\sum_{n^{\eps}\leq j_\ell \leq n}\ \sum_{j_\ell<i_1<\ldots <i_{m_\ell}\leq n}\prod_{t=1}^{m_\ell} \frac{W_\ell}{i_t\E W(1-2\zeta_n)}\!\prod_{s=j_\ell+1}^{i_{m_\ell}-1}\!\!\Big(1-\frac{W_\ell}{s\E W(1+\zeta_n)}\Big)\!\Bigg],
		\ee 
		where the product can be taken out of the expected value due to the independence of the vertex-weights $W_1,\ldots W_\ell$. As a result, we can deal with each of expected values individually. Following the same approach as in~\eqref{eq:expbound} and setting $a_\ell:=W_\ell/(\E W(1+\zeta_n))$, we obtain the upper bound
		\be \ba\label{eq:exp}
		\mathbb E\Bigg[{}&\sum_{n^{\eps}\leq j_\ell \leq n}\ \sum_{j_\ell<i_1<\ldots <i_{m_\ell}\leq n}\prod_{t=1}^{m_\ell} \frac{W_\ell}{i_t\E W(1-2\zeta_n)}\Big(\frac{i_{m_\ell}}{j_\ell}\Big)^{-a_\ell}\Bigg]\Big(1+\mathcal O\big(n^{-\eps}\big)\Big)\\
		={}&\mathbb E\Bigg[a_\ell^{m_\ell}\!\!\!\!\!\!\sum_{n^{\eps}\leq j_\ell \leq n}\ \sum_{j_\ell<i_1<\ldots <i_{m_\ell}\leq n}\prod_{t=1}^{m_\ell} i_t^{-1}\Big(\frac{i_{m_\ell}}{j_\ell}\Big)^{-a_\ell}\Bigg]\Big(\frac{1+\zeta_n}{1-2\zeta_n}\Big)^{m_\ell}\Big(1+\mathcal O\big(n^{-\eps}\big)\Big).
		\ea\ee 
		Let us first take out the term $(a_\ell j_\ell)^{ a_\ell}$ from the inner sums. We then observe that the summand of the inner sum over the indices $i_1,\ldots, i_{m_\ell}$ is decreasing, so that we can bound it from above almost surely by the multiple integrals
		\be 
		\int_{j_\ell}^n \int_{x_1}^n \cdots \int_{x_{m_\ell-1}}^n \prod_{t=1}^{m_\ell-1}x_t^{-1} x_{m_\ell}^{-(1+a_\ell)}\,\d x_{m_\ell}\ldots \d x_1=a_\ell^{-m_\ell}j_\ell^{-a_\ell}\bigg(1-\Big(\frac{n}{j_\ell}\Big)^{-a_\ell}\!\sum_{s=0}^{m_\ell-1}a_\ell^{s}\frac{(\log(n/j_\ell))^s}{s!}\bigg).
		\ee 
		The equality follows from~\eqref{eq:logint2} in Lemma~\ref{lemma:logints} with $a=j_\ell, b=n, c=a_\ell$, and $k=m_{\ell}$.	Again reintroducing the term $a_\ell^{m_\ell}j^{a_\ell}$ in the expected value in~\eqref{eq:exp}, we arrive at 
		\be 
		1-\sum_{s=0}^{m_\ell-1}(n/j_\ell)^{-a_\ell}\frac{(a_\ell\log(n/j_\ell))^s}{s!}=\Pf{P(a_\ell)\geq m_\ell},
		\ee 
		where $P(a_\ell)\sim \text{Poi}(a_\ell \log(n/j_\ell))$, conditionally on $W_\ell$. With a similar duality between Poisson and gamma random variables, we obtain, as in~\eqref{eq:poitogamma}, 
			\be
			\Pf{P(a_\ell)\geq m_\ell}=\Pf{Y_\ell\leq \log(n/j_\ell)}, 
			\ee
			where $Y_\ell\sim \text{Gamma}(m_\ell,a_\ell)$, conditionally on $W_\ell$. Then, by the choice of $\zeta_n$, it follows that $((1+\zeta_n)/(1-2\zeta_n))^{m_\ell}=1+\mathcal O\big(n^{-\delta\eps}\log n \big)$. Using both these results in~\eqref{eq:exp}, we arrive at
			\be \label{eq:exppw}
			\mathbb E\Bigg[\sum_{n^{\eps}\leq j_\ell \leq n}\Pf{Y_\ell\leq \log(n/j_\ell)}\Bigg]\Big(1+\mathcal O\big( n^{-\delta\eps}\log n \big)\Big).
			\ee   
			As the conditional probability is decreasing in $j_\ell$, we can bound the sum from above by an integral almost surely to obtain, as in~\eqref{eq:sumtoint} and~\eqref{eq:gammaexp}, 
			\be
			\int_{\lfloor n^{\eps}\rfloor}^n \Pf{Y_\ell\leq \log(n/x)}\,\d x \leq n\Ef{}{\e^{-Y_\ell}\ind_{\{Y_\ell\leq \log(n/\lfloor n^\eps \rfloor)\}}}\leq n\Big(\frac{a_\ell}{1+a_\ell}\Big)^{m_\ell}.
			\ee 
		Combining this almost sure upper bound with~\eqref{eq:exppw} in~\eqref{eq:expprod}, we arrive at 
		\be \label{eq:tailfin}
		n^{k-|S|}\prod_{\ell\in[k]\backslash S} \E{\Big(\frac{a_\ell}{1+a_\ell}\Big)^{m_\ell}}\Big(1+\mathcal O \big(n^{-\delta\eps}\log n\big)\Big).
		\ee 
		Finally, with the same steps as in~\eqref{eq:aapprox}, we obtain 
		\be 
		n^{k-|S|}\prod_{\ell\in[k]\backslash S} \E{\Big(\frac{W}{\E W+W}\Big)^{m_\ell}}\big(1+o\big(n^{-\beta}\big)\big),
		\ee 
		for some small $\beta>0$. Using this bound in~\eqref{eq:detbound} yields, for some constant $K>0$, the upper bound
		\be 
		K\sum_{i=1}^{k-1}\sum_{\substack{S\subseteq [k]\\ |S|=i}}n^{C(S)(1+o(1))}\prod_{\ell\in [k]\backslash S}\E{\Big(\frac{W}{\E W+W}\Big)^{m_\ell}}=K\sum_{i=1}^{k-1}\sum_{\substack{S\subseteq [k]\\ |S|=i}}n^{C(S)(1+o(1))}\prod_{\ell\in [k]\backslash S}p_{\geq m_\ell}.
		\ee 
		By Remark~\ref{rem:pgeqk}, the tail probability $p_{\geq m_\ell}=\mathcal O(p_{m_\ell})$ and by~\eqref{eq:csineq} we have\\  $n^{C(S)(1+o(1))}=o\big(n^{-\eta(S)}\prod_{\ell\in S}p_{m_\ell}\big)$ for some $\eta(S)>0$. Combined, this yields 
		\be \ba 
		\frac{1}{(n)_k}\sum_{i=1}^{k-1} \sum_{\textbf{j}\in I_n(\eps,i)}\E{\ind_{E^{(3)}_n}\ind_{E_n^{(1)}}\prod_{\ell=1}^k \Pf{\Zm_n(j_\ell)\geq m_\ell}}&\leq K\sum_{i=1}^{k-1}\sum_{\substack{S\subseteq [k]\\ |S|=i}}n^{C(S)(1+o(1))}\prod_{\ell\in [k]\backslash S}p_{\geq m_\ell}\\
		&=o\big(n^{-\wt \eta}\prod_{\ell=1}^k p_{m_\ell}\big),
		\ea \ee 
		with 
		\be
		\wt \eta:=\min_{\substack{ S\subseteq [k]\\ 1\leq |S|\leq k-1}}\Big(C(S)-\sum_{\ell\in S}\log(\theta+\xi)c_\ell\Big),
		\ee 
		which is strictly positive when $\eps$ and $\xi$ are sufficiently small, similar to what is discussed above. Combining this with the fact that $\P{\big(E^{(1)}_n\cap E^{(3)}_n\cap E^{(4)}_n\cap E^{(5)}_n\big)^c}=o\big(n^{-\beta}\prod_{\ell=1}^k p_{m_\ell}\big)$ uniformly in $m_1,\ldots, m_k< c\log n$ for some $\beta>0$ by Lemma~\ref{lemma:weightsumbounds}, and the result in~\eqref{eq:inepsk}, we finally conclude,
		\be\ba\label{eq:finboundineps}  
		\frac{1}{(n)_k}\sum_{\textbf{j}\in I_n(\eps)}\!\!\!\P{\Zm_n(j_\ell)=m_\ell\text{ for all }\ell\in[k]}={}&\frac{1}{(n)_k}\sum_{\textbf{j}\in I_n(\eps,k)}\P{\Zm_n(j_\ell)=m_\ell\text{ for all }\ell\in[k]}\\
		&+\frac{1}{(n)_k}\sum_{i=1}^{k-1}\sum_{\textbf{j}\in I_n(\eps,i)}\!\!\!\P{\Zm_n(j_\ell)=m_\ell\text{ for all }\ell\in[k]}\\
		={}&o\Big(n^{-\beta}\prod_{\ell=1}^k p_{m_\ell}\Big),
		\ea \ee 
		for some $\beta>0$ in the case that $m_\ell=c_\ell \log n+o(\log n)$ with $c_\ell\in[1/\log \theta, \theta/(\theta-1))$ for all $\ell\in[k]$ as well. 
		
		When the $m_\ell$ do not all behave the same, that is, for some $\ell\in[k]$ $c_\ell\in[0,1/\log \theta)$ and for some $c_\ell\in[1/\log \theta, \theta/(\theta-1))$, we can use a combination of the approaches outlined for either of the cases, which concludes the proof.
	\end{proof}
	
	We now discuss the required adaptations so that the proof holds for the model with \emph{random out-degree} as well. This follows from the fact that Lemma~\ref{lemma:in0eps} holds for this model, together with the fact that the degrees are still negatively quadrant dependent when the out-degree is random. As a result, all probabilities related to the degrees of multiple vertices can either be dealt with using Lemma~\ref{lemma:in0eps} or can be split as a product of probabilities of individual vertices. From the perspective of the in-degree of an individual vertex $i\in[n]$, the model with out-degree one and the model with random out-degree are equivalent, as in every step the in-degree of vertex $i$ increases by one with the same probability. Hence, the proof follows through in exactly the same way and Lemma~\ref{lemma:inepsterm} holds for the model with random out-degree as well.
	
	\subsection{Proof of Proposition~\ref{lemma:degprobasymp}}\label{sec:proof5.1}
	
	We finally prove Proposition~\ref{lemma:degprobasymp}, using Lemmas~\ref{lemma:in0eps} and~\ref{lemma:inepsterm}. We remark that the proof does not use that the out-degree is deterministic but uses~\eqref{eq:degdist} only, so that the proof and hence~\eqref{eq:degtail} hold for both the model with fixed and \emph{random out-degree}.
	
	\begin{proof}[Proof of Proposition~\ref{lemma:degprobasymp}]
		As discussed before,~\eqref{eq:degdist} directly follows from~\eqref{eq:splitsum2} combined with Lemmas~\ref{lemma:in0eps} and~\ref{lemma:inepsterm}. Using~\eqref{eq:degdist}, we then prove~\eqref{eq:degtail}. For ease of writing, we recall that
		\be 
		p_k:=\E{\frac{\theta-1}{\theta-1+W}\Big(\frac{W}{\theta-1+W}\Big)^k},\quad\text{and} \quad \; p_{\geq k}:=\E{\Big(\frac{W}{\theta-1+W}\Big)^k}.
		\ee  
		We start by assuming that $m_\ell=c_\ell \log n+o(\log n)$ with $c_\ell\in(0,c)$ for each $\ell\in[k]$. We discuss how to adjust the proof when $m_\ell=o(\log n)$ for some or all $\ell\in[k]$ at the end. 
		
		For each $\ell\in [k]$ take an $\eta_\ell\in (0,c-c_\ell)$ so that $\lceil (1+\eta_\ell)m_\ell\rceil <c\log n$. Then, we use the upper bound
		\be \ba\label{eq:tailub}
		\mathbb P({}&\Zm_n(v_\ell)\geq m_\ell\text{ for all }\ell\in[k])\\
		\leq{}& \sum_{j_1=m_1}^{\lfloor(1+\eta_1)m_1\rfloor}\cdots \sum_{j_k=m_k}^{\lfloor (1+\eta_k)m_k\rfloor}\P{\Zm_n(v_\ell)=j_\ell\text{ for all }\ell\in[k]}\\
		&+\sum_{i=1}^k \P{\Zm_n(v_i)\geq \lceil (1+\eta_i)m_i\rceil, \Zm_n(v_\ell)\geq m_\ell\text{ for all }\ell \neq i}.
		\ea \ee 
		We first discuss the first term on the right-hand side. As~\eqref{eq:degdist} holds uniformly in $m_1,\ldots, m_k$ $<c\log n$, we find
		\be \ba\label{eq:degtailub}
		\sum_{j_1=m_1}^{\lfloor(1+\eta_1)m_1\rfloor}\!\!\!\!\cdots\!\!\!\! \sum_{j_k=m_k}^{\lfloor (1+\eta_k)m_k\rfloor}\!\!\!\!\!\!\!\P{\Zm_n(v_\ell)=j_\ell\text{ for all }\ell\in[k]}&=\!\!\!\sum_{j_1=m_1}^{\lfloor(1+\eta_1)m_1\rfloor}\!\!\!\!\cdots\!\!\!\! \sum_{j_k=m_k}^{\lfloor (1+\eta_k)m_k\rfloor}\!\prod_{\ell=1}^kp_{j_\ell}\big(1+o\big(n^{-\beta}\big)\big)\\
		&=\prod_{\ell=1}^k \big(p_{\geq m_\ell}-p_{\geq \lceil (1+\eta_\ell)m_\ell \rceil}\big)\big(1+o\big(n^{-\beta}\big)\big)\\
		&\leq \prod_{\ell=1}^k p_{\geq m_\ell}\big(1+o\big(n^{-\beta}\big)\big).
		\ea\ee
		To finish the upper bound, it remains to show that the term on the second line of~\eqref{eq:tailub} can be incorporated in the $o\big(n^{-\beta}\big)$ term, and it suffices to show this can be done for each term in the sum, independent of the value of $i$. Using the negative quadrant dependence of the degrees under the conditional probability measure $\mathbb P_W$ (see Remark~\ref{rem:degtail} and~\cite[Lemma $7.1$]{LodOrt21}), we find 
		\be \ba \label{eq:degtailend}
		\mathbb P(\Zm_n(v_i){}&\geq \lceil (1+\eta_i)m_i\rceil, \Zm_n(v_\ell)\geq m_\ell\text{ for all } \ell \in [k]\backslash \{i\})\\
		={}&\frac{1}{(n)_k}\sum_{1\leq j_1\neq \ldots \neq j_k\leq n}\mathbb E\Big[\Pf{\Zm_n(j_i)\geq \lceil (1+\eta_i)m_i\rceil}\prod_{\ell\in [k]\backslash \{i\}}\Pf{\Zm_n(j_\ell)\geq m_\ell}\Big]\\
		={}&\frac{1}{(n)_k}\sum_{\textbf{j}\in I_n(\eps)}\mathbb E\Big[\Pf{\Zm_n(j_i)\geq \lceil (1+\eta_i)m_i\rceil}\prod_{\ell\in [k]\backslash \{i\}}\Pf{\Zm_n(j_\ell)\geq m_\ell}\Big]\\
		&+\frac{1}{(n)_k}\sum_{n^{\eps}\leq j_1\neq \ldots \neq j_k\leq n}\!\!\!\!\!\!\!\!\mathbb E\Big[\Pf{\Zm_n(j_i)\geq \lceil (1+\eta_i)m_i\rceil}\prod_{\ell\in [k]\backslash \{i\}}\Pf{\Zm_n(j_\ell)\geq m_\ell}\Big].
		\ea \ee 
		The first term in the last step can be included in the little $o$ term in~\eqref{eq:degtailub} (even when considering $m_i$ rather than $\lceil (1+\eta_i)m_i\rceil$ in the probability), as follows from computations similar to the ones in~\eqref{eq:errorterm} through~\eqref{eq:finboundineps}, combined with Remark~\ref{rem:pgeqk} (which states that $p_{\geq k}=\mathcal O(p_k)$). It remains to show that the same holds for the second term in the last step. Again, we use an argument similar to the steps performed in~\eqref{eq:detbound} through~\eqref{eq:tailfin} to arrive at 
		\be \ba \label{eq:sumbound}
		\frac{1}{(n)_k}{}&\sum_{n^{\eps}\leq j_1\neq \ldots \neq j_k\leq n}\mathbb E\Big[\Pf{\Zm_n(j_i)\geq \lceil (1+\eta_i)m_i\rceil}\prod_{\ell\in [k]\backslash \{i\}}\Pf{\Zm_n(j_\ell)\geq m_\ell}\Big]\\
		&\leq Kp_{\geq \lceil (1+\eta_i)m_i\rceil}\prod_{\ell\in [k]\backslash \{i\}}p_{\geq m_\ell},
		\ea \ee 
		for some positive constant $K$. By Lemma~\ref{lemma:pkbound} we have the inequalities 
		\be 
		p_{\geq \lceil (1+\eta_i)m_i\rceil}\leq \theta^{-\lceil (1+\eta_i)m_i\rceil}\leq \theta^{-m_i}\theta^{-\eta_im_i},\quad\text{and}\quad p_{\geq m_i}\geq (\theta+\xi)^{-m_i},
		\ee 
		for any $\xi>0$. As a result, taking $\xi\in(0,\theta(\theta^{\eta_i}-1))$ and setting $\phi_i:=1-(1+\xi/\theta)\theta^{-\eta_i}>0$, we obtain
		\be \label{eq:expcomp}
		p_{\geq \lceil (1+\eta_i)m_i\rceil}\leq (\theta+\xi)^{-m_i}\big((1+\xi/\theta)\theta^{-\eta_i}\big)^{m_i}\leq p_{\geq m_i}(1-\phi_i)^{m_i}.
		\ee 
		As $m_i=c_i\log n(1+o(1))$, it follows that $(1-\phi_i)^{m_i}=n^{-c_i\log (1/(1-\phi_i))(1+o(1))}$, so that
		\be 
		\frac{1}{(n)_k}\sum_{n^{\eps}\leq j_1\neq \ldots \neq j_k\leq n}\mathbb E\Big[\Pf{\Zm_n(j_i)\geq \lceil (1+\eta_i)m_i\rceil}\prod_{\ell\in [k]\backslash \{i\}}\Pf{\Zm_n(j_\ell)\geq m_\ell}\Big]
		\ee 
		can be incorporated in the little $o$ term in~\eqref{eq:degtailub} for each $i\in[k]$ when we take a \\$\beta'<\beta\wedge \min_{i\in[k]}c_i\log(1/(1-\phi_i))$. This yields
		\be \label{eq:degtailfinub}
		\P{\Zm_n(v_\ell)\geq m_\ell\text{ for all } \ell \in [k]}\leq \prod_{\ell=1}^k \E{\Big(\frac{W}{\theta-1+W}\Big)^{m_\ell}}\big(1+o\big(n^{-\beta'}\big)\big).
		\ee 
		For a lower bound, we can omit the second line of~\eqref{eq:tailub} and use~\eqref{eq:degtailub} to immediately obtain 
		\be \ba \label{eq:geqpklb}
		\P{\Zm_n(v_\ell)\geq m_\ell\text{ for all }\ell\in[k]}\geq{}& \sum_{j_1=m_1}^{\lfloor(1+\eta_1)m_1\rfloor}\cdots \sum_{j_k=m_k}^{\lfloor (1+\eta_k)m_k\rfloor}\P{\Zm_n(v_\ell)=j_\ell\text{ for all }\ell\in[k]}\\
		={}& \prod_{\ell=1}^k \big(p_{\geq m_\ell}-p_{\geq \lceil (1+\eta_\ell)m_\ell\rceil}\big)\big(1+o\big(n^{-\beta}\big)\big).
		\ea \ee 
		Again using~\eqref{eq:expcomp} yields $p_{\geq m_\ell}-p_{\geq \lceil (1+\eta_\ell)m_\ell\rceil}=p_{\geq m_\ell}\big(1+o\big(n^{-\beta'}\big)\big)$ when we set \\$\beta'<\beta\wedge \min_{i\in[k]} c_i\log(1/(1-\phi_i))$. Combined with~\eqref{eq:degtailfinub} this yields~\eqref{eq:degtail} for the case $m_\ell=c_\ell\log n+o(\log n)$ with $c_\ell\in(0,c)$ for all $\ell\in[k]$.
		
		When $c_\ell=0$ and thus $m_\ell=o(\log n)$ for some (or all) $i\in[k]$, a few minor alterations are required. Most importantly, we substitute the quantity $\lfloor(c/2)\log n\rfloor $ for $\lfloor (1+\eta_\ell)m_\ell\rfloor$ in all the steps. With this, the upper bound in~\eqref{eq:degtailub} follows directly. Similarly, the first term on the right-hand side of~\eqref{eq:degtailend} can be included in the little $o$ term in~\eqref{eq:degtailub}, as the computations required (which are similar to those in~\eqref{eq:errorterm} through~\eqref{eq:inepsk}) include the case $m_\ell=o(\log n)$. We again use the upper bound in~\eqref{eq:sumbound} to obtain 
			\be\ba  
			\frac{1}{(n)_k}{}&\sum_{n^{\eps}\leq j_1\neq \ldots \neq j_k\leq n}\mathbb E\Big[\Pf{\Zm_n(j_i)\geq \lceil(c/2)\log n\rceil}\prod_{\ell\in [k]\backslash \{i\}}\Pf{\Zm_n(j_\ell)\geq m_\ell}\Big]\\
			&\leq Kp_{\geq \lceil (c/2)\log n\rceil}\prod_{\ell\in [k]\backslash \{i\}}p_{\geq m_\ell}\\
			&=K\frac{p_{\geq \lceil (c/2)\log n\rceil}}{p_{\geq m_i}}\prod_{\ell\in [k]}p_{\geq m_{\ell}}. 
			\ea\ee
			By Lemma~\ref{lemma:pkbound}, we obtain that the fraction is at most $\theta^{-(c/2)\log n}(\theta+\xi)^{m_i}\leq \theta^{-(c/4)\log n}=n^{-c\log(\theta)/4}$ for all $n$ sufficiently large, when $m_i=o(\log n)$. As a result, this term can be included in the $o(n^{-\beta})$ term in~\eqref{eq:degtailub} for any $i\in[k]$ such that $m_i=o(\log n)$, which proves the upper bound. 
			
			For the lower bound, we again substitute $\lfloor(c/2)\log n\rfloor$ for $\lfloor (1+\eta_i)m_i\rfloor$ when $m_i=o(\log n)$. In~\eqref{eq:geqpklb}, this yields a term 
			\be 
			(p_{\geq m_i}-p_{\geq\lceil (c/2)\log n\rceil})(1+o(n^{-\beta}))=p_{\geq m_i}\Big(1-\frac{p_{\geq\lceil (c/2)\log n\rceil}}{p_{\geq m_i}}\Big)(1+o(n^{-\beta})).
			\ee 
			Again, the fraction is at most $n^{-\eta}$ for some $\eta>0$, which proves the lower bound and concludes the proof.
	\end{proof}
	
	\section{Proofs of the main theorems}\label{sec:mainproof}
	With the tools developed in Section~\ref{sec:taildeg}, in particular Propositions~\ref{lemma:degprobasymp} and~\ref{prop:factmean} and Lemma~\ref{lemma:maxdegwhp}, we now prove the main results formulated in Section~\ref{sec:results}. 
	
	First, we prove the main result for high degree vertices when the vertex-weight distribution has an atom at one, as in the~\ref{ass:weightatom} case.
	
	\begin{proof}[Proof of Theorem~\ref{thrm:mainatom}]
		The proof follows the same approach as~\cite[Theorem $1.2$]{AddEsl18}. For an integer subsequence $(n_\ell)_{\ell\in\N}$ such that $\eps_{n_\ell}\to\eps$ as $\ell\to\infty$, it suffices to prove that for any $i<i'\in \Z$,			
		\be 
		(X^{(n_\ell)}_i,X^{(n_\ell)}_{i+1},\ldots,X^{(n_\ell)}_{i'-1},X^{(n_\ell)}_{\geq i'})\toindis (\mathcal P^\eps(i),\mathcal P^\eps(i+1),\ldots,\mathcal P^\eps (i'-1),\mathcal P^\eps([i',\infty)))\qquad\text{as }\ell\to\infty
		\ee 
		holds. We obtain this via the convergence of the factorial moments of $X^{(n_\ell)}_i,\ldots,X^{(n_\ell)}_{i'-1},X^{(n_\ell)}_{\geq i'}$. Recall $r_k$ from~\eqref{eq:ek}. By Proposition~\ref{prop:factmean}, for any non-negative integers $a_i,\ldots,a_{i'}$, 
		\be \ba
		\E{\Big(X^{(n_\ell)}_{\geq i'}\Big)_{a_{i'}}\prod_{k=i}^{i'-1}\Big(X_k^{(n_\ell)}\Big)_{a_k}}={}&\Big(q_0\theta^{-i'+\eps_{n_\ell}}\Big)^{a_{i'}}\prod_{k=i}^{i'-1}\Big(q_0(1-\theta^{-1})\theta^{-k+\eps_{n_\ell}}\Big)^{a_k}\\
		&\times\big(1+\mathcal O\big(r_{\lfloor \log_\theta n_\ell\rfloor +i}\vee n_\ell^{-\beta}\big)\big)\\
		\to{}& \Big(q_0\theta^{-i'+\eps}\Big)^{a_{i'}}\prod_{k=i}^{i'-1}\Big(q_0(1-\theta^{-1})\theta^{-k+\eps}\Big)^{a_k},
		\ea\ee 
		as $\ell\to\infty$. By using the properties of Poisson processes, it follows that the limit equals
		\be 
		\E{\big(\mathcal P^\eps[i',\infty)\big)_{a_{i'}}\prod_{k=i}^{i'-1}\big(\mathcal P^\eps(k)\big)_{a_k}},
		\ee 
		due to the particular form of the intensity measure of the Poisson process $\mathcal P$ (which is used in the definition of the Poisson process $\mathcal P^\eps$). The result then follows from~\cite[Theorem $6.10$]{JanLucRuc11}.
	\end{proof}
	
	For the results for the~\ref{ass:weightweibull} and~\ref{ass:weightgumbel} cases, as outlined in Theorems~\ref{thrm:mainweibull} and~\ref{thrm:maingumbel}, respectively, we combine the asymptotic behaviour of $p_{\geq k}$ in Theorem~\ref{thrm:pkasymp} with Proposition~\ref{lemma:degprobasymp} and Lemma~\ref{lemma:maxdegwhp}. 
	
	\begin{proof}[Proof of Theorem~\ref{thrm:mainweibull}]
		To establish the convergence in probability, it follows from Lemma~\ref{lemma:maxdegwhp} that we need only consider $n\P{\Zm_n(v_1)\geq k_n}$ for some adequate integer-valued $k_n$ such that $k_n<c\log n$ for some $c\in(0,\theta/(\theta-1))$ and where $v_1$ is a vertex selected from $[n]$ uniformly at random. By Proposition~\ref{lemma:degprobasymp}, this quantity equals $np_{\geq k_n}(1+o(1))$. Then, we use Theorem~\ref{thrm:pkasymp} and Remark~\ref{rem:pgeqk} to obtain that, when $W$ satisfies the~\ref{ass:weightweibull} case in Assumption~\ref{ass:weights}, this quantity is at most 
		\be 
		n\overline C\, \overline L(k_n)k_n^{-(\alpha-1)}\theta^{-k_n},
		\ee 
		where $\overline C>1$ is a constant. Now fix an arbitrary $\eta>0$ and set $k_n:=\lfloor \log_\theta n-(\alpha-1)(1-\eta) \log_\theta\log_\theta n\rfloor$. This yields
		\be \ba\label{eq:weibullub}
		n\overline C{}&\,\overline L(\log_\theta n(1+o(1)))(\log_\theta n)^{-(\alpha-1)}\theta^{-\lfloor \log_\theta n-(\alpha-1)(1-\eta) \log_\theta\log_\theta n\rfloor}(1+o(1)) \\
		&\leq \overline C_2 \overline L(\log_\theta n)(\log_\theta n)^{-(\alpha-1)}(\log_\theta n)^{(\alpha-1)(1-\eta)}\\
		&=\overline  C_2\overline L(\log_\theta n)(\log _\theta n)^{-(\alpha-1)\eta}.
		\ea \ee 
		Here, $\overline C_2>0$ is a suitable constant and we use that $k_n=\log_\theta n(1+o(1))$ in the first step. Furthermore, we use~\cite[Theorem $1.5.2$]{BinGolTeu87}, which states that for a slowly-varying function $\overline L$, $\overline L(\lambda x)/\overline L(x)$ converges to one uniformly in $\lambda$ on any interval $[a,b]$ such that $0<a\leq b<\infty$. As a result, $\overline L(\log_\theta n(1+o(1)))\leq c\overline  L(\log_\theta n)$ for some constant $c>1$ and $n$ sufficiently large. Finally, we use~\cite[Proposition $1.3.6$ $(v)$]{BinGolTeu87} to obtain that for any $\eta>0$, the final line of~\eqref{eq:weibullub} tends to zero with $n$. This shows that for any $\eta>0$, with high probability,
		\be 
		\max_{\jnn}\frac{\znj-\log_\theta n}{\log_\theta\log_\theta n}\leq -(\alpha-1)(1-\eta)
		\ee 
		holds, due to the first result in Lemma~\ref{lemma:maxdegwhp}. A similar approach, when setting $k_n:=\lfloor \log_\theta n-(\alpha-1)(1+\eta)\log_\theta \log_\theta n\rfloor$, yields 
		\be 
		n\P{\Zm_n(v_1)\geq k_n}\to\infty,
		\ee 
		so that for any $\eta>0$, with high probability,
		\be 
		\max_{\jnn}\frac{\znj-\log_\theta n}{\log_\theta\log_\theta n}\geq -(\alpha-1)(1+\eta)
		\ee 
		holds. Together, these two bounds prove the desired result.
	\end{proof}
	
	\begin{proof}[Proof of Theorem~\ref{thrm:maingumbel}]
		The proof of this theorem follows a similar approach to the proof of Theorem~\ref{thrm:mainweibull}. That is, we again apply the results from Theorem~\ref{thrm:pkasymp} together with the fact that 
		\be 
		n\P{\Zm_n(v_1)\geq k_n}=np_{\geq k_n}(1+o(1)),
		\ee 
		for some adequate integer-valued $k_n$ such that $k_n<c\log n$ for some $c\in(0,\theta/(\theta-1))$, as follows from Proposition~\ref{lemma:degprobasymp} and Lemma~\ref{lemma:maxdegwhp}. In the~\ref{ass:weightgumbel}-\ref{ass:weighttaufin} sub-case, we know that
		\be \ba\label{eq:pkgumblb}
		p_{\geq k_n}=\exp\Big(-\frac{\tau^\gamma}{1-\gamma}\Big(\frac{(1-\theta^{-1})k_n}{c_1}\Big)^{1-\gamma}(1+o(1))\Big)\theta^{-k_n},
		\ea\ee 
		where we recall that $\gamma=1/(\tau+1)$. To prove the desired results, we first set $k_n=\lfloor \log_\theta n-(1+\eta)C_{\theta,\tau,c_1}(\log_\theta n)^{1-\gamma}\rfloor $ for any $\eta>0$, where we recall $C_{\theta,\tau,c_1}$ from~\eqref{eq:gumbrv2nd}. Using this in~\eqref{eq:pkgumblb} then yields 
		\be \ba
		np_{\geq k_n}&= \frac{n}{\theta^{k_n}}\e^{-\log (\theta) C_{\theta,\tau,c_1}k_n^{1-\gamma}(1+o(1))}\geq \e^{\eta\log (\theta) C_{\theta,\tau,c_1}(\log_\theta n)^{1-\gamma}(1+o(1))},
		\ea\ee 
		where we use that $k_n^{1-\gamma}=(\log_\theta n)^{1-\gamma}(1+o(1))$ in the last step.	Hence, $n\P{\Zm_n(v_1)\geq k_n}$ diverges. We thus conclude from Lemma~\ref{lemma:maxdegwhp} that 
		\be 
		\max_{\jnn}\frac{\znj-\log_\theta n}{(\log_\theta n)^{1-\gamma}}\geq -(1+\eta)C_{\theta,\tau,c_1}
		\ee  
		holds with high probability. A similar approach, setting \\ $k_n:=\lceil  \log_\theta n-(1-\eta)C_{\theta,\tau,c_1}(\log_\theta n)^{1-\gamma}\rceil$ and combining this with the first result of Lemma~\ref{lemma:maxdegwhp} yields
		\be 
		\max_{\jnn}\frac{\znj-\log_\theta n}{(\log_\theta n)^{1-\gamma}}\leq -(1-\eta)C_{\theta,\tau,c_1}
		\ee 
		holds with high probability. Together, these two bounds prove~\eqref{eq:gumbrv2nd}.
		
		To prove~\eqref{eq:gumbrav2nd} we apply the same methodology but use the asymptotic expression of $p_k$ (and $p_{\geq k}$ by adjusting constants), as in~\eqref{eq:pkbddgumbelrav}. We recall the constants $C_1,C_2,C_3$ from~\eqref{eq:c123} and set $k_n=:\lceil \log_\theta n-C_1(\log_\theta \log_\theta n)^\tau+C_2(\log_\theta \log_\theta n)^{\tau-1}\log_\theta\log_\theta \log_\theta n+(C_3+\eta)(\log_\theta\log_\theta n)^{\tau-1}\rceil$, for any $\eta>0$. Then,~\eqref{eq:pkbddgumbelrav} yields
		\be 
		np_{\geq k_n}= \frac{n}{\theta^{k_n}} \exp\Big(-\Big(\frac{\log k_n}{c_1}\Big)^\tau\Big(1+\frac{\tau(\tau-1)\log\log k_n}{\log k_n}-\frac{K_{\tau,c_1,\theta}}{\log k_n}(1+o(1))\Big)\Big).
		\ee 
		Using Taylor expansions, we obtain
		\be \ba
		-\Big(\frac{\log k_n}{c_1}\Big)^\tau={}&-\Big(\frac{\log \log_\theta n}{c_1}\Big)^\tau +o(1)=-\log (\theta) C_1(\log_\theta \log_\theta n)^\tau+o(1),\\
		\frac{\tau(\tau-1)}{c_1^\tau}(\log k_n)^{\tau-1}\log\log k_n={}&\frac{\tau(\tau-1)}{c_1^\tau}(\log\log_\theta n)^{\tau-1}\log\log\log_\theta n+o(1)\\
		={}&\log(\theta )C_2(\log_\theta\log_\theta n)^{\tau-1}\log_\theta\log_\theta\log_\theta n\\
		&+ (\log_\theta(\log\theta))(\log\theta)^\tau\frac{\tau(\tau-1)}{c_1^\tau}(\log_\theta\log_\theta n)^{\tau-1}+o(1),\\
		-\frac{K_{\tau,c_1,\theta}}{c_1^\tau}(\log k_n)^{\tau-1}={}&-\frac{K_{\tau,c_1,\theta}}{c_1^\tau}(\log\log_\theta n)^{\tau-1}+o(1)\\
		={}&-(\log \theta)^{\tau-1}\frac{K_{\tau,c_1,\theta}}{c_1^\tau}( \log_\theta\log_\theta n)^{\tau-1}+o(1),
		\ea \ee 
		where we recall $K_{\tau,c_1,\theta}$ from~\eqref{eq:pkbddgumbelrav} in Theorem~\ref{thrm:pkasymp}. As a result, 
		\be \ba 
		n{}&\exp\Big(-\Big(\frac{\log k_n}{c_1}\Big)^\tau\!\!+\frac{\tau(\tau-1)}{c_1^\tau}(\log k_n)^{\tau-1}\log\log k_n-\frac{K_{\tau,c_1,\theta}}{c_1^{\tau}}(\log k_n)^{\tau-1}(1+o(1))\Big)\theta^{-k_n}\\
		&= n\exp\big(-\log (\theta) C_1(\log_\theta\log_\theta n)^\tau+\log(\theta )C_2 (\log_\theta \log_\theta n)^{\tau-1}\log_\theta\log_\theta\log_\theta n\\
		&\hspace{1.8cm}+\log( \theta) C_3 (\log_\theta\log_\theta n)^{\tau-1}(1+o(1))\big)\theta^{-k_n}\\
		&\leq \exp\big(-(\eta-o(1)) (\log_\theta\log_\theta n)^{\tau-1}\big),
		\ea\ee 
		where in the last step we use that $k_n\geq \log_\theta n-C_1(\log_\theta \log_\theta n)^\tau+C_2(\log_\theta \log_\theta n)^{\tau-1}\log_\theta\log_\theta \log_\theta n+(C_3+\eta)(\log_\theta\log_\theta n)^{\tau-1}$. As the right-hand side tends to zero with $n$, Lemma~\ref{lemma:maxdegwhp} yields for any fixed $\eta>0$, with high probability,
		\be
		\max_{\inn}\frac{\zni-\big(\log_\theta n-C_1(\log_\theta\log_\theta n)^\tau+C_2(\log_\theta \log_\theta n)^{\tau-1}\log_\theta \log_\theta\log_\theta n\big)}{(\log_\theta\log_\theta n)^{\tau-1}}\leq C_3+\eta.
		\ee 
		With a similar approach, setting
		\be 
		k_n=\lfloor \log_\theta n-C_1(\log_\theta \log_\theta n)^\tau+C_2(\log_\theta \log_\theta n)^{\tau-1}\log_\theta\log_\theta \log_\theta n+(C_3-\eta)(\log_\theta\log_\theta n)^{\tau-1}\rfloor, 
		\ee
		we can obtain that for any fixed $\eta>0$, with high probability,
		\be 
		\max_{\inn}\frac{\zni-\big(\log_\theta n-C_1(\log_\theta\log_\theta n)^\tau+C_2(\log_\theta \log_\theta n)^{\tau-1}\log_\theta \log_\theta\log_\theta n\big)}{(\log_\theta\log_\theta n)^{\tau-1}}\geq C_3-\eta.
		\ee  
		Together these two bounds yield~\eqref{eq:gumbrav2nd}, which concludes the proof.
	\end{proof}
	
	\begin{proof}[Proof of Theorem~\ref{thrm:maxtail}]
		We first prove the asymptotic distribution of the maximum degree, whose proof follows the same approach as the proof of~\cite[Theorem $1.3$]{AddEsl18}. We need to consider two cases: $i_n=\mathcal O(1)$ and $i_n\to\infty$ such that $i_n+\log_\theta n <(\theta/(\theta-1))\log n$ and $\liminf_{n\to\infty}i_n >-\infty$. For the former case, as $\exp(-q_0\theta^{-i_n+\eps_n})=\mathcal O(1)$, it suffices to prove
		\be 
		\P{\max_{\jnn}\znj\geq \lfloor \log_\theta n\rfloor +i_n}-(1-\exp(-q_0\theta^{-i_n+\eps_n}))\to 0\qquad \text{as }n\to\infty.
		\ee  
		By the definition of $X^{(n)}_{\geq i}$ in~\eqref{eq:xni}, this is equivalent to 
		\be \label{eq:limconv}
		\P{X^{(n)}_{\geq i_n}=0}-\exp(-q_0\theta^{-i_n+\eps_n})\to0\qquad \text{as }n\to\infty.
		\ee 
		This follows from Theorem~\ref{thrm:mainatom} and the subsubsequence principle. That is, if we assume the convergence in~\eqref{eq:limconv} does not hold, then there exists a subsequence $(n_\ell)_{\ell\in\N}$ and a $\delta>0$ such that 
		\be \label{eq:subseq}
		\P{X^{(n_\ell)}_{\geq i_{n_\ell}}=0}-\exp(-q_0\theta^{-i_{n_\ell}+\eps_{n_\ell}})>\delta\ \ \forall\,\ell\in\N.
		\ee 
		However, as $\eps_{n_\ell}$ is bounded, there exists a subsubsequence $\eps_{n_{\ell_k}}$ such that $\eps_{n_{\ell_k}}\to\eps$ for some $\eps\in[0,1]$. Then, by Theorem~\ref{thrm:mainatom}, the statement in~\eqref{eq:subseq} is false, from which the result follows.
		
		In the latter case, we need only consider $\E{X^{(n)}_{\geq i_n}}$ and $\E{\big(X^{(n)}_{\geq i_n}\big)_2}$, as by~\cite[Corollary $1.11$]{Bol01},
		\be \label{eq:probbounds}
		\E{X^{(n)}_{\geq i_n}}-\frac{1}{2}\E{\big(X^{(n)}_{\geq i_n}\big)_2}\leq \P{X^{(n)}_{\geq i_n}>0}\leq \E{X^{(n)}_{\geq i_n}}.
		\ee 
		By Proposition~\ref{prop:factmean}, we have that 
		\be 
		\E{X^{(n)}_{\geq i_n}}=q_0\theta^{-i_n+\eps_n}(1+o(1)),\qquad \E{\big(X^{(n)}_{\geq i_n}\big)_2}=\big(q_0\theta^{-i_n+\eps_n}\big)^2(1+o(1)).
		\ee 
		Hence, as $i_n\to\infty$ and $\eps_n$ is bounded,
		\be \ba
		\E{X^{(n)}_{\geq i_n}}&=\big(1-\exp\big(-q_0\theta^{-i_n+\eps_n}\big)\big)(1+o(1)),\\ \E{X^{(n)}_{\geq i_n}}-\frac12 \E{\big(X^{(n)}_{\geq i_n}\big)_2}&=\big(1-\exp\big(-q_0\theta^{-i_n+\eps_n}\big)\big)(1+o(1)).
		\ea\ee 
		Combining this with~\eqref{eq:probbounds} yields the desired result.
		
		Recall $\eps_n:=\log_\theta n-\lfloor \log_\theta n\rfloor$. We now prove the limiting distribution of $|\M_{n_\ell}|$ has the desired distribution, as in~\eqref{eq:meps}, when the subsequence $(n_\ell)_{\ell\in\N}$ is such that $\eps_{n_\ell}\to \eps\in[0,1]$, which follows the same approach as~\cite[Theorem $1.1$]{Esl16}. Consider the event $\mathcal E_{j,k}:=\{X^{(n_\ell)}_j=k, X^{(n_\ell)}_{\geq j+1}=0\}$ for $j\in\Z$ and $k\in\N$. $\mathcal E_{j,k}$ implies that there are exactly $k$ vertices which attain the maximum degree $\lfloor \log_\theta n\rfloor+j$ in the tree $T_{n_\ell}$. We observe that the events $\mathcal E_{j,k}$ are pairwise disjoint for different values of $j$. As a result, we obtain the following inequalities: For any $M\in\N$,
		\be\ba 
		\P{|\M_{n_\ell}|=k}=\mathbb P\bigg(\bigcup_{j\in\Z}\mathcal E_{j,k}\bigg)&\leq \mathbb P\bigg(\bigcup_{j\leq -(M+1)} \mathcal E_{j,k}\bigg)+\sum_{j=-M}^{M-1}\P{\mathcal E_{j,k}}+\mathbb P\bigg(\bigcup_{j\geq M}\mathcal E_{j,k}\bigg)\\
		&\leq \P{X^{(n_\ell)}_{\geq -M}=0}+\sum_{j=-M}^{M-1}\P{\mathcal E_{j,k}}+\P{X^{(n_\ell)}_{\geq M}>0},
		\ea\ee 
		and
		\be
		\P{|\M_{n_\ell}|=k}=\mathbb P\bigg(\bigcup_{j\in\Z}\mathcal E_{j,k}\bigg)\geq \sum_{j=-M}^{M-1}\P{\mathcal E_{j,k}}.
		\ee 
		By Theorem~\ref{thrm:mainatom} it thus follows that 
		\be \ba 
		\limsup_{n_\ell\to\infty}\P{|\M_{n_\ell}|=k}&\leq \liminf_{M\in\N}\bigg(\P{X_{\geq -M}^\eps=0}+\sum_{j=-M}^{M-1}\P{X_j^\eps=k,X^\eps_{\geq j}=0}+\P{X^\eps_{\geq M}>0}\bigg),\\
		\liminf_{n_\ell\to\infty}\P{|\M_{n_\ell}|=k}&\geq \limsup_{M\in\N}\sum_{j=-M}^M \P{X^\eps_j=k,X^\eps_{\geq j+1}=0},
		\ea\ee 
		where $X_j^\eps\sim \text{Poi}\big(q_0(1-\theta^{-1})\theta^{-j+\eps}\big)$ and $X^\eps_{\geq j}\sim\text{Poi}\big(q_0\theta^{-j+\eps}\big)$ are independent Poisson random variables, for $j\in\Z$. As a result, 
		\be \ba
		\P{X_{\geq -M}^\eps=0}&=\e^{-q_0\theta^{M+\eps}}, \qquad \P{X^\eps_{\geq M}>0}=1-\e^{-q_0\theta^{-M+\eps}}, \\
		\P{X^\eps_j=k, X^\eps_{\geq j+1}=0}&=\P{X^\eps_j=k}\P{X^\eps_{\geq j+1}=0}=\frac{1}{k!}\big(q_0(1-\theta^{-1})\theta^{-j+\eps}\big)^k\e^{-q_0\theta^{-j+\eps}}.
		\ea\ee 
		Hence, we obtain 
		\be \ba
		\limsup_{n_\ell\to\infty}\P{|\M_{n_\ell}|=k}&\leq \liminf_{M\in\N}\bigg(\e^{-q_0\theta^{M+\eps}}\!\!+\!\sum_{j=-M}^{M-1}\!\frac{1}{k!}\big(q_0(1-\theta^{-1})\theta^{-j+\eps}\big)^k\e^{-q_0\theta^{-j+\eps}}\!+1-\e^{-q_0\theta^{-M+\eps}}\bigg)\\
		&\leq\sum_{j\in\Z}\frac{1}{k!}\big(q_0(1-\theta^{-1})\theta^{-j+\eps}\big)^k\e^{-q_0\theta^{-j+\eps}},\\
		\liminf_{n_\ell\to\infty}\P{|\M_{n_\ell}|=k}&\geq\limsup_{M\in\N}\sum_{j=-M}^M\frac{1}{k!}\big(q_0(1-\theta^{-1})\theta^{-j+\eps}\big)^k\e^{-q_0\theta^{-j+\eps}}\\
		&=\sum_{j\in\Z}\frac{1}{k!}\big(q_0(1-\theta^{-1})\theta^{-j+\eps}\big)^k\e^{-q_0\theta^{-j+\eps}}.
		\ea \ee 
		It is then readily checked that the limit is indeed finite and that summing over all $k\in\N$ yields one, which concludes the proof.
	\end{proof}
	
	\begin{proof}[Proof of Theorem~\ref{thrm:asympnormal}]
		The proof follows the same argument as the proof of~\cite[Theorem $1.4$]{AddEsl18}, which is based on~\cite[Theorem $1.24$]{Bol01}. Let $1\leq a\leq b$ be integers. Then, by Proposition~\ref{prop:factmean} and since $i_n=o(\log n)$,
		\be 
		\E{\Big(X^{(n)}_{i_n}\Big)_a}-\Big(q_0(1-\theta^{-1})\theta^{-i_n+\eps_n}\Big)^a=\mathcal O\big(\theta^{-i_na}\big(r_{\lfloor \log_\theta n+i_n\rfloor}\vee n^{-\beta}\big)\big).
		\ee 
		It then remains to show that the right-hand side is in fact $o(\theta^{i_nb})$. We note that $i_n=o(\log r_{\log_\theta n}\wedge \log n)$, so that we can write the right-hand side as
		\be 
		\mathcal O\big((r_{\log_\theta n})^{1-i_na\log \theta/\log r_{\log_\theta n}}\vee n^{-\beta -i_na\log\theta/\log n}\big)=\mathcal O\big((r_{\log_\theta n})^{1-o(1)}\vee n^{-\beta -o(1)}\big) =o(\theta^{i_nb}), 
		\ee 
		by the constraints on $i_n$, from which the result follows.
	\end{proof} 
	
	\section{Technical details of examples}\label{sec:ex}
	
	In this section we discuss some technical details of the examples discussed in Section \ref{sec:exres}. In particular, for each example we provide a precise asymptotic expression of $p_k$ and $p_{\geq k}$ as well as a key element that leads to the results in Section \ref{sec:exres}. That is, for each of the examples we state and prove the analogue of Proposition~\ref{prop:factmean}. The three theorems presented in each of the examples in Section \ref{sec:exres} mimic three of the theorems presented in Section~\ref{sec:results}. That is, Theorems~\ref{thrm:betappp} and~\ref{thrm:gumbppp} are the analogue of Theorems~\ref{thrm:mainatom}, Theorems~\ref{thrm:betamaxtail} and~\ref{thrm:gumbmaxtail} are the analogue of Theorem~\ref{thrm:maxtail}, and Theorems~\ref{thrm:betaasympnorm} and~\ref{thrm:gumbasympnorm} are the analogue of Theorem~\ref{thrm:asympnormal}. As a result, their proofs are very similar to the proofs of Theorems~\ref{thrm:mainatom},~\ref{thrm:maxtail}, and~\ref{thrm:asympnormal},  so we omit proving the theorems stated in Section~\ref{sec:exres}.
	
	\subsection{Example \ref{ex:beta}, beta distribution}
	Let $W$ be beta distributed, i.e.\ its distribution function satisfies~\eqref{eq:betacdf} for some $\alpha,\beta>0$. We prove a result on (the tail of) the limiting degree distribution and provide additional building blocks required to prove the results in Example \ref{ex:beta}.		 
	
	\begin{lemma}\label{lemma:betapkasymp}
		Let the distribution of $W$ satisfy~\eqref{eq:betacdf} for some $\alpha,\beta>0$ and recall $p_k,p_{\geq k}$ from~\eqref{eq:pk}. Then,
		\be\ba \label{eq:betapk}
		p_k&= \frac{\Gamma(\alpha+\beta)}{\Gamma(\alpha)}(1-\theta^{-1})^{1-\beta}k^{-\beta}\theta^{-k}\big(1+\mathcal O(1/k)\big),\\ 
		p_{\geq k}&=\frac{\Gamma(\alpha+\beta)}{\Gamma(\alpha)} (1-\theta^{-1})^{-\beta} k^{-\beta}\theta ^{-k}\big(1+\mathcal O(1/k)\big).
		\ea \ee 
	\end{lemma} 
	
	Note that this lemma improves on the bounds in~\eqref{eq:pkbddweibull} by providing a precise multiplicative constant, rather than two slowly-varying functions that are (possibly) of different order.
	
	\begin{proof}
		By the distribution of $W$, we immediately obtain that
		\be\ba 
		p_k&=\int_0^1 (\theta-1)x^k (\theta-1+x)^{-(k+1)}\frac{\Gamma(\alpha+\beta)}{\Gamma(\alpha)\Gamma(\beta)}x^{\alpha-1}(1-x)^{\beta-1}\,\d x\\
		&=(\theta-1)^{-k}\frac{\Gamma(\alpha+\beta)}{\Gamma(\alpha)\Gamma(\beta)}\int_0^1 x^{k+\alpha-1}(1-x)^{\beta-1}(1+x/(\theta-1))^{-(k+1)}\,\d x,
		\ea\ee 
		We now use Euler's integral representation of the hypergeometric function. That is, for $a,b,c,z\in\mathbb C$ such that $\mathrm{Re}(c)>\mathrm{Re}(b)>0$ and $z$ is not a real number greater than one,
		\be 
		\int_0^1 x^{b-1}(1-x)^{c-b-1}(1-zx)^{-a}\,\d x=\frac{\Gamma(c-b)\Gamma(b)}{\Gamma(c)} {}_2F_1(a,b,c,z),
		\ee 
		where $_2F_1$ is the hypergeometric function. We thus obtain 
		\be 
		(\theta-1)^{-k}\frac{\Gamma(\alpha+\beta)\Gamma(k+\alpha)}{\Gamma(\alpha)\Gamma(k+\alpha+\beta)} {}_2F_1(k+1,k+\alpha,k+\alpha+\beta,-1/(\theta-1)).
		\ee  
		We then use one of the Euler transformations of the hypergeometric function, 
		\be 
		_2F_1(a,b,c,z)=(1-z)^{c-a-b} {}_2F_1(c-a,c-b,c,z),
		\ee 
		to arrive at
		\be \label{eq:betapkexpl}
		\theta^{-k}\frac{\Gamma(\alpha+\beta)\Gamma(k+\alpha)}{\Gamma(\alpha)\Gamma(k+\alpha+\beta)} \Big(\frac{\theta}{\theta-1}\Big)^{\beta-1} {}_2F_1(\alpha+\beta-1,\beta,k+\alpha+\beta,-1/(\theta-1)).
		\ee 
		We directly find a particular case in which we can find the value of the hypergeometric function explicitly, namely when $\alpha+\beta=1$. When $\alpha+\beta=1$, we find that the hypergeometric function on the right-hand side of~\eqref{eq:betapkexpl} equals one as the first argument equals zero, independent of the other arguments, so that we arrive at
		\be 
		\frac{(1-\theta^{-1})^{1-\beta}\Gamma(k+\alpha)}{\Gamma(\alpha)\Gamma(k+1)}\theta^{-k}=\frac{(1-\theta^{-1})^{1-\beta}}{\Gamma(\alpha)}k^{-\beta}\theta^{-k}\big(1+\mathcal O(1/k)\big),
		\ee 
		since $\Gamma(x+a)/\Gamma(x)=x^a\big(1+\mathcal O(1/x)\big)$ as $x\to\infty$ and $\alpha=1-\beta$ in this particular case. When $\alpha+\beta\neq 1$, we can obtain a similar expression. First, we use one of Pfaff's transformations for the hypergeometric function, 
		\be 
		_2F_1(a,b,c,z)=(1-z)^{-b}{}_2F_1(b,c-a,c,z/(z-1)).
		\ee 
		Then, applying this to the right-hand side of~\eqref{eq:betapkexpl}, so that $z/(z-1)=1/\theta\in(-1,1)$, we can express the hypergeometric function as a power series. Namely, for $z$ such that $|z|<1$, 
		\be 
		_2F_1(a,b,c,z)=\sum_{j=0}^\infty \frac{a^{(j)}b^{(j)}}{c^{(j)}}\frac{z^j}{\Gamma(j)},
		\ee 
		where $a^{(j)}:=\prod_{\ell=1}^j (a+(\ell-1))$ (and similarly for $b^{(j)},c^{(j)}$). Thus, combining the Pfaff transformation and the power series representation yields
		\be \label{eq:alphabetaseries}
		{}_2F_1(\alpha+\beta-1,\beta,k+\alpha+\beta,-1/(\theta-1))=\Big(\frac{\theta}{\theta-1}\Big)^{-\beta}\sum_{j=0}^\infty \frac{\beta^{(j)}(k+1)^{(j)}}{(k+\alpha+\beta)^{(j)}}\frac{\theta^{-j}}{j!}.
		\ee 
		From the $\alpha+\beta=1$ case, we immediately distil that 
		\be \label{eq:betathetasum}
		\sum_{j=0}^\infty \frac{\beta^{(j)}}{j!}\theta^{-j}=\Big(\frac{\theta}{\theta-1}\Big)^\beta.
		\ee 
		The aim is to show that for $k$ large, the series in~\eqref{eq:alphabetaseries} is close to $(\theta/(\theta-1))^\beta$ for any choice of $\alpha,\beta>0$, so that the entire term in~\eqref{eq:alphabetaseries} is close to one. We consider two cases, namely $\alpha+\beta<1$ and $\alpha+\beta>1$. Let us start with the latter. We immediately obtain the upper bound $(k+\alpha+\beta)^{(j)}>(k+1)^{(j)}$, so that using~\eqref{eq:betathetasum} yields that the right-hand side of~\eqref{eq:alphabetaseries} is at most one. For a lower bound, note that 
		\be 
		\frac{(k+1)^{(j)}}{(k+\alpha+\beta)^{(j)}}=\prod_{\ell=1}^j \Big(1-\frac{\alpha+\beta-1}{k+\alpha+\beta+(\ell-1)}\Big)\geq \Big(1-\frac{\alpha+\beta-1}{k+\alpha+\beta}\Big)^j,
		\ee  
		as the fraction in the second step in decreasing in $\ell$, since $\alpha+\beta-1>0$. We thus have
		\be
		\Big(\frac{\theta}{\theta-1}\Big)^{-\beta}\sum_{j=0}^\infty \frac{\beta^{(j)}(k+1)^{(j)}}{(k+\alpha+\beta)^{(j)}}\frac{\theta^{-j}}{j!}\geq \Big(\frac{\theta}{\theta-1}\Big)^{-\beta}\sum_{j=0}^\infty \frac{\beta^{(j)}}{j!}\Big( \Big(1-\frac{\alpha+\beta-1}{k+\alpha+\beta}\Big)\frac{1}{\theta}\Big)^j,
		\ee 
		which, using~\eqref{eq:betathetasum}, equals
		\be \Big(\frac{\theta-1}{\theta-1+\frac{\alpha+\beta-1}{k+\alpha+\beta}}\Big)^\beta=\Big(1-\frac{\alpha+\beta-1}{(\theta-1)(k+\alpha+\beta)+(\alpha+\beta-1)}\Big)^\beta =1-\mathcal O(1/k).
		\ee 
		A similar approach can be used for $\alpha+\beta<1$, where one would have an elementary lower bound and an upper bound that is $1+\mathcal O(1/k)$. In total, combining the above in~\eqref{eq:alphabetaseries} and in~\eqref{eq:betapkexpl} yields
		\be 
		p_k=\frac{\Gamma(\alpha+\beta)}{\Gamma(\alpha)}(1-\theta^{-1})^{1-\beta}k^{-\beta}\theta^{-k}\big(1+\mathcal O(1/k)\big).
		\ee 
		An equivalent computation can be performed for
		\be \label{eq:int2}
		\int_0^1 x^k (\theta-1+x)^{-k}\frac{\Gamma(\alpha+\beta)}{\Gamma(\alpha)\Gamma(\beta)}x^{\alpha-1}(1-x)^{\beta-1}\,\d x,
		\ee 
		to obtain that it equals 
		\be \ba
		\theta ^{-k}{}&\frac{\Gamma(\alpha+\beta)\Gamma(k+\alpha)}{\Gamma(\alpha)\Gamma(k+\alpha+\beta)} \Big(\frac{\theta}{\theta-1}\Big)^\beta {}_2F_1(\alpha+\beta,\beta,k+\alpha+\beta,-1/(\theta-1))\\
		&=\theta ^{-k}\frac{\Gamma(\alpha+\beta)\Gamma(k+\alpha)}{\Gamma(\alpha)\Gamma(k+\alpha+\beta)} {}_2F_1(\beta,k,k+\alpha+\beta,1/\theta)\\
		&=\theta ^{-k}\frac{\Gamma(\alpha+\beta)\Gamma(k+\alpha)}{\Gamma(\alpha)\Gamma(k+\alpha+\beta)} \sum_{j=0}^\infty \frac{\beta^{(j)}k^{(j)}}{(k+\alpha+\beta)^{(j)}}\frac{\theta^{-j}}{j!}.
		\ea \ee 
		In this case an equivalent approach for bounding the sum on the right-hand side works for all $\alpha,\beta>0$. Hence, for~\eqref{eq:int2} we obtain the asymptotic expression 
		\be 
		p_{\geq k}=\frac{\Gamma(\alpha+\beta)}{\Gamma(\alpha)} (1-\theta^{-1})^{-\beta} k^{-\beta}\theta ^{-k}\big(1+\mathcal O(1/k)\big),
		\ee 
		which concludes the proof.
	\end{proof}
	
	Recall that in this example we set
	\be\ba
	X^{(n)}_i&:=|\{\jnn: \znj=\lfloor \log_\theta n-\beta\log_\theta\log_\theta n \rfloor +i\}|,\\ 
	X^{(n)}_{\geq i}&:=|\{\jnn: \znj\geq \lfloor \log_\theta n-\beta \log_\theta\log_\theta n\rfloor +i\}|,\\
	\eps_n&:=(\log_\theta n-\beta \log_\theta\log_\theta n)-\lfloor\log_\theta n-\beta\log_\theta\log_\theta n \rfloor.
	\ea\ee 
	We then state the analogue of Proposition~\ref{prop:factmean}.
	
	\begin{proposition}\label{prop:betafactmean}
		Consider the WRT model as in Definition~\ref{def:WRT} with vertex-weights $(W_i)_{\inn}$ whose distribution satisfies~\eqref{eq:betacdf} for some $\alpha,\beta>0$, and recall $c_{\alpha,\beta,\theta}$ from~\eqref{eq:epsnbeta}. For a fixed $K\in\N,c\in(0,\theta/(\theta-1))$, the following holds. For any integer-valued $i_n,i_n',$ such that $0<\log_\theta n+i_n<\log_\theta n+i_n'< c\log n$ and $i_n,i_n'= \delta\log_\theta n+o(\log n)$, for some $\delta\in(-1,c\log\theta-1)\cup\{0\}$ and for $a_{i_n},\ldots,a_{i_n'}\in\N_0$ satisfying $a_{i_n}+\cdots+a_{i_n'}=K$,
		\be\ba
		\mathbb E\bigg[\Big(X^{(n_\ell)}_{\geq i_n'}\Big)_{a_{i_n'}}\prod_{k=i_n}^{i_n'-1}\Big(X_k^{(n_\ell)}\Big)_{a_k}\bigg]={}&\Big(\frac{c_{\alpha,\beta,\theta}}{(1+\delta)^{\beta}}\theta^{-i_n'+\eps_n}\Big)^{a_{i_n'}}\prod_{k=i_n}^{i_n'-1}\Big(\frac{c_{\alpha,\beta,\theta}(1-\theta^{-1})}{(1+\delta)^{\beta}}\theta^{-k+\eps_n}\Big)^{a_k}\\
		&\times\Big(1+\mathcal O\Big(\frac{\log\log n}{\log n}\vee \frac{|i_n-\delta\log_\theta n|\vee |i_n'-\delta\log_\theta n|}{\log n}\Big)\Big).
		\ea\ee
	\end{proposition} 
	
	\begin{proof}
		Set $K':=K-a_{i_n'}$ and for each $i_n\leq k\leq i_n'$ and each $u$ such that $\sum_{\ell=i_n}^{k-1}a_\ell<u\leq \sum_{\ell=i_n}^k a_\ell$, let $m_u=\lfloor \log_\theta n-\log_\theta\log_\theta n \rfloor +k$. Also, let $(v_u)_{u\in[K]}$ be $K$ vertices selected uniformly at random without replacement from $[n]$. Then, as the $X^{(n)}_{\geq k}$ and $X^{(n)}_{k}$ can be expressed as sums of indicators, following the same steps as in the proof of Proposition~\ref{prop:factmean},
		\be\label{eq:meanexunif}
		\E{\Big(X^{(n)}_{\geq i_n'}\Big)_{a_{i_n'}}\prod_{k=i_n}^{i_n'-1}\Big(X_k^{(n)}\Big)_{a_k}}=(n)_K\sum_{\ell=0}^{K'}\sum_{\substack{S\subseteq [K']\\ |S|=\ell}}(-1)^\ell \P{\Zm_n(v_u)\geq m_u+\ind_{\{u\in S\}}\text{ for all } u\in [K]}.
		\ee
		By Proposition~\ref{lemma:degprobasymp},
		\be \label{eq:tailprobunif}
		\P{\Zm_n(v_\ell)\geq m_u+\ind_{\{u\in S\}}\text{ for all } u\in [K]}=\prod_{u=1}^K \E{\Big(\frac{W}{\E{W}+W}\Big)^{m_u+\ind_{\{u\in S\}}}}(1+o(n^{-\beta})),
		\ee 
		for some $\beta>0$. By Lemma~\ref{lemma:betapkasymp}, when $|S|=\ell$,
		\be\ba
		\prod_{u=1}^K {}&\E{\Big(\frac{W}{\E{W}+W}\Big)^{m_u+\ind_{\{u\in S\}}}}\\
		&=\Big(\frac{\Gamma(\alpha+\beta)}{\Gamma(\alpha)}(1-\theta^{-1})^{-\beta}\Big)^K \theta^{-\sum_{u=1}^K m_u-\ell}\prod_{u=1}^K  (m_u+\ind_{\{u\in S\}})^{-\beta}(1+\mathcal O (1/\log n)).
		\ea\ee 
		Here, we are able to obtain the error term $1-\mathcal O(1/\log n)$ due to the fact that $\log_\theta n+i_n>\eta \log n$ for some $\eta\in(0,1+\delta)$ when $n$ is large. Moreover, as $i_n,i_n'= \delta\log_\theta n+o(\log n)$ and thus $m_u\sim (1+\delta)\log_\theta n$ for each $u\in[K]$, 
		\be
		\prod_{u=1}^K(m_u+\ind_{\{u\in S\}})^{-\beta }=( (1+\delta)\log_\theta n)^{-\beta K}\Big(1+\mathcal O\Big(\frac{\log\log n}{\log n}\vee \frac{|i_n-\delta\log_\theta n|\vee |i_n'-\delta\log_\theta n|}{\log n}\Big)\Big),
		\ee
		uniformly in $S$ (and $\ell$). Recalling that $c_{\alpha,\beta,\theta}=(\Gamma(\alpha+\beta)/\Gamma(\alpha))(1-\theta^{-1})^{-\beta}$, we thus arrive at
		\be \ba
		(n)_K{}&\sum_{\ell=0}^{K'}\sum_{\substack{S\subseteq [K']\\ |S|=\ell}}(-1)^\ell\big(c_{\alpha,\beta,\theta}(1+\delta)^{-\beta}(\log_\theta n)^{-\beta}\big)^K \theta^{-\sum_{u=1}^K m_u-\ell}\\
		&\times \Big(1+\mathcal O\Big(\frac{\log\log n}{\log n}\vee \frac{|i_n-\delta\log_\theta n|\vee |i_n'-\delta\log_\theta n|}{\log n}\Big)\Big)\\
		={}&\big(c_{\alpha,\beta,\theta}(1+\delta)^{-\beta}\theta^{-i_n'+\eps_n}\big)^{a_{i_n'}}\prod_{k=i_n}^{i_n'-1}\!\big(c_{\alpha,\beta,\theta}(1+\delta)^{-\beta}(1-\theta^{-1})\theta^{-k+\eps_n}\big)^{a_k}\\
		&\times\Big(1+\mathcal O\Big(\frac{\log\log n}{\log n}\vee \frac{|i_n-\delta\log_\theta n|\vee |i_n'-\delta\log_\theta n|}{\log n}\Big)\Big),
		\ea\ee 
		where the last step follows from a similar argument as in the proof of Proposition~\ref{prop:factmean}.
	\end{proof}
	
	With Proposition~\ref{prop:betafactmean} at hand, the proofs of Theorems~\ref{thrm:betappp},~\ref{thrm:betamaxtail}, and~\ref{thrm:betaasympnorm} follow the same approach as the proofs of Theorems~\ref{thrm:mainatom},~\ref{thrm:maxtail}, and~\ref{thrm:asympnormal}, respectively.
	
	\subsection{Example \ref{ex:gumb}, fraction of `gamma' random variables}
	Let the distribution of $W$ satisfy~\eqref{eq:gumbex} for some $b\in\R,c_1>0$ with $bc_1\leq 1$. We prove a result on (the tail of) the limiting degree distribution and provide additional building blocks required to prove the results in Example \ref{ex:gumb}.
	
	\begin{lemma}\label{lemma:gumbpkasymp}
		Let the distribution of $W$ satisfy~\eqref{eq:gumbex} for some $b\in\R,c_1>0$ such that $bc_1\leq 1$, and recall $p_k,p_{\geq k},$ and $C$ from~\eqref{eq:pk} and~\eqref{eq:c}, respectively. Then,
		\be\ba \label{eq:gumbpk}
		p_k&=(1-\theta^{-1})Ck^{b/2+1/4}\e^{-2\sqrt{c_1^{-1}(1-\theta^{-1})k}}\theta^{-k}\big(1+\mathcal O\big(1/\sqrt k\big)\big),\\ 
		p_{\geq k}&=Ck^{b/2+1/4}\e^{-2\sqrt{c_1^{-1}(1-\theta^{-1})k}}\theta^{-k}\big(1+\mathcal O\big(1/\sqrt k\big)\big).
		\ea \ee 
	\end{lemma} 
	
	Note that this lemma improves on the bounds in~\eqref{eq:pkbddgumbelrv} by providing a polynomial correction term and a precise multiplicative constant.
	
	\begin{proof} 
		We start by proving the equality for $p_{\geq k}$ and then show the similar result for $p_k$. By~\eqref{eq:gumbex}, we obtain the following expression for $p_{\geq k}$.
		\be \ba\label{eq:intrep}
		p_{\geq k}={}&\int_0^1 x^k (\theta-1+x)^{-k}c_1^{-1}(1-x)^{-(2+b)}\e^{-c_1^{-1}x/(1-x)}\,\d x\\
		&-\int_0^1 x^k (\theta-1+x)^{-k}b(1-x)^{-(1+b)}\e^{-c_1^{-1}x/(1-x)}\,\d x.
		\ea\ee  
		The second integral is of a similar form as the first, with a different constant in front and a different polynomial exponent. We hence only provide an explicit analysis of the first integral. Using a variable transform $u=x/(1-x)$, we find the first integral equals
		\be 
		\theta^{-k}c_1^{-1}\int_0^\infty u^k (1+u)^{b-k}\Big(1-\frac{1}{\theta(1+u)}\Big)^{-k}\e^{-c_1^{-1}u}\,\d u.
		\ee 
		We now define $X_u$ to be a negative binomial random variable with parameters $k$ and $p_u:=(\theta(1+u))^{-1}$, for any $u>0$. As the sum over the probability mass function of $X_u$ is one irrespectively of the value of $u$, we obtain that the above equals
		\be \ba
		\theta^{-k}c_1^{-1}{}&\int_0^\infty \sum_{j=0}^\infty \binom{j+k-1}{j} p_u^j(1-p_u)^{k} u^k (1+u)^{b-k}\Big(1-\frac{1}{\theta(1+u)}\Big)^{-k}\e^{-c_1^{-1}u}\,\d u\\
		&=\theta^{-k}c_1^{-1}\int_0^\infty \sum_{j=0}^\infty \binom{j+k-1}{j}\theta^{-j}u^k (1+u)^{b-(j+k)}\e^{-c_1^{-1}u}\,\d u\\
		&=\theta^{-k}c_1^{-1}\sum_{j=0}^\infty \binom{j+k-1}{j}\theta^{-j}\Gamma(k+1)\, U(k+1,2+b-j,c_1^{-1}),
		\ea \ee 
		where $U(a,b,z)$ is the confluent hypergeometric function of the second kind, defined as 
		\be \label{eq:confhyp}
		U(a,b,z):=\frac{1}{\Gamma(a)}\int_0^\infty x^{a-1}(1+x)^{b-a-1}\e^{-zx}\,\d x,
		\ee 
		whenever $\text{Re}(a)>0$. Using the Kummer transform $U(a,b,z)=z^{1-b}U(1+a-b,2-b,z)$ yields
		\be\label{eq:finexpr}
		\theta^{-k}{}c_1^{-1}\sum_{j=0}^\infty \binom{j+k-1}{j}\theta^{-j}\Gamma(k+1)\,c_1^{b-(j-1)}U(j+k-b,j-b,c_1^{-1}).
		\ee 
		Again using the definition of the confluent hypergeometric function of the second kind, we obtain
		\be\ba 
		c_1^b\theta^{-k}{}&\sum_{j=0}^\infty \frac{\Gamma(j+k)\Gamma(k+1)}{\Gamma(k)\Gamma(j+1)\Gamma(j+k-b)}(c_1\theta)^{-j}\int_0^\infty u^{j+k-b-1}(1+u)^{-(k+1)}\e^{-c_1^{-1}u}\, \d u\\
		={}& c_1^b k\theta^{-k}\frac{\Gamma(k)}{\Gamma(k-b)}\sum_{j=0}^\infty \frac{\Gamma(j+k)\Gamma(k-b)}{\Gamma(j+k-b)\Gamma(k)}\frac{1}{j!}(c_1\theta)^{-j}\int_0^\infty u^{j+k-b-1}(1+u)^{-(k+1)}\e^{-c_1^{-1}u}\, \d u\\
		={}& c_1^b k\theta^{-k}\frac{\Gamma(k)}{\Gamma(k-b)}\sum_{j=0}^\infty \frac{(k)^{(j)}}{(k-b)^{(j)}}\frac{1}{j!}(c_1\theta)^{-j}\int_0^\infty u^{j+k-b-1}(1+u)^{-(k+1)}\e^{-c_1^{-1}u}\, \d u,
		\ea\ee 
		where $(x)^{(j)}:=x(x+1)\cdots (x+(j-1))$, for any $x\in\R$  and $j\in\N_0$. We can then bound
		\be \label{eq:bbounds}
		\frac{(k)^{(j)}}{(k-b)^{(j)}}\geq \begin{cases}
			1 & \mbox{if }  k>b\geq 0,\\
			\Big(\frac{k}{k-b}\Big)^j & \mbox{if }  b<0.
		\end{cases} \  \text{ and }\  
		\frac{(k)^{(j)}}{(k-b)^{(j)}}\leq \begin{cases}
			1 & \mbox{if } b<0,\\
			\Big(\frac{k}{k-b}\Big)^j & \mbox{if }  k>b\geq 0.
		\end{cases}
		\ee  
		As the bounds are symmetric, we can assume that $b\geq 0$ without loss of generality; the other case follows similarly. We deal with the lower bound first. This yields
		\be \ba \label{eq:uexp}
		c_1^b{}&k \theta^{-k}\frac{\Gamma(k)}{\Gamma(k-b)}\sum_{j=0}^\infty \frac{1}{j!}(c_1\theta)^{-j}\int_0^\infty u^{j+k-b-1}(1+u)^{-(k+1)}\e^{-c_1^{-1}u}\, \d u\\
		&=c_1^bk \theta^{-k}\frac{\Gamma(k)}{\Gamma(k-b)}\int_0^\infty u^{k-b-1}(1+u)^{-(k+1)}\e^{-c_1^{-1}(1-\theta^{-1})u}\, \d u\\
		&=c_1^bk \theta^{-k}\frac{\Gamma(k)}{\Gamma(k-b)}\Gamma(k-b)U(k-b,-b,c_1^{-1}(1-\theta^{-1})).
		\ea \ee 
		It follows from~\cite[$(3.12)$ and $(3.15)$]{Tem13} that, when $a>d/2$ is large and $d,z$ are bounded,
		\be  
		\Gamma(a)U(a,d,z^2)=2\e^{z^2/2}\Big(\frac{2z}{u}\Big)^{1-d}K_{1-d}(uz)\big(1+\mathcal O(1/u)\big), 
		\ee 
		where $u=2\sqrt{a-d/2}$ and $K_{1-d}(uz)$ is a modified Bessel function. Combined with the asymptotic expression for the modified Bessel function as in~\cite[$(10.40.2)$]{Loz03}, we obtain 
		\be \label{eq:uasymp}
		\Gamma(a)U(a,d,z^2)=2\sqrt{\frac{\pi}{2uz}}\e^{z^2/2-uz}\Big(\frac{2z}{u}\Big)^{1-d}\big(1+\mathcal O(1/u)\big).
		\ee  
		In this particular case, it yields
		\be \ba
		\Gamma(k-b)&U(k-b,-b,c_1^{-1}(1-\theta^{-1}))\\
		&=\e^{c_1^{-1}(1-\theta^{-1})/2}\sqrt{\pi}(c_1^{-1}(1-\theta^{-1}))^{1/4+b/2}k^{-b/2-3/4}\e^{-2\sqrt{c_1^{-1}(1-\theta^{-1})k}}\big(1+\mathcal O\big(1/\sqrt k\big)\big).
		\ea \ee 
		Using this, as well as $\Gamma(k)/\Gamma(k-b)=k^b(1+\mathcal O(1/k))$, in~\eqref{eq:uexp}, we arrive at
		\be \label{eq:finuasymp}
		\e^{c_1^{-1}(1-\theta^{-1})/2}\sqrt{\pi}c_1^{-1/4+b/2}((1-\theta^{-1})k)^{1/4+b/2}\e^{-2\sqrt{c_1^{-1}(1-\theta^{-1})k}}\theta^{-k}\big(1+\mathcal O\big(1/\sqrt k\big)\big).
		\ee 
		We then tend to the upper bound in~\eqref{eq:bbounds} for $b\geq 0$, which yields 
		\be \ba
		c_1^b{}&k \theta^{-k}\frac{\Gamma(k)}{\Gamma(k-b)}\sum_{j=0}^\infty \frac{1}{j!}\Big(\frac{(c_1\theta)^{-1} k}{k-b}\Big)^{j}\int_0^\infty u^{j+k-b-1}(1+u)^{-(k+1)}\e^{-c_1^{-1}u}\, \d u\\
		&=c_1^bk \theta^{-k}\frac{\Gamma(k)}{\Gamma(k-b)}\int_0^\infty u^{k-b-1}(1+u)^{-(k+1)}\e^{-(c_1^{-1}(1-\theta^{-1})-(c_1\theta)^{-1}b/(k-b))u}\, \d u\\
		&=c_1^bk \theta^{-k}\frac{\Gamma(k)}{\Gamma(k-b)}\Gamma(k-b)U(k-b,-b,c_1^{-1}(1-\theta^{-1})-(c_1\theta)^{-1}b/(k-b)).
		\ea \ee 
		From the asymptotic result in~\eqref{eq:uasymp} we find that 
		\be 
		U\Big(k-b,-b,\frac{1}{c_1}\big(1-\theta^{-1}\big)-\frac{1}{c_1\theta}\frac{b}{k-b}\Big)=U\Big(k-b,-b,\frac{1}{c_1}\big(1-\theta^{-1}\big)\Big)\big(1+\mathcal O(1/\sqrt{k})\big),
		\ee 
		so that the lower and upper bound match up to error terms. By~\eqref{eq:uasymp}, we thus arrive at 
		\be\ba \label{eq:finuasymp2}
		\int_0^1{}& x^k (\theta-1+x)^{-k}c_1^{-1}(1-x)^{-(2+b)}\e^{-c_1^{-1}x/(1-x)}\,\d x\\
		&=\e^{c_1^{-1}(1-\theta^{-1})/2}\sqrt{\pi}c_1^{-1/4+b/2}((1-\theta^{-1})k)^{1/4+b/2}\e^{-2\sqrt{c_1^{-1}(1-\theta^{-1})k}}\theta^{-k}\big(1+\mathcal O\big(1/\sqrt k\big)\big)\\
		&=C k^{1/4+b/2}\e^{-2\sqrt{c_1^{-1}(1-\theta^{-1})k}}\theta^{-k}\big(1+\mathcal O\big(1/\sqrt k\big)\big).
		\ea \ee  
		Finally, when considering the second integral in~\eqref{eq:intrep}, we observe its integrand is similar to that of the first integral but with a different constant in front and with a constant $\wt b=b-1$ in the polynomial exponent. We can thus follow the exact same steps as for the first integral in~\eqref{eq:intrep} to conclude that it can be included in the $\mathcal O(1/\sqrt k)$ term in~\eqref{eq:finuasymp2}. We thus obtain that 
		\be 
		p_{\geq k}=Ck^{1/4+b/2}\e^{-2\sqrt{c_1^{-1}(1-\theta^{-1})k}}\theta^{-k}\big(1+\mathcal O\big(1/\sqrt k\big)\big),
		\ee 
		as required.
		
		We now show the result for $p_k$, which uses the above steps with several minor adjustments. First, 
		\be \ba \label{eq:pkintrep}
		p_k={}&(\theta -1)\int_0^1 x^k (\theta-1+x)^{-(k+1)}c_1^{-1}(1-x)^{-(2+b)}\e^{-c_1^{-1}x/(1-x)}\,\d x\\
		&-(\theta-1)\int_0^1  x^k (\theta-1+x)^{-(k+1)}b(1-x)^{-(b+1)}\e^{-c_1^{-1}x/(1-x)}\,\d x.
		\ea \ee
		As for the proof of the asymptotic expression of $p_{\geq k}$, we consider the first integral only as the second one is of lower order. This yields
		\be\ba
		(\theta -1){}&\int_0^1 x^k (\theta-1+x)^{-(k+1)}c_1^{-1}(1-x)^{-(2+b)}\e^{c_1^{-1}x/(1-x)}\,\d x\\
		={}&(1-\theta^{-1})c_1^{-1}\theta^{-k}\int_0^\infty u^k (1+u)^{b-k}\Big(1-\frac{1}{\theta(1+u)}\Big)^{-(k+1)}\e^{-c_1^{-1}u}\,\d u\\
		={}&(1-\theta^{-1})c_1^{-1}\theta^{-k}\sum_{j=0}^\infty \binom{j+k}{j}\theta^{-j}\int_0^\infty u^k (1+u)^{b-(j+k)}\e^{-c_1^{-1}u}\,\d u\\
		={}&(1-\theta^{-1})c_1^{-1}\theta^{-k}\sum_{j=0}^\infty \binom{j+k}{j}\theta^{-j}\Gamma(k+1)U(k+1,2+b-j,c_1^{-1})\\
		={}&(1-\theta^{-1})c_1^b\theta^{-k}\sum_{j=0}^\infty \binom{j+k}{j}(c_1\theta)^{-j}\Gamma(k+1)U(k+j-b,j-b,c_1^{-1})\\
		={}&(1-\theta^{-1})c_1^b\theta^{-k}\frac{\Gamma(k+1)}{\Gamma(k-b)}\sum_{j=0}^\infty\frac{(k+1)^{(j)} }{(k-b)^{(j)}}\frac{(c_1\theta)^{-j}}{j!}\int_0^\infty \!\!u^{k+j-b-1}(1+u)^{-(k+1)}\e^{-c_1^{-1}u}\,\d u.
		\ea\ee 
		Similar to~\eqref{eq:bbounds}, we bound
		\be \label{eq:bbounds2}
		\frac{(k-1)^{(j)}}{(k-b)^{(j)}}\geq \begin{cases}
			1 & \mbox{if }  k>b\geq -1,\\
			\Big(\frac{k+1}{k-b}\Big)^j & \mbox{if }  b<-1.
		\end{cases} \quad \text{ and }\quad 
		\frac{(k)^{(j)}}{(k-b)^{(j)}}\leq \begin{cases}
			1 & \mbox{if } b<-1,\\
			\Big(\frac{k+1}{k-b}\Big)^j & \mbox{if }  k>b\geq -1.
		\end{cases}
		\ee  
		Again, we assume without loss of generality that $b\geq -1$. Moreover, we only concern ourselves with the lower bound on $(k+1)^{(j)}/(k-b)^{(j)}$ when $b\geq -1$, since we obtain a matching upper bound with the required error term when using the upper bound on $(k+1)^{(j)}/(k-b)^{(j)}$ when $b\geq -1$, as in the proof for $p_{\geq k}$. Thus, we obtain the lower bound 
		\be \ba 
		(1-\theta^{-1}){}&c_1^b\theta^{-k}\frac{\Gamma(k+1)}{\Gamma(k-b)}\sum_{j=0}^\infty\frac{(c_1\theta)^{-j}}{j!}\int_0^\infty u^{k+j-b-1}(1+u)^{-(k+1)}\e^{-c_1^{-1}u}\,\d u\\
		&=(1-\theta^{-1})c_1^bk\theta^{-k}\frac{\Gamma(k)}{\Gamma(k-b)}\int_0^\infty u^{k+j-b-1}(1+u)^{-(k+1)}\e^{-c_1^{-1}(1-\theta^{-1})u}\,\d u,
		\ea \ee 
		which, up to the constant $(1-\theta^{-1})$, is the exact same expression as in~\eqref{eq:uexp}. As discussed above, using the upper bound on $(k+1)^{(j)}/(k-b)^{(j)}$ yields a matching upper bound (up to error terms), from which the result follows.
	\end{proof}
	
	Recall that in this example we set 
	\be\ba
	X^{(n)}_i:={}&\big|\big\{\jnn: \znj=\big\lfloor \log_\theta n-C_{\theta,1,c_1}\sqrt{\log_\theta n}+(b/2+1/4) \log_\theta \log_\theta n \big\rfloor +i\big\}\big|,\\ 
	X^{(n)}_{\geq i}:={}&\big|\big\{\jnn: \znj\geq \big\lfloor \log_\theta n-C_{\theta,1,c_1}\sqrt{\log_\theta n}+(b/2+1/4) \log_\theta \log_\theta n \big\rfloor +i\big\}\big|,\\
	\eps_n:={}&\big(\log_\theta n-C_{\theta,1,c_1}\sqrt{\log_\theta n}+(b/2+1/4) \log_\theta \log_\theta n\big )\\
	&-\big\lfloor\log_\theta n-C_{\theta,1,c_1}\sqrt{\log_\theta n}+(b/2+1/4) \log_\theta \log_\theta n \big \rfloor.
	\ea\ee
	We then state the analogue of Proposition~\ref{prop:factmean}.
	
	\begin{proposition}\label{prop:gumbfactmean}
		Consider the WRT model as in Definition~\ref{def:WRT} with vertex-weights $(W_i)_{\inn}$ whose distribution satisfies~\eqref{eq:gumbex} for some $b\in\R,c_1>0$ such that $bc_1\leq 1$, and recall $c_{c_1,b,\theta}$ from~\eqref{eq:c}. For a fixed $K\in\N,c\in(1,\theta/(\theta-1))$ the following holds. For any integer-valued $i_n,i_n'$ such that $0<\log_\theta n+i_n<\log_\theta n+i_n'<c\log_\theta n$ and $i_n,i_n'=\delta\sqrt{\log_\theta n}+o(\sqrt{\log n})$ for some $\delta\in\R$ and for $a_{i_n},\ldots,a_{i_n'}\in\N_0$ satisfying $a_{i_n}+\ldots+a_{i_n'}=K$,
		\be\ba
		\mathbb E\bigg[\!\Big(X^{(n_\ell)}_{\geq i_n'}\Big)_{a_{i_n'}}\!\prod_{k=i_n}^{i_n'-1}\!\!\Big(X_k^{(n_\ell)}\Big)_{a_k}\bigg]
		={}&\big(c_{c_1,b,\theta}\theta^{-i_n'+\eps_n-C_{\theta,1,c_1}\delta/2}\big)^{a_{i_n'}}\!\!\prod_{k=i_n}^{i_n'-1}\!\!\big(c_{c_1,b,\theta}(1-\theta^{-1})\theta^{-k+\eps_n-C_{\theta,1,c_1}\delta/2}\big)^{a_k}\\
		&\times \Big(1+\mathcal O\Big(\frac{\log_\theta\log_\theta n}{\sqrt{\log_\theta n}}\vee \frac{|i_n-\sqrt{\log_\theta n}|\vee |i_n'-\sqrt{\log_\theta n}|}{\sqrt{\log_\theta n}}\Big)\Big).
		\ea\ee
	\end{proposition} 
	
	\begin{proof}
		Set $K':=K-a_{i_n'}$ and for each $i_n\leq k\leq i_n'$ and each $u$ such that $\sum_{\ell=i_n}^{k-1}a_\ell<u\leq \sum_{\ell=i_n}^k a_\ell$, let $m_u=\big\lfloor\log_\theta n-C_{\theta,1,c_1}\sqrt{\log_\theta n}+(b/2+1/4) \log_\theta \log_\theta n \big \rfloor +k$. Also, let $(v_u)_{u\in[K]}$ be $K$ vertices selected uniformly at random without replacement from $[n]$. Then, as the $X^{(n)}_{\geq k}$ and $X^{(n)}_{k}$ can be expressed as sums of indicators, using the same steps as in the proof of Proposition~\ref{prop:factmean},
		\be \label{eq:gumbmeanexunif}
		\E{\Big(X^{(n)}_{\geq i_n'}\Big)_{a_{i_n'}}\prod_{k=i}^{i_n'-1}\Big(X_k^{(n)}\Big)_{a_k}}=(n)_K\sum_{\ell=0}^{K'}\sum_{\substack{S\subseteq [K']\\ |S|=\ell}}(-1)^\ell \P{\Zm_n(v_u)\geq m_u+\ind_{\{u\in S\}}\text{ for all } u\in [K]}.
		\ee
		By Proposition~\ref{lemma:degprobasymp},
		\be \label{eq:gumbtailprobunif}
		\P{\Zm_n(v_\ell)\geq m_u+\ind_{\{u\in S\}}\text{ for all } u\in [K]}=\prod_{u=1}^K \E{\Big(\frac{W}{\E{W}+W}\Big)^{m_u+\ind_{\{u\in S\}}}}(1+o(n^{-\beta})),
		\ee 
		for some $\beta>0$. By Lemma~\ref{lemma:gumbpkasymp} (and recalling the constant $C$ in~\eqref{eq:c}), when $|S|=\ell$,
		\be\ba
		\prod_{u=1}^K {}&\E{\Big(\frac{W}{\E{W}+W}\Big)^{m_u+\ind_{\{u\in S\}}}}\\
		={}&(C)^K \theta^{-\sum_{u=1}^K m_u-\ell}\exp\Bigg(\!-2\sum_{u=1}^K\sqrt{\frac{1-\theta^{-1}}{c_1}\big(m_u+\ind_{\{u\in S\}}\big)}\Bigg)\prod_{u=1}^K  (m_u+\ind_{\{u\in S\}})^{b/2+1/4}\\
		&\times (1+\mathcal O (1/\sqrt{\log n})).
		\ea\ee 
		Here, we are able to obtain the error term $1+\mathcal O(1/\sqrt{\log n})$ due to the fact that $\log_\theta n+i_n>\eta  \log n$ for some $\eta>0$ when $n$ is large. We note that $C_{\theta,1,c_1}\log\theta=2\sqrt{c_1^{-1}(1-\theta^{-1})}$. As $i_n,i_n'\sim \delta\sqrt{\log_\theta n}$,
		\be
		\prod_{u=1}^K(m_u+\ind_{\{u\in S\}})^{b/2+1/4 }=(\log_\theta n)^{ K(b/2+1/4)}\big(1+\mathcal O\big(1/\sqrt{\log_\theta n}\big)\big),
		\ee
		uniformly in $S$ (and $\ell$). Moreover, again uniform in $S$ and $\ell$,
		\be \ba
		\exp{}&\bigg(-C_{\theta,1,c_1}\log\theta\sum_{u=1}^K\sqrt{m_u+\ind_{\{u\in S\}}}\bigg)\\
		={}&\exp\Big(-\Big( C_{\theta,1,c_1}\log\theta\sqrt{\log_\theta n}-\frac{C_{\theta,1,c_1}-\delta}{2}\Big)\Big)^K\\
		&\times \Big(1+\mathcal O\Big(\frac{\log_\theta\log_\theta n}{\sqrt{\log_\theta n}}\vee \frac{|i_n-\sqrt{\log_\theta n}|\vee |i_n'-\sqrt{\log_\theta n}|}{\sqrt{\log_\theta n}}\Big)\Big).
		\ea \ee 
		This last step follows from the fact that the first-order term of $m_u$ is $\log_\theta n$ and its second-order term is $-(C_{\theta,1,c_1}-\delta)\sqrt{\log_\theta n}$. Finally, its lower-order terms are $\log_\theta\log_\theta n+(|i-\sqrt{\log_\theta n}|\vee |i_n'-\sqrt{\log_\theta n}|)$. Then, using a Taylor expansion for the square root yields the result. Combining all of the above and recalling that $c_{c_1,b,\theta}=C\theta^{C_{\theta,1,c_1}^2/2}$, we thus arrive at
		\be \ba
		(n)_K{}&\sum_{\ell=0}^{K'}\sum_{\substack{S\subseteq [K']\\ |S|=\ell}}(-1)^\ell\Big(C (\log_\theta n)^{b/2+1/4}\exp\Big(-\Big(C_{\theta,1,c_1}\log\theta\Big(\sqrt{\log_\theta n}-\frac{C_{\theta,1,c_1}-\delta}{2}\Big)\Big)\Big)\Big)^K \\
		&\times  \theta^{-\sum_{u=1}^K m_u-\ell}\Big(1+\mathcal O\Big(\frac{\log_\theta\log_\theta n}{\sqrt{\log_\theta n}}\vee \frac{|i_n-\sqrt{\log_\theta n}|\vee |i_n'-\sqrt{\log_\theta n}|}{\sqrt{\log_\theta n}}\Big)\Big)\\
		={}&\big(c_{c_1,b,\theta}\theta^{-i_n'+\eps_n-C_{\theta,1,c_1}\delta/2}\big)^{a_{i_n'}}\prod_{k=i_n}^{i_n'-1}\big(c_{c_1,b,\theta}(1-\theta^{-1})\theta^{-k+\eps_n-C_{\theta,1,c_1}\delta/2}\big)^{a_k}\\
		&\times \Big(1+\mathcal O\Big(\frac{\log_\theta\log_\theta n}{\sqrt{\log_\theta n}}\vee \frac{|i_n-\sqrt{\log_\theta n}|\vee |i_n'-\sqrt{\log_\theta n}|}{\sqrt{\log_\theta n}}\Big)\Big),
		\ea\ee 
		where the last step follows from a similar argument as in the proof of Proposition~\ref{prop:factmean}.
	\end{proof}
	
	With Proposition~\ref{prop:gumbfactmean} at hand, the proofs of Theorems~\ref{thrm:gumbppp},~\ref{thrm:gumbmaxtail}, and~\ref{thrm:gumbasympnorm} follow the same approach as the proofs of Theorems~\ref{thrm:mainatom},~\ref{thrm:maxtail}, and~\ref{thrm:asympnormal}.

	\textbf{Acknowledgements}\\
	The authors would like to sincerely thank the referees for their feedback, which helped to significantly improve the presentation of the article.
	
	BL has been funded by an URSA whilst at the University of Bath and is currently funded by the grant GrHyDy ANR-20-CE40-0002. LE was partially supported by the grant PAPIIT IN102822.
	
	\bibliographystyle{abbrv}
	\bibliography{wrtbib}

	\appendix
	\section{}
	\label{sec:appendix}
	
	In this appendix we collect various estimates on the sum of i.i.d.\ random variables, including a quantitative version of the law of large numbers,
		see Lemmas~\ref{lemma:sumbound} and~\ref{lemma:weightsumbounds} and Corolary~\ref{cor:uil}. Finally, we also include the
		details of the calculations of certain iterated integrals in Lemma~\ref{lemma:logints} and the proof of some of the properties of the sequences defined in~\eqref{eq:ek} in Lemma~\ref{lemma:rk}.

		\begin{lemma}\label{lemma:sumbound}
			Let $(W_i)_{i\in\N}$ be i.i.d.\ copies of a random variable $W$ that satisfies conditions~\ref{ass:weightsup} and~\ref{ass:weightzero} of Assumption~\ref{ass:weights}. Then, there exists $C' >0$ such that for any $n\in\N$ and any $x>0$, 
			\be 
			\P{S_N\leq x}\leq \big(C'\Gamma(\rho)x^\rho\big)^N \Gamma(\rho N+1)^{-1}.
			\ee 
		\end{lemma}
		
		\begin{proof}
				We prove the bound by induction. We first couple the random variable $W$ and $(W_i)_{i\in\N}$ to i.i.d.\ random variables $X$ and $(X_i)_{i\in\N}$, respectively. $W$ satisfies condition~\ref{ass:weightzero} for some $C,x_0,\rho>0$, and we can without loss of generality assume that $Cx_0^\rho=1$. Indeed, if $Cx_0^\rho<1$ holds, then we can increase $C$, and if $Cx_0^\rho>1$ then we can decrease $x_0$. Hence, we let $X$ be a random variable with probability density function $f_X:[0,1]\to [0,1]$
				\be 
				f_X(x):=\begin{cases} C\rho x^{\rho-1} &\mbox{if } x\in[0,x_0],\\
					0 &\mbox{if } x\in(x_0,1].
				\end{cases} 
				\ee 
				It is clear, by the assumption $Cx_0^\rho=1$, that $f_X$ indeed is a probability density function, and that $\P{W\leq x}\leq \P{X\leq x}$ for all $x\in [0,1]$ by Condition~\ref{ass:weightzero}. If we thus let $(X_i)_{i\in\N}$ be i.i.d.\ copies of $X$, then $\P{S_N\leq x}\leq \P{S'_N\leq x} $ for any $N\in\N$ and $x\in\R$, where $S'_N:=\sum_{i=1}^N X_i$. It thus suffices to bound $\P{S_N'\leq x}$ from above. 
				
				For $N=1$ it directly follows that $\P{X_1\leq x}\leq Cx^\rho$ for any $x>0$, so that the case $N=1$ is satisfied by setting $C':=C\rho=C\Gamma(\rho+1)/\Gamma(\rho)$. Let us then assume that the bound 
				\be 
				\P{S_N'\leq x}\leq (C'\Gamma(\rho)x^\rho)^N\Gamma(\rho N+1)^{-1}
				\ee 
				holds for any $x>0$ and some $N\in\N$. Then, by the induction hypothesis and the definition of $f_X$,
				\be 
				\P{S_{N+1}'\leq x}=\int_0^{x\wedge 1}\P{S_N'\leq x-t}f_X(t)\,\d t\leq \int_0^{x\wedge 1}(C'\Gamma(\rho)(x-t)^\rho)^N\Gamma(\rho N+1)^{-1}C\rho t^{\rho-1}\,\d t.
				\ee 
				By extending the upper range of the integral to $x$, using the substitution $s=t/x$, and recalling that $C\rho=C'$, we arrive at the upper bound
				\be 
				\frac{C'^{N+1}\Gamma(\rho)^Nx^{(N+1)\rho}}{\Gamma(\rho N+1)}\int_0^1 (1-s)^{\rho N}s^{\rho-1}\,\d s=\frac{C'^{ N+1} \Gamma(\rho)^Nx^{(N+1)\rho}}{\Gamma(\rho N+1)}B(\rho,\rho N+1)=\frac{(C'\Gamma(\rho)x^\rho)^{N+1}}{\Gamma(\rho(N+1)+1)},
				\ee 
				where we use in the final step that, for $x,y>0$, $B(x,y):=\Gamma(x)\Gamma(y)/\Gamma(x+y)$ is the Beta function. This yields the desired result and concludes the proof.
		\end{proof}
		
		\begin{lemma}[Bounds on partial sums of vertex-weights]\label{lemma:weightsumbounds}
			Let $(W_i)_{i\in\N}$ be i.i.d.\ copies of a random variable $W$ with finite mean. Let $\eps\in(0,1), \delta\in(0,1/2)$, $k\in\N$, $C>k\log(\theta)/(2\theta(\theta-1))$, and $c,\alpha>0$, and set $\zeta_n:=n^{-\delta\eps}/\E{W}$ and $\zeta_n':=(C\log n)^{-\delta/(1-2\delta)}/\E W$. Consider the events 
			\be\ba  \label{eq:events}
			E_n^{(1)}&:=\bigg\{ \sum_{\ell=1}^j W_\ell \in ((1-\zeta_n)\E{W}j,(1+\zeta_{n})\E{W}j),\text{ for all } n^{\eps}\leq j\leq n\bigg\},\\
			E_n^{(2)}&:=\Big\{\sum_{\ell=k+1 }^j W_\ell\in((1-\zeta_n)\E{W}j,(1+\zeta_n)\E{W}j),\text{ for all } n^{\eps}\leq j\leq n\Big\},\\
			E_n^{(3)}&:=\bigg\{\sum_{\ell=1}^j W_\ell \geq (1-\zeta'_n)\E{W} j,\text{ for all } (C\log n)^{1/(1-2\delta)}\leq j\leq n\bigg\},\\
			E_n^{(4)}&:=\{S_{\lceil \alpha \log n\rceil}\geq c\log n\}.
			\ea \ee 
			Then, for any $\gamma>0$ and any $i\in[4]$ (where, for $i=3,4$, we choose $C$ and $\alpha$ sufficiently large depending on $\gamma$, respectively), for all $n$ large,
			\be 
			\P{(E_n^{(i)})^c}\leq n^{-\gamma}.
			\ee 
			Finally, additionally assume that $W$ satisfies conditions~\ref{ass:weightsup} and~\ref{ass:weightzero} of Assumption~\ref{ass:weights}, set $f(n):=\lceil \log n/\log\log\log n\rceil, g(n):=\lceil \log\log n\rceil$, and define $E_n^{(5)}:=\{S_{f(n)}\geq g(n)\}$. Then, for any $\gamma>0$ and all $n$ large, 
			\be 
			\P{(E_n^{(5)})^c}\leq n^{-\gamma}.
			\ee 
		\end{lemma} 
		
		\begin{proof}
			We prove the bounds for the complement of the events one by one. By noting that $\wt S_j:=\sum_{\ell=1}^j W_\ell-j\E{W}$ is a martingale, that $|\wt S_j-\wt S_{j-1}|\leq 1+\E{W}=\theta$ and that $\zeta_n\geq j^{-\delta}/\E W$ for $j\geq n^{\eps}$, we can use the Azuma-Hoeffding inequality to obtain 
			\be \label{eq:proben}
			\P{(E_n^{(1)})^c}\leq \sum_{j\geq n^{\eps}}\P{\big|\wt S_j\big|\geq \zeta_n j\E W}\leq 2\sum_{j\geq n^{\eps}}\exp\Big(-\frac{j^{1-2\delta}}{2\theta^2}\Big).
			\ee 
			Writing $c_\theta:=1/(2\theta^2)$, we further bound the sum from above by 
			\be \label{eq:ecnbound}
			2\int_{\lfloor n^{\eps}\rfloor}^\infty \exp\big(-c_\theta x^{1-2\delta}\big)\,\d x=2\frac{c_{\theta}^{-1/(1-2\delta)}}{1-2\delta}\Gamma\Big(\frac{1}{1-2\delta},c_\theta \lfloor n^{\eps}\rfloor ^{1-2\delta}\Big),
			\ee 	
			where $\Gamma(a,x)$ is the incomplete Gamma function. As, for $a>0$ fixed, $\Gamma(a,x)= (1+o(1))x^{a-1}\e^{-x}$, this yields the desired bound. The proof for the event $E^{(2)}_n$ follows in an analogous way. 
			
			For the event $E^{(3)}_n$, we again use the Azuma-Hoeffding inequality to obtain
			that there exist constants $C_{\theta,\delta}, \tilde C_{\theta,\delta} > 0$ such that
			\be \ba
			\P{(E^{(3)}_n)^c}&\leq C_{\theta,\delta}\Big(c_\theta C\log n\Big)^{2\delta/(1-2\delta)}\exp\Big(-c_\theta \lfloor (C\log n)^{1/(1-2\delta)}\rfloor^{1-2\delta}\Big)(1+o(1))\\
			&=\wt C_{\theta,\delta}(\log n)^{2\delta/(1-2\delta)}\exp(-c_\theta C\log n)(1+o(1))\\
			&=n^{-c_\theta C(1+o(1))}.
			\ea \ee 
			Choosing $C$ sufficiently large then yields the desired result. A similar application of the same inequality for $E^{(4)}_n$ yields
			\be 
			\P{(E^{(4)}_n)^c}\leq \P{|\wt S_{\lceil \alpha \log n\rceil}|\geq |c-\alpha \E W|\log n}\leq 2\exp\Big(-\frac{(\alpha \E W-c)^2 \log n}{2\alpha \theta^2}\Big).
			\ee 
			Again, choosing $\alpha$ sufficiently large yields the desired result. Finally, we use Lemma~\ref{lemma:sumbound} to obtain 
			\be \ba
			\P{(E_n^{(5)})^c}&\leq \big(C'\Gamma(\rho)g(n)^\rho\big)^{f(n)}\Gamma(\rho f(n)+1)^{-1}\\
			&=\frac{1+o(1)}{\sqrt{2\pi \rho f(n)}}\exp(\rho f(n)(1+\log g(n))+f(n)\log(C\Gamma(\rho))-\rho f(n)\log(\rho f(n)))\\
			&\leq \exp\big(-\tfrac{\rho}{2}f(n)\log f(n)\big),
			\ea \ee 
			where the final inequality holds for all $n$ sufficiently large. By the choice of $f(n)$, this yields the desired result and concludes the proof.		 
		\end{proof}
		
		\begin{corollary}\label{cor:uil}
			Let $(W_i)_{i\in\N}$ be i.i.d.\ copies of a random variable $W$ that satisfies conditions~\ref{ass:weightsup} and~\ref{ass:weightzero} of Assumption~\ref{ass:weights} and define, for $i\in\N$, 
			\be 
			U_{i,n}:=\frac{1}{\log n}\sum_{j=i}^{n-1} \frac{W_i}{S_j}. 
			\ee 
			Then, for any $i\in[n]$ and any $\gamma>0$, there exists a sequence $\eta_n \downarrow 0$ such that
			\be 
			\P{U_{i,n}\leq \frac{1+\eta_n}{\E W}}=1-\mathcal O(n^{-\gamma}).
			\ee 
		\end{corollary} 
		
		\begin{proof}
			For some constants $\alpha>0$ as in Lemma~\ref{lemma:weightsumbounds} and $\delta\in(0,1/4)$ we bound $U_{i,n}$ from above by 
			\be \label{eq:splitsum}
			U_{i,n}\leq \frac{1}{\log n}\bigg[\sum_{j=1}^{\lfloor\alpha \log n\rfloor }1\wedge \frac{1}{S_j}+\sum_{j=\lceil \alpha \log n\rceil}^{\lfloor (C\log n)^{1/(1-2\delta)}\rfloor}\frac{1}{S_j}+\sum_{j=\lceil (C\log n)^{1/(1-2\delta)}\rceil}^{n-1}\frac{1}{S_j}\bigg],
			\ee  
			where in the first sum on the right-hand side we use that $W_i\leq 1\wedge S_j$ for any $j\geq i$. We now consider each of the sums one at a time. For the final sum, we use the event $E_n^{(3)}$ from Lemma~\ref{lemma:weightsumbounds} to obtain that on this event,
			\be \label{eq:uil1}
			\frac{1}{\log n}\sum_{j=\lceil (C\log n)^{1/(1-2\delta)}\rceil}^{n-1}\frac{1}{S_j} \leq \frac{1}{\E W} \frac{1}{1-\zeta_n'} =\frac{1+o(1)}{\E W}.
			\ee 
			For the middle sum in~\eqref{eq:splitsum}, we can bound $S_j\geq S_{\lceil \alpha \log n\rceil}$ for each index $j$ in the sum. Together with the event $E^{(4)}_n$ from Lemma~\ref{lemma:weightsumbounds}, we obtain the upper bound 
			\be \label{eq:uil2}
			\frac{(C\log n)^{1/(1-2\delta)}}{\log n}\frac{1}{S_{\lceil \alpha \log n\rceil}}\leq \frac{C^{1/(1-2\delta)}}{c} (\log n)^{(4\delta-1)/(1-2\delta)}=o(1),
			\ee 
			by the choice of $\delta$. Finally, we bound the first sum in~\eqref{eq:splitsum}. We set $f(n):=\lceil \log n/\log\log\log n\rceil$ and $g(n):=\lceil \log\log n\rceil$. On the event $E_n^{(5)}$ from Lemma~\ref{lemma:weightsumbounds}, we then have
			\be \label{eq:uil3}
			\frac{1}{\log n}\sum_{j=1}^{\lfloor \alpha \log n\rfloor} 1\wedge \frac{1}{S_j}\leq \frac{f(n)}{\log n}+\frac{1}{\log n}\sum_{j=f(n)+1}^{\lfloor \alpha \log n\rfloor} \frac{1}{S_j}\leq \frac{f(n)}{\log n}+\frac{\alpha }{ S_{f(n)}}\leq \frac{f(n)}{\log n}+\frac{\alpha}{g(n)}=o(1).
			\ee
			Combining~\eqref{eq:uil1}, \eqref{eq:uil2} and~\eqref{eq:uil3} yields  a sequence $(\eta_n)_{n \geq 1}$ with $\eta_n \downarrow 0$ such that
			\be 
			\P{U_{i,n}\leq \frac{1+\eta_n}{\E W}}\leq \P{\big(E_n^{(3)}\cap  E_n^{(4)}\cap E_n^{(5)}\big)^c}=O(n^{-\gamma}), 
			\ee 
			for any $\gamma>0$, by Lemma~\ref{lemma:weightsumbounds}, which concludes the proof.
		\end{proof}
		
		\begin{lemma}\label{lemma:logints}
			For any $k\in\N$ and any $0< a\leq b<\infty$, 
			\be \label{eq:logint}
			\int_a^b \int_{x_1}^b\cdots \int_{x_{k-1}}^b \prod_{j=1}^kx_j^{-1}\,\d x_k\ldots\d x_1= \frac{(\log(b/a))^k}{k!}.
			\ee 
			Similarly, for any $k\in\N$ and any $0< a\leq b-k<\infty$,
			\be \label{eq:loglb}
			\int_{a+1}^b\int_{x_1+1}^b\cdots \int_{x_{k-1}+1}^b\prod_{j=1}^{k}x_{j}^{-1}\,\d x_k\ldots \d x_1\geq \frac{(\log(b/(a+k)))^k}{k!}.
			\ee 
			Moreover, for any $k\in\N$ and any $0<a\leq b<\infty$ and $c>0$,
			\be \label{eq:logint2}
			\int_{a}^b \int_{x_1}^b \cdots \int_{x_{k-1}}^b x_k^{-(1+c)}\prod_{j=1}^{k-1}x_j^{-1} \,\d x_k\ldots \d x_1=c^{-k}a^{-c}\bigg(1-(b/a)^{-c}\sum_{j=0}^{k-1}c^j\frac{(\log(b/a))^j}{j!}\bigg).
			\ee 
		\end{lemma} 
		
		\begin{proof}
			We prove~\eqref{eq:logint} by recursively computing the integrals. For the inner integral, we directly obtain 
			\be \label{eq:log}
			\int_{x_{k-1}}^b x_k^{-1}\,\d x_k=\log(b/x_{k-1}).
			\ee 
			Now, for the $j^{\text{th}}$ innermost integral we use the substitution $y_{k-(j-1)}=\log(b/ x_{k-(j-1)})$, with $2\leq j\leq k$. This yields, for the second innermost integral (i.e.\ $j=2$), by~\eqref{eq:log},
			\be
			\int_{x_{k-2}}^b x_{k-1}^{-1}\log (b/x_{k-1})\,\d x_{k-1}=\int_0^{\log(b/x_{k-2})} y_{k-1}\,\d y_{k-1}=\frac12 (\log(b/x_{k-2}))^2. 
			\ee 
			Reiterating this approach for $j=3,\ldots, k$ yields the desired result.
			
			We use the same approach for the proofs of~\eqref{eq:loglb} and~\eqref{eq:logint2}. For the former, the inner integral equals $\log(b/(x_{k-1}+1))$. The integrand we then obtain can be bounded from below by using $x_{k-1}^{-1}\geq (x_{k-1}+1)^{-1}$. Moreover, we restrict the upper limit of the integral over $x_{k-1}$ from $b$ to $b-1$ (note that this is possible since $a\leq b-1$). By using the substitution $y_{k-1}=\log(b/(x_{k-1}+1))$, this yields the lower bound, for any $x_{k-2} \leq b-2$,
			\be 
			\int_{x_{k-2}+1}^{b-1} (x_{k-1}+1)^{-1}\log(b/(x_{k-1}+1))\,\d x_{k-1}=\int_0^{\log(b/(x_{k-2}+2))} y_{k-1}\,\d y_{k-1}=\frac12 (\log(b/(x_{k-2}+2)))^2.
			\ee 
			By repeating this procedure for the remaining integrals, we obtain the desired result.
			
			To finally prove~\eqref{eq:logint2}, let us define the multiple integrals on the left-hand side as $I_k$. Then, by computing the value of the innermost integral, 
			\be 
			I_k=\frac{I_{k-1}}{c}-\frac{1}{b^cc} \int_a^b\int_{x_1}^b \cdots \int_{x_{k-2}}^b \prod_{j=1}^{k-1}x_j^{-1}\,\ d x_{k-1}\ldots \d x_1=\frac{I_{k-1}}{c}-\frac{(\log(b/a))^{k-1}}{b^cc(k-1)!}, 
			\ee 
			where the last step follows from~\eqref{eq:logint}. Continuing this recursion yields
			\be 
			I_k=\frac{I_1}{c^{k-1}}-\frac{1}{b^c}\sum_{j=1}^{k-1}c^{-j}\frac{(\log(b/a))^{k-j}}{(k-j)!}
			=c^{-k}a^{-c}\bigg(1-(b/a)^{-c}\sum_{j=0}^{k-1}c^j\frac{(\log(b/a))^j}{j!}\bigg),
			\ee 
			which concludes the proof.
		\end{proof} 
	
	\begin{lemma}\label{lemma:rk}
		Consider the sequences $(s_k,r_k)_{k\in\N}$ in~\eqref{eq:ek}. These sequences have the following properties:
		\begin{enumerate}
			\item[$(i)$] $s_k$ is increasing,
			\item[$(ii)$] $r_k$ is decreasing and $\lim_{k\to\infty}r_k=0$.
		\end{enumerate}
	\end{lemma} 
	
	\begin{proof}
		$(i)$ Assume that $s_{k+1}<s_k$ for some $k\in\N$ and take $x\in(s_{k+1},s_k)$. By the definition of $s_{k+1},s_k$ and the choice of $x$,
		\be 
		\P{W\in(x,1)}\leq \e^{-(1-\theta^{-1})(1-x)(k+1)}< \e^{-(1-\theta^{-1})(1-x)k}<\P{W\in(x,1)},
		\ee 
		which leads to a contradiction. 
		
		$(ii)$ 	Assume that $s_k < s_{k+1}$ (otherwise the claim is immediately clear).
		Note that since the function $\P{W \in (x,1)}$ is c\`adl\`ag, we have for any $x < s_k$,
		\be\label{eq:rk} 
		\P{ W \in (s_k,1)} \leq 	r_k \leq \lim_{y \uparrow s_k} \P{W \in (y,1)} \leq \P{W \in (x,1)}. 
		\ee
		Hence, we have that for any $x \in (s_k, s_{k+1})$.
		\[ r_k \geq \P{ W \in (s_k,1)} \geq \P{W \in (x,1)}
		\geq r_{k+1} .\]
		For the second part, since $s_k$ is increasing by $(i)$, we have that $s_k \rightarrow s^* \in (0,1]$. 
		Suppose that $s^* \in (0,1)$. Then, for $k$ sufficiently large, we have 
		$s_k\leq (1+s^*)/2$ and so $r_k \leq \e^{-(1-\theta^{-1})(1-s^*) k/2}$ and so $r_k$ converges to $0$. 
		
		Therefore, we can assume that $s_k \uparrow 1$.
		Let $k_0$ be the smallest $k$ such that $s_k < s_{k+1}$. 
		Such a $k_0$ exists,  otherwise $s^* < 1$ since each $s_k <1$.
		Then, for $k \geq k_0$,
		let $\ell_k$ be the largest integer such  that $s_{\ell_k} < s_k$.
		
		The assumption that $s_k \uparrow 1$ also excludes that $s_{\ell_k}$ is  eventually constant and so 
		$\ell_k \rightarrow \infty$. 
		In particular, we can argue as in~\eqref{eq:rk} to see that 
		\[ r_{k} \leq \P{ W \in (s_{\ell_k},1)}. \]
		Moreover, as $s_{\ell_k} \rightarrow 1$ as $k \rightarrow \infty$, we deduce that $r_k \rightarrow 0$.
	\end{proof}
\end{document}